\documentclass[12pt]{amsart}

\usepackage{amsmath}
\usepackage{amsthm}
\usepackage{amssymb}
\usepackage{amsfonts}
\usepackage{amsrefs}
\usepackage{color}
\usepackage{tikz}

\tikzstyle{shaded}=[fill=red!10!blue!20!gray!30!white]
\tikzstyle{shaded line}=[double=red!10!blue!20!gray!30!white, double distance=1.5mm, draw=black]
\tikzstyle{unshaded}=[fill=white]
\tikzstyle{unshaded line}=[double=white, double distance=1.5mm, draw=black]
\tikzstyle{Tbox}=[circle, draw, thick, fill=white, opaque,]
\tikzstyle{empty box}=[circle, draw, thick, fill=white, opaque, inner sep=2mm]
\tikzstyle{background rectangle}= [fill=red!10!blue!20!gray!40!white,rounded corners=2mm] 
\tikzstyle{on}=[very thick, red!50!blue!50!black]
\tikzstyle{off}=[gray]

\tikzstyle{traces}=[scale=.2, inner sep=1mm]
\tikzstyle{quadratic}=[scale=.4, inner sep=1mm, baseline]
\tikzstyle{annular}=[scale=.7, inner sep=1mm, baseline]
\tikzstyle{make triple edge size}= [scale=.4, inner sep=1mm,baseline] 
\tikzstyle{icosahedron network}=[scale=.3, inner sep=1mm, baseline]
\tikzstyle{ATLsix}=[scale=.25, baseline]
\tikzstyle{TL12}=[scale=.15,baseline]
\tikzstyle{PAdefn}=[scale=.7,baseline]
\tikzstyle{TLEG}=[scale=.5,baseline]

\makeatletter

\newtheorem{lemma}{Lemma}[section]
\newtheorem{definition}[lemma]{Definition}
\newtheorem{theorem}[lemma]{Theorem}
\newtheorem{proposition}[lemma]{Proposition}
\newtheorem{remark}[lemma]{Remark}
\newtheorem{corollary}[lemma]{Corollary}
\newtheorem{conjecture}[lemma]{Conjecture}
\newtheorem{examples}[lemma]{Examples}
\newtheorem{problem}[lemma]{Problem}
\newtheorem{question}[lemma]{Question}

\newtheorem{statement}[lemma]{Statement}

\makeatother

 \title[Ore's theorem for cyclic subfactor planar algebras]{Ore's theorem for cyclic subfactor planar algebras and applications}
  \author[Sebastien Palcoux]{Sebastien Palcoux}
\address{Institute of Mathematical Sciences, Chennai, India}
\email{palcoux@imsc.res.in}

\keywords{von Neumann algebras; intermediate subfactors; distributive lattice; subfactor planar algebras}

\begin{document}

\begin{abstract}
This paper introduces the cyclic subfactors, generalizing the cyclic groups as the subfactors generalize the groups, and generalizing the natural numbers as the maximal subfactors generalize the prime numbers. On one hand, a theorem of O. Ore states that a finite group is cyclic if and only if its subgroups lattice is distributive, and on the other hand, every subgroup of a cyclic group is normal. Then, a subfactor planar algebra is called cyclic if all the biprojections are normal and form a distributive lattice. The main result shows in what sense a cyclic subfactor is singly generated, by generalizing one side of Ore's theorem as follows: if a subfactor planar algebra is cyclic then it is weakly cyclic (or w-cyclic), i.e. there is a minimal $2$-box projection generating the identity biprojection. Some extensions of this result are discussed, and some applications of it are given in subfactors, quantum groups and finite group theories, for example, a non-trivial upper bound for the minimal number of irreducible complex representations generating the left regular representation.
\end{abstract}

\maketitle

\section{Introduction} 

This paper gives a first result emerging from the nascent theory of cyclic subfactors; we first narrate how this theory was born. Vaughan Jones proved in \cite{jo2} that the set of possible index $[M:N]$ for a subfactor $(N \subseteq M)$ is exactly
$$\{ 4 cos^2(\frac{\pi}{n}) \  \vert \  n \geq 3  \} \sqcup [4,\infty]$$ We observe that it's the disjoint union of a discrete series and a continuous series.
Moreover, for a given intermediate subfactor $N \subseteq P \subseteq M$, $[M:N] = [M:P] \cdot  [P:N] $, so by applying a kind of Eratosthenes sieve, we get that a subfactor of index in the discrete series or in $(4,8)$ except the countable set of numbers which are product of numbers in the discrete series, can't have a non-trivial  intermediate subfactor. A subfactor without non-trivial intermediate subfactor is called maximal \cite{bi}. So for example, any subfactor of index in $(4,3+\sqrt{5})$ is maximal; (except $A_{\infty}$) there are at least $19$ irreducible subfactor planar algebras for this interval (see \cite{jms}), the first example is the Haagerup subfactor \cite{ep}. Thanks to the Galois correspondence \cite{nk}, a finite group subfactor, $(R^G \subseteq R)$ or $(R \subseteq R \rtimes G)$, is maximal if and only if it's a prime order cyclic group subfactor (i.e. $G = \mathbb{Z}/p$ with $p$ a prime number). We can see the class of maximal subfactors as a quantum generalization of the prime numbers. Now the natural informal question is:
\begin{question}
What's the quantum generalization of the natural numbers (as the class of maximal subfactors is for the prime numbers)?
\end{question}
For answering this question, we need to find a natural class of subfactors, called the ``cyclic subfactors'', checking: 
\begin{itemize}
\item[(1)] Every maximal subfactor is cyclic.
\item[(2)] A finite group subfactor $(R^G \subseteq R)$ or $(R \subseteq R \rtimes G)$ is cyclic if and only if the group $G$ is cyclic.
\end{itemize} 

An old and little known theorem published in 1938 by the Norwegian mathematician Oystein Ore states that: 
\begin{theorem}[\cite{or}]
A finite group $G$ is cyclic if and only if its subgroups lattice $\mathcal{L}(G)$ is distributive.
\end{theorem}

First, the intermediate subfactors lattice of a maximal subfactor is obviously distributive. Next, by Galois correspondence, the intermediate subfactors lattice of a finite group subfactor is exactly the subgroups lattice (or its reverse) of the group; but the distributivity is reverse invariant, so it permits to get (1) and (2) by Ore's theorem.  

Now an abelian group, so a fortiori a cyclic group, admits only normal subgroups; but T. Teruya has generalized in \cite{teru} the notion of normal subgroup by the notion of normal intermediate subfactor, so: 

\begin{definition}
A finite index irreducible subfactor is cyclic if all its intermediate subfactors are normal and form a distributive lattice.
\end{definition} 

% Every cyclic group subfactor and every maximal subfacto are cyclic, so we get the \textit{natural} definition expected.  \\
 
\begin{remark}
Our motivation comes from our own interpretation of the cyclic subfactors theory as a ``quantum arithmetic''. \end{remark}

% It's important to keep in mind that the cyclic subfactors theory is a kind of ``quantum arithmetic'', and should be central in the subfactors theory, as the following slogan promotes:  \begin{center}  \textit{The prime numbers are for the natural numbers \\ what the maximal subfactors are for the cyclic subfactors, \\ and the cyclic groups are for the groups \\ what the cyclic subfactors are for the subfactors} \end{center} 

Note that an irreducible finite index subfactor $(N \subseteq M)$ admits a finite lattice $\mathcal{L}(N \subseteq M)$ by \cite{wa}, as for the subgroups lattice of a finite group. Moreover, a finite group subfactor remembers the group by \cite{jo}.

There are plenty of examples of cyclic subfactors (see Section \ref{cycl}): of course the cyclic group subfactors and the (irreducible finite index) maximal subfactors, but also (up to equivalence) exactly $23279$ among $34503$ inclusions of groups of index $< 30$, give a cyclic subfactor (more than $65\%$). Moreover, the class of cyclic subfactors is stable by free composition (see Corollary \ref{corofree}), by tensor product ``generically'' (see Remark \ref{coroprod}), and finally, stable by taking the dual or any intermediate. 

Now, the natural problem about the cyclic subfactors is to understand in what sense they are ``singly generated'', and the following theorem, generalizing one side of Ore's theorem, is a first step.   

\begin{theorem}  \label{introthm}
Let $\mathcal{P}$ be a finite index irreducible subfactor planar algebra.
If all its biprojections are normal and form a distributive lattice (i.e. cyclic subfactor) then there is a minimal projection generating the identity biprojection (i.e. w-cyclic subfactor).
\end{theorem}  

It's the main theorem of the paper. The initial statement, strictly in the subfactor framework (see Corollary \ref{mainsub}), has been translated to the planar algebra framework by Zhengwei Liu. 

The converse is not true, counter-examples come from the result that a subfactor $(R^G \subseteq R)$ is w-cyclic if and only if $G$ is linearly primitive, whereas it is cyclic if and only if $G$ is cyclic, but ``linearly primitive'' is strictly weaker than ``cyclic'', see for example $S_3$. That's why the name w-cyclic (i.e. weakly cyclic) was chosen. We are looking for a natural additional assumption to w-cyclic for having a complete characterization of the cyclic subfactors.

 We investigate some extensions of Theorem \ref{introthm} to distributive subfactor planar algebras (i.e. without assuming the biprojections to be normal). We can reduce to the boolean case (i.e. the biprojections lattice is equal to $\mathcal{B}_n$, the subsets lattice of $\{1, \dots , n \}$, for some $n$); we give a proof for $n \le 4$, as a corollary of the following result.
\begin{theorem} \label{le2intro}
A distributive subfactor planar algebra with the maximal biprojections $b_1, \dots, b_n$  satisfying $\sum_i \frac{1}{[id:b_i]} \le 2$, is w-cyclic.
\end{theorem}

The reformulation of Theorems \ref{introthm} and \ref{le2intro}, to interval of finite groups, gives surprising purely group theoretic results:

\begin{corollary}[O. Ore, \cite{or}] \label{introre2} If the interval of finite groups $[H,G]$ is distributive, then $\exists g \in G$ such that $\langle H, g \rangle = G$. 
\end{corollary}  

\begin{remark} Corollary \ref{dualore} is a new result in finite group theory, it is proved by subfactor methods, and is almost the dual of Corollary \ref{introre2}. \end{remark}

\begin{corollary} \label{dualore} If the interval of finite groups $[H,G]$ is distributive, and if one of the following (non-equivalent) statements occurs, 
\begin{itemize}
\item[(1)]  $\forall K \in [H,G] ,\ \forall g \in G ,\ HgK = KgH$,
\item[(2)] $\sum_i \frac{1}{\vert K_i : H  \vert} \le 2$, with $K_1, \dots , K_n$  the minimal overgroups of $H$,
\end{itemize}
then $\exists V$ irreducible complex representation of $G$ such that $G_{(V^H)} = H$. 
\begin{remark}  Moreover, if $H$ is core-free, then $G$ is linearly primitive. Also,  (2) is clearly satisfied for $\left\vert [H,G] \right\vert$  or  $|G:H| < 32$. The statement of Corollary \ref{dualore} is conjectured true without assuming (1) or (2).  \end{remark}
\end{corollary} 
 We also describe others applications of Theorem \ref{introthm}. First for an irreducible finite index subfactor planar algebra, we get an upper bound, called the \textit{cyclic length}, for the minimal number of minimal projections generating the identity biprojection. The reformulation of that in the  subfactors, quantum groups and finite group theories gives applications in each of these fields. For example, in the subfactors theory, the cyclic length of a subfactor $(N \subseteq M)$ gives a non-trivial upper bound for the minimal number of algebraic irreducible sub-$N$-$N$-bimodules of $M$ generating $M$ (as von Neumann algebra); in particular, if $(N \subseteq M)$ is cyclic then $M$ can be generated by one irreducible sub-$N$-$N$-bimodule. For a finite group $G$, the cyclic length gives a non-trivial upper bound for the minimal number of elements generating $G$, and of irreducible complex representations of $G$ generating the left regular representation.
So we have found a bridge linking combinatorics and representations, built in the language of subfactor planar algebras.\vspace*{0mm} \footnotesize
\tableofcontents
\normalsize 
\newpage
\section{Ore's theorem for groups and intervals} 

\begin{definition}
A lattice $(L, \wedge , \vee)$ is a partially ordered set (or poset) $L$  in which every two elements $a,b$ have a unique supremum (or join) $a \vee b$ and a unique infimum (or meet) $a \wedge b$.
\end{definition}  
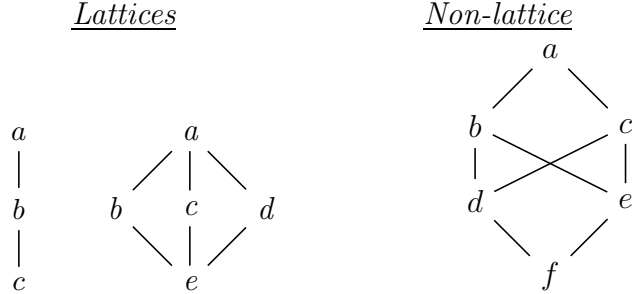
\begin{examples} \hspace*{1cm}
\begin{center}
\underline{Lattices} $\hspace{3cm}$ \underline{Non-lattice} \\
$\begin{tikzpicture}
\node (A1) at (0,0) {\small $a$};
\node (A3) at (0,-1) {\small $b$};
\node (A5) at (0,-2) {\small $c$};
\tikzstyle{segm}=[-,>=latex, semithick]
\draw [segm] (A1)to(A3); \draw [segm] (A3)to(A5);
\end{tikzpicture}$ $\hspace{0.5cm}$ $\begin{tikzpicture}
\node (A1) at (0,0) {a};
\node (A2) at (-1,-1) {b};
\node (A3) at (0,-1) {c};
\node (A4) at (1,-1) {d};
\node (A5) at (0,-2) {e};
\tikzstyle{segm}=[-,>=latex, semithick]
\draw [segm] (A1)to(A2); \draw [segm] (A1)to(A3);\draw [segm] (A1)to(A4);
\draw [segm] (A2)to(A5);\draw [segm] (A3)to(A5);\draw [segm] (A4)to(A5);
\end{tikzpicture}$ $\hspace{2cm}$  $\begin{tikzpicture}
\node (A1) at (0,0) {$a$};
\node (A2) at (-1,-1) {$b$};
\node (A4) at (1,-1) {$c$};
\node (A5) at (-1,-2) {$d$};
\node (A7) at (1,-2) {$e$};
\node (A8) at (0,-3) {$f$};
\tikzstyle{segm}=[-,>=latex, semithick]
\draw [segm] (A1)to(A2); \draw [segm] (A1)to(A4);
\draw [segm] (A2)to(A5); \draw [segm] (A2)to(A7);
\draw [segm] (A4)to(A7);  \draw [segm] (A4)to(A5);
\draw [segm] (A5)to(A8);\draw [segm] (A7)to(A8);
\end{tikzpicture}$ 
\end{center}
\end{examples}

\begin{definition}
A sublattice of $(L, \wedge , \vee)$ is a subposet $L' \subseteq L$ such that $(L', \wedge , \vee)$ is also a lattice. Let $a,b \in L$ with $a \le b$, then the interval $[a,b]$ is the sublattice $\{c \in L \ \vert \ a \le c \le b \}$. \end{definition}

\begin{definition}
A finite lattice $L$ admits a smallest and a greatest element, called $0$ and $1$. 
\end{definition} 

\begin{definition} \label{top}
The top interval of a finite lattice $L$ is  $[B,1]$ with $B= \bigwedge_i B_i$ and  $B_1, \dots, B_n$ the maximal elements.
\end{definition}

 \begin{definition} \label{bottom}
The bottom interval of a finite lattice $L$ is  $[0,b]$ with $b= \bigvee_i b_i$ and  $b_1, \dots, b_m$ the minimal elements.
\end{definition}

\begin{definition}
The height $h(L)$ of a finite lattice $L$ is the greatest length $\ell$ of a chain $ 0 < a_1 < a_2 < \cdots < a_{\ell} = 1$ with $a_i \in \mathcal{L}$.
\end{definition} 

\begin{definition}  Let $G$ be a finite group. The set of subgroups $ K \subseteq G$ is a lattice, noted $\mathcal{L}(G)$, ordered by $\subseteq$, with $K_1 \vee K_2 = \langle K_1,K_2 \rangle$ and $K_1 \wedge K_2 =  K_1 \cap K_2 $. Let $\mathcal{L}(H \subseteq G)$ be the interval lattice $[H , G]$. \end{definition}
\begin{examples} We give the lattices $\mathcal{L}(\mathbb{Z}/6)$, $\mathcal{L}(S_3)$ and  $[S_2 , S_4]$.
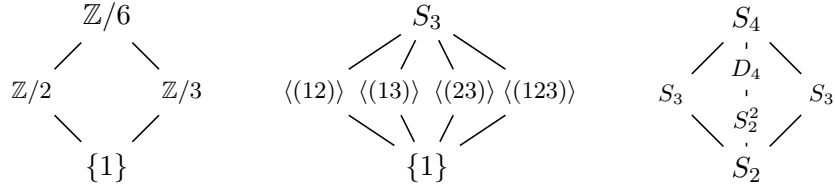
\begin{center}
$\begin{tikzpicture}
\node (A1) at (0,0) {\small $\mathbb{Z}/6$};
\node (A2) at (-1,-1) {\scriptsize $\mathbb{Z}/2$};
\node (A4) at (1,-1) {\scriptsize $\mathbb{Z}/3$};
\node (A5) at (0,-2) {\small $\{ 1 \}$};
\tikzstyle{segm}=[-,>=latex, semithick]
\draw [segm] (A1)to(A2); \draw [segm] (A1)to(A4);
\draw [segm] (A2)to(A5);\draw [segm] (A4)to(A5);
\end{tikzpicture}$  $\hspace{0.5cm}$  
$\begin{tikzpicture}
\node (A1) at (0,0) {\small $S_3$};
\node (A2) at (-1.5,-1) {\scriptsize $\langle (12) \rangle$};
\node (A3) at (-0.5,-1) {\scriptsize $\langle (13) \rangle$};
\node (A4) at (0.5,-1) {\scriptsize $\langle (23) \rangle$};
\node (A4b) at (1.5,-1) {\scriptsize $\langle (123) \rangle$};
\node (A5) at (0,-2) {\small $\{ 1 \}$};
\tikzstyle{segm}=[-,>=latex, semithick]
\draw [segm] (A1)to(A2); \draw [segm] (A1)to(A3);\draw [segm] (A1)to(A4);
\draw [segm] (A1)to(A4b);  \draw [segm] (A4b)to(A5);
\draw [segm] (A2)to(A5);\draw [segm] (A3)to(A5);\draw [segm] (A4)to(A5);
\end{tikzpicture}$  $\hspace{0.5cm}$ 
$\begin{tikzpicture}
\node (A1) at (0,0) {\small $S_4$};
\node (A2) at (-1,-1) {\scriptsize $S_3$};
\node (A3a) at (0,-2/3) {\scriptsize $D_4$};
\node (A3b) at (0,-4/3) {\scriptsize $S_2^2$};
\node (A4) at (1,-1) {\scriptsize $S_3$};
\node (A5) at (0,-2) {\small $S_2$};
\tikzstyle{segm}=[-,>=latex, semithick]
\draw [segm] (A1)to(A2); \draw [segm] (A1)to(A3a);\draw [segm] (A1)to(A4); \draw [segm] (A3a)to(A3b); \draw [segm] (A2)to(A5);\draw [segm] (A3b)to(A5);\draw [segm] (A4)to(A5);
\end{tikzpicture}$ 
\end{center} 
\end{examples}

\begin{remark}
Every finite lattice can't be realized as a subgroups lattice $\mathcal{L}(G)$, the first counter-example is the pentagon lattice (defined in Theorem \ref{dipe}) but it can be realized as an intermediate subgroups lattice $[H,G]$ (see \cite{wa} p 331). The realizability of every finite lattice as an intermediate subgroups lattice is an open problem\footnote{http://mathoverflow.net/q/196033/34538}, but the Shareshian's conjecture (see \cite{asch} p72) states that the following lattice (among others)
 
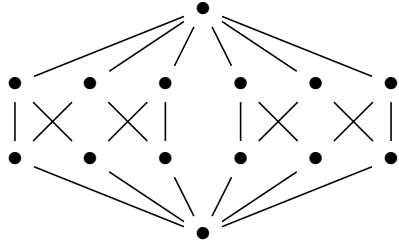
\begin{center} $\begin{tikzpicture}
\node (A1) at (0,0) {$\bullet$};
\node (A2) at (-2.5,-1) {$\bullet$};
\node (A3) at (-1.5,-1) {$\bullet$};
\node (A4) at (-0.5,-1) {$\bullet$};
\node (A5) at (0.5,-1) {$\bullet$};
\node (A6) at (1.5,-1) {$\bullet$};
\node (A7) at (2.5,-1) {$\bullet$};
\node (A8) at (-2.5,-2) {$\bullet$};
\node (A9) at (-1.5,-2) {$\bullet$};
\node (A10) at (-0.5,-2) {$\bullet$};
\node (A11) at (0.5,-2) {$\bullet$};
\node (A12) at (1.5,-2) {$\bullet$};
\node (A13) at (2.5,-2) {$\bullet$};
\node (A14) at (0,-3) {$\bullet$};
\tikzstyle{segm}=[-,>=latex, semithick]
\draw [segm] (A1)to(A2); \draw [segm] (A3)to(A8); \draw [segm] (A14)to(A13); 
\draw [segm] (A1)to(A3); \draw [segm] (A3)to(A10); \draw [segm] (A14)to(A12); 
\draw [segm] (A1)to(A4); \draw [segm] (A4)to(A9); \draw [segm] (A14)to(A11); 
\draw [segm] (A1)to(A5); \draw [segm] (A4)to(A10); \draw [segm] (A14)to(A10); 
\draw [segm] (A1)to(A6); \draw [segm] (A5)to(A11); \draw [segm] (A14)to(A9); 
\draw [segm] (A1)to(A7); \draw [segm] (A5)to(A12); \draw [segm] (A14)to(A8); 
\draw [segm] (A2)to(A8); \draw [segm] (A6)to(A11); \draw [segm] (A7)to(A12); 
\draw [segm] (A2)to(A9); \draw [segm] (A6)to(A13); \draw [segm] (A7)to(A13); 
\end{tikzpicture} 
$ \end{center}
is not realizable as an intermediate subgroups lattice (and so its realizability as an intermediate subfactors lattice is an interesting problem).    
\end{remark}

\begin{definition}
A lattice $(L, \wedge , \vee)$ is distributive if $\forall a,b,c \in L$: 
$$a \vee (b \wedge c) = (a \vee b) \wedge (a \vee c)$$  
 {\scriptsize \textnormal{ [or equivalently: $a \wedge (b \vee c) = (a \wedge b) \vee (a \wedge c)$]}} \end{definition}
 \begin{examples} \hspace*{1cm}
 \begin{center}
 \underline{Non-distributive} $\hspace{4cm}$ \underline{Distributive}  \\
$
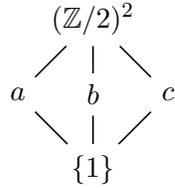
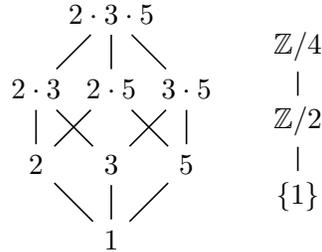
\begin{tikzpicture}
\node (A1) at (0,1) {\small $(\mathbb{Z}/2)^2$};
\node (A2) at (-1,0) {\small $a$};
\node (A3) at (0,0) {\small $b$};
\node (A4) at (1,0) {\small $c$};
\node (A5) at (0,-1) {\small $\{ 1 \}$};
\node (A6) at (0,-1.7) {\scriptsize with $a,b,c \simeq \mathbb{Z}/2$};
\tikzstyle{segm}=[-,>=latex, semithick]
\draw [segm] (A1)to(A2); \draw [segm] (A1)to(A3);\draw [segm] (A1)to(A4);
\draw [segm] (A2)to(A5);\draw [segm] (A3)to(A5);\draw [segm] (A4)to(A5);
\end{tikzpicture}$ $\hspace{2.5cm}$
$\begin{tikzpicture}
\node (A1) at (0,0) {\small $2 \cdot 3 \cdot 5$};
\node (A2) at (-1,-1) {\small $2 \cdot 3$};
\node (A3) at (0,-1) {\small $2 \cdot 5$};
\node (A4) at (1,-1) {\small $3 \cdot 5$};
\node (A5) at (-1,-2) {\small $2$};
\node (A6) at (0,-2) {\small $3$};
\node (A7) at (1,-2) {\small $5$};
\node (A8) at (0,-3) {\small $1$};
\tikzstyle{segm}=[-,>=latex, semithick]
\draw [segm] (A1)to(A2); \draw [segm] (A1)to(A3);\draw [segm] (A1)to(A4);
\draw [segm] (A2)to(A5);\draw [segm] (A2)to(A6);
\draw [segm] (A3)to(A5);\draw [segm] (A3)to(A7);
\draw [segm] (A4)to(A6);\draw [segm] (A4)to(A7);
\draw [segm] (A5)to(A8);\draw [segm] (A6)to(A8);\draw [segm] (A7)to(A8);
\end{tikzpicture}$
 $\hspace{0.25cm}$ 
 $\begin{tikzpicture}
\node (A1) at (0,0) {\small $\mathbb{Z}/4$};
\node (A3) at (0,-1) {\small $\mathbb{Z}/2$};
\node (A5) at (0,-2) {\small $\{ 1 \}$};
\node (A6) at (0,-2.7) {\scriptsize };
\tikzstyle{segm}=[-,>=latex, semithick]
\draw [segm] (A1)to(A3); \draw [segm] (A3)to(A5);
\end{tikzpicture}$
 \end{center}
\begin{scriptsize} 
proof:  $a \wedge (b \vee c) = a \neq $ $ \{ 1\} = (a \wedge b) \vee (a \wedge c) $ 
\end{scriptsize}
\end{examples}
 
 \begin{theorem} \label{dipe}
A lattice is distributive if and only if it has no sublattice equivalent to the diamond lattice $M_3$ or the pentagon lattice $N_5$,  below. \end{theorem}
\begin{center} $\begin{tikzpicture}
\node (A1) at (0,0) {$\bullet$};
\node (A2) at (-1,-1) {$\bullet$};
\node (A3) at (0,-1) {$\bullet$};
\node (A4) at (1,-1) {$\bullet$};
\node (A5) at (0,-2) {$\bullet$};
\node (A6) at (0,-2.4) {};
\tikzstyle{segm}=[-,>=latex, semithick]
\draw [segm] (A1)to(A2); \draw [segm] (A1)to(A3);\draw [segm] (A1)to(A4);
\draw [segm] (A2)to(A5);\draw [segm] (A3)to(A5);\draw [segm] (A4)to(A5);
\end{tikzpicture} 
\hspace*{2cm}
\begin{tikzpicture}
\node (A1) at (0,0) {$\bullet$};
\node (A2) at (-1,-0.85) {$\bullet$};
\node (A3) at (-1,-1.75) {$\bullet$};
\node (A4) at (1,-1.3) {$\bullet$};
\node (A5) at (0,-2.6) {$\bullet$};
\tikzstyle{segm}=[-,>=latex, semithick]
\draw [segm] (A1)to(A2); \draw [segm] (A2)to(A3);\draw [segm] (A1)to(A4);
\draw [segm] (A3)to(A5);\draw [segm] (A4)to(A5);
\end{tikzpicture}
$ \end{center}
\begin{proof}
See \cite{gr} Theorem 101 p109.
\end{proof} 
\begin{remark}  \label{modular} The lattices admitting no sublattice equivalent to (just) the pentagon lattice $N_5$ are called modular in the literature.
\end{remark}
\begin{lemma} \label{distri}
 The distributivity is self-dual and hereditary, i.e. for $L$ distributive, its reverse lattice and its sublattices are also distributive. Moreover it is stable by concatenation and direct product. \end{lemma}   
 \begin{proof}
Immediate from the definition.
\end{proof} 

\begin{definition}
The subsets lattice of $\{ 1, \dots , n \}$ is called the boolean lattice of rank $n$ and noted $\mathcal{B}_n$ (see the lattice $\mathcal{B}_3$ below).
\begin{center} $\begin{tikzpicture}
\node (A1) at (0,0) {$\bullet$};
\node (A2) at (-1,-1) {$\bullet$};
\node (A3) at (0,-1) {$\bullet$};
\node (A4) at (1,-1) {$\bullet$};
\node (A8) at (-1,-2) {$\bullet$};
\node (A9) at (0,-2) {$\bullet$};
\node (A10) at (1,-2) {$\bullet$};
\node (A14) at (0,-3) {$\bullet$};
\tikzstyle{segm}=[-,>=latex, semithick]
\draw [segm] (A1)to(A2); \draw [segm] (A3)to(A8); 
\draw [segm] (A1)to(A3); \draw [segm] (A3)to(A10);  
\draw [segm] (A1)to(A4); \draw [segm] (A4)to(A9);  
\draw [segm] (A14)to(A9); \draw [segm] (A14)to(A10); 
 \draw [segm] (A14)to(A8); \draw [segm] (A2)to(A9);
\draw [segm] (A2)to(A8); \draw [segm] (A4)to(A10); 

\end{tikzpicture} 
$ \end{center}
 \end{definition}
 
 \begin{remark}
 A boolean lattice is distributive. 
 \end{remark}
 
 \begin{definition} \label{comp}
Let $b^{\complement}$ be the complement of the element $b$ in a boolean lattice, i.e. $b \wedge b^{\complement} = 0$ and $b \vee b^{\complement} = 1$ (existence and unicity of the complement is immediate by boolean structure).
 \end{definition}   

\begin{lemma} \label{topBn}
The top interval of a distributive lattice is boolean.
 \end{lemma}   
 \begin{proof}
See \cite{sta} items a-i on pages 254-255. \end{proof} 

The following theorem is due to Oystein Ore (1938). It's the starting point of this work. 
\begin{theorem} \label{ore1}
 
A finite group $G$ is cyclic if and only if its subgroups lattice $\mathcal{L}(G)$ is distributive. 
\end{theorem} 
\begin{proof}
See \cite{or} Theorem 4 p 267 or \cite{sc} Theorem 1.2.3 p12 for a proof of the more general statement ``$G$ locally cyclic if and only if $\mathcal{L}(G)$ distributive''. We give an indirect short proof assuming $G$ finite:  \\ 
 ($\Leftarrow$): By applying Theorem \ref{ore2} with $H = \{1\}$.   \\
 ($\Rightarrow$): A finite cyclic group $G = \mathbb{Z}/n$ has exactly one subgroup of order $d$, noted $\mathbb{Z}/d$, for every divisor $d$ of $ord(G)$; now $\mathbb{Z}/d_1 \vee \mathbb{Z}/d_2 = \mathbb{Z}/lcm(d_1,d_2)$ and $\mathbb{Z}/d_1 \wedge \mathbb{Z}/d_2 = \mathbb{Z}/gcd(d_1,d_2)$, but lcm and gcd are distributive, so the result follows.
\end{proof}  
 
 \begin{remark}
Theorem \ref{ore1} admits the following surprising corollary: any distributive finite lattice which is not equivalent to $\mathcal{L}(\mathbb{Z}/n)$, is not realizable as a subgroup lattice: for example the concatenation lattices $\mathcal{L}(\mathbb{Z}/m) * \mathcal{L}(\mathbb{Z}/n)$ with $m$ or $n$ not a prime power, as   $\mathcal{L}(\mathbb{Z}/6) * \mathcal{L}(\mathbb{Z}/6)$ below.
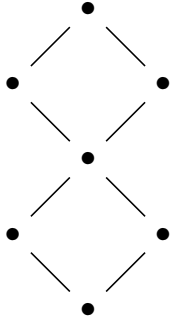
\begin{center} $\begin{tikzpicture}
\node (A1) at (0,0) {$\bullet$};
\node (A2) at (-1,-1) {$\bullet$};
\node (A4) at (1,-1) {$\bullet$};
\node (A5) at (0,-2) {$\bullet$};
\node (A6) at (-1,-3) {$\bullet$};
\node (A8) at (1,-3) {$\bullet$};
\node (A9) at (0,-4) {$\bullet$};
\tikzstyle{segm}=[-,>=latex, semithick]
\draw [segm] (A1)to(A2); \draw [segm] (A1)to(A4);
\draw [segm] (A2)to(A5);\draw [segm] (A4)to(A5);
\draw [segm] (A5)to(A6);\draw [segm] (A5)to(A8);
\draw [segm] (A6)to(A9);\draw [segm] (A8)to(A9);
\end{tikzpicture} 
$ \end{center}
 Nevertheless all the distributive finite lattices are realizable as intermediate subgroups lattices (see \cite{palfy} Corollary 5.3. p 448).
\end{remark}

\begin{definition} 
Let $G$ be a finite group and $H$ a subgroup, then $G$ is called $H$-cyclic if there is $g \in G$ such that $\langle H,g \rangle = G$.
\end{definition}

\begin{lemma} \label{max} Let $G$ be a finite group and $M$ a maximal subgroup, then $G$ is $M$-cyclic.
\end{lemma}
\begin{proof}
Let $g \in G$ such that $g \not \in M$, then $\langle M,g \rangle = G$ by maximality.
\end{proof}

\begin{lemma} \label{topred} 
Let $[H,G]$ be an interval of finite groups and $[K,G]$ be its top interval (see Section \ref{2}). If $G$ is $K$-cyclic, then it is $H$-cyclic
\end{lemma}
\begin{proof}
Let $g \in G$ such that $\langle K,g \rangle = G$. Now for any $M$ maximal subgroup of $G$ containing $H$, we have $K \subset M$ by definition, and so $g \not \in M$ (otherwise $\langle K,g \rangle \subset M$),  then a fortiori $\langle H,g \rangle \not \subset M$, for any such $M$. It follows that $\langle H,g \rangle = G$.
\end{proof} 

O. Ore has extended one side of Theorem \ref{ore1} to interval of finite groups as follows (see Theorem 7 p269 in \cite{or}). It's precisely this theorem that this paper generalizes to the subfactor planar algebras. The new proof above is a translation of our planar algebraic proof to the group theoretic framework.
\begin{theorem} \label{ore2}
Let $[H,G]$ be a distributive interval of finite groups. Then $\exists g \in G$ such that $\langle H,g \rangle = G$ (i.e. $G$ is $H$-cyclic).
\end{theorem} 
\begin{proof}
By Lemmas \ref{topBn} and \ref{topred} we can assume $[H,G]$ to be boolean.
Let $M$ be a maximal subgroup of $G$ containing $H$, and $M^{\complement}$ the complement of $M$ in $[H,G]$ (see Definition \ref{comp}). By Lemma \ref{max} and induction on the height of the lattice, we can assume $M$ and $M^{\complement}$  both $H$-cyclic, i.e. there are $a, b \in G$ such that $\langle H,a \rangle = M$ and $\langle H,b \rangle = M^{\complement}$. Let $g=a  b$ then $a=g  b^{-1}$ and $b=a^{-1}g$, so $\langle H,a,g \rangle = \langle H,g,b \rangle = \langle H,a,b \rangle =  M \vee M^{\complement} = G$.    
Now, $\langle H,g \rangle =  \langle H,g \rangle \vee H = \langle H,g \rangle \vee (M \wedge M^{\complement})$ but by \textit{distributivity} $\langle H,g \rangle \vee (M \wedge M^{\complement}) = (\langle H,g \rangle \vee M \rangle) \wedge (\langle H,g \rangle \vee M^{\complement} \rangle)$. So $ \langle H,g \rangle = \langle H,a,g \rangle \wedge \langle H,g,b \rangle = G$. The result follows.
\end{proof}
 
The lattice $[S_2 , S_4]$ is not distributive but $\langle S_2,(1234) \rangle = S_4$, so the converse is  false.  We are looking for a complete equivalent characterization of the distributivity property.

\section{Subfactors and planar algebras} 
\subsection{Short introduction to subfactors}  \hspace*{1cm} \\
For more details see the book of V. Jones and V.S. Sunder \cite{js}. Let $B(H)$ be the algebra of bounded operators on $H$  a separable Hilbert space. A $\star$-algebra $M \subseteq B(H)$ is a \textit{von Neumann algebra} if it has a unit element and is equal to its bicommutant ($I \in M = M^{\star}=M'' $).
It is \textit{hyperfinite} if it's a ``limit'' of finite dimensional von Neumann algebras (see \cite{js} p18). It's a \textit{factor} if its center is trivial ($M' \cap M = \mathbb{C}I$). A factor $M$ is type ${\rm II}_1$ if it admits a \textit{trace} $tr$ such that the set of projections maps to $[0,1]$. From $tr$ we get the space $L^2(M,tr)$. Every factor here will be of type ${\rm II}_1$ (there is a unique hyperfinite one called $R$). A \textit{subfactor} is an inclusion of factors $(N \subseteq M)$. It's \textit{irreducible}  if the relative commutant is trivial ($N' \cap M = \mathbb{C}I$). Let $e^M_N : L^2(M) \to L^2(N)$ be the orthogonal projection and let $M_1 = \langle M,e^M_N \rangle$, then the \textit{index} of $(N \subseteq M)$ is $$[M : N] = dim_N(L^2(M)) =  (tr_{M_1}(e^M_N))^{-1}$$
By \cite{jo2} the set of indices of subfactors is  %Jones, 1983) 
$$\{ 4 cos^2(\frac{\pi}{n}) \vert n \geq 3  \} \cup [4,\infty]$$

An irreducible finite index subfactor has a finite intermediate subfactors \textit{lattice} \cite{wa} (as for an interval of finite groups). Any finite group $G$  acts outerly on the hyperfinite   ${\rm II}_1$ factor $R$, and the fixed point subfactor $(R^G \subseteq R)$, of index $\vert G \vert$, is irreducible and remembers $G$ \cite{su}, which is a complete invariant  (because two outer actions are outer conjugate \cite{jo}; and this is true in general if and only if $G$ is amenable \cites{jo3, oc}). % Jones 1980)
This means that for $G$ \textit{finite},  $(R^{G} \subseteq R)$ is the ``same thing'' than $G$.
 %Jones 1983; Ocneanu 1985)
  The Galois correspondence  \cites{nk, ilp} means that for every intermediate subfactor $R^{G} \subseteq P \subseteq R $ there is $H \le G$ such that $  P=R^{H}$. %Suzuki 1958; Nakamura-Takeda 1960 
In general $(R^{G} \subseteq R^{H})$ does \textit{not} remember $[H,G]$ up to equivalence (see Definition \ref{equiv}) \cite{sk}. %Kodiyalam-Sunder 2000)   
The subfactor $(R^{G} \subseteq R^{H})$ is the dual of $(R \rtimes H \subseteq R \rtimes G)$ whose lattice of intermediate subfactors is exactly $[H,G]$.   

The basic construction is the following tower
$$N=M_{-1} \subseteq M=M_0 \subseteq M_1 \subseteq M_2 \subseteq \cdots \subseteq M_n \subseteq \cdots  $$ 
with $M_{n+1}:=\langle M_n, e_n \rangle$ and $e_n: L^2(M_n) \to L^2(M_{n-1})$ Jones projection. At \textit{finite index} the higher relative commutants $\mathcal{P}_{n,+}= N' \cap M_{n-1}$ and $\mathcal{P}_{n,-}= M' \cap M_{n}$, are finite-dimensional ${\rm C}^{\star}$-algebras.  The subfactor is \textit{finite depth} if the number of factors of $\mathcal{P}_{n,+}$ is bounded, and irreducible depth $2$ if $\mathcal{P}_{3,+}$ is a factor. The standard invariant of $(N \subseteq M)$ is the following grid 
\begin{center}
$\mathbb{C} = \mathcal{P}_{0,+} \subseteq \mathcal{P}_{1,+} \subseteq \mathcal{P}_{2,+} \subseteq \cdots \subseteq \mathcal{P}_{n,+} \subseteq \cdots  $  \\
\hspace*{2.7cm} $ \cup   \hspace{1cm}  \cup \hspace{2cm}  \cup$ \newline
\hspace*{1.6cm} $\mathbb{C} = \mathcal{P}_{0,-} \subseteq \mathcal{P}_{1,-} \subseteq \cdots \subseteq \mathcal{P}_{n-1,-} \subseteq \cdots  $  \end{center}
which is a complete invariant on the amenable case (\cite{po}). \\ %Popa94) 
Every finite depth subfactor of the hyperfinite ${\rm II}_1$ factor is amenable. 

\subsection{Short introduction to planar algebras} \label{planar} \hspace*{1cm} \\
The idea of the planar algebra is to be a diagrammatic axiomatization of the standard invariant.  For more details, see \cite{jo4} by V. Jones  and \cite{sk2} by  V.S. Sunder and V. Kodiyalam, with recorded lectures\footnote{\tiny https://www.youtube.com/playlist?list=PLjudt2gd3iAe40juwS3dKocAoooXF91Vi}. The diagrams of this subsection come from this paper \cite{ep} of E. Peters.  

 A (shaded) \textit{planar tangle} is the data of  finitely many ``input'' disks, one ``output'' disk, non-intersecting strings giving $2n$ intervals per disk and one $\star$-marked  interval per disk. A tangle is defined up to ``isotopy''.  
  
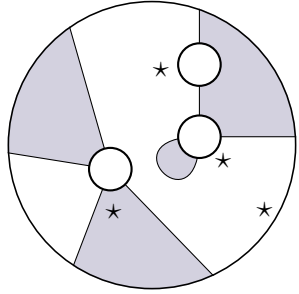
 \begin{center}
$\begin{tikzpicture}[scale=.64]
	\clip (0,0) circle (3cm);
	
	\begin{scope}[shift=(10:1cm)]
		\draw[shaded] (0,0)--(0:6cm)--(90:6cm)--(0,0);	
		\draw[shaded] (0,0) .. controls ++(180:2cm) and ++(-90:2cm) .. (0,0);
	\end{scope}
	
	\draw[shaded] (-150:1cm) -- (120:4cm) -- (180:4cm) -- (-150:1cm);
	\draw[shaded] (-150:1cm) -- (-120:4cm) -- (-60:4cm) -- (-150:1cm);
	
	\begin{scope}[shift=(10:1cm)]	
		\node at (0,0) [Tbox, inner sep=2mm] {};
		\node at (90:1.5cm) [Tbox, inner sep=2mm] {};
		\node at (-45:.7cm) {$\star$};
		\node at (120:1.6cm) {$\star$};
	\end{scope}
	\node at (-150:1cm) [Tbox, inner sep=2mm] {};
	\node at (-120:1.6cm) {$\star$};
	\node at (-30:2.7cm) {$\star$};
	
	\draw[very thick] (0,0) circle (3cm);
\end{tikzpicture} $
 \end{center}

To \textit{compose} two planar tangles, put the output disk of one into an input of the other, having as many intervals, same shading of marked intervals and such that the $\star$-marked intervals coincide. Finally we remove the coinciding circles (possibly zero, one or several compositions).
 \begin{center}
$
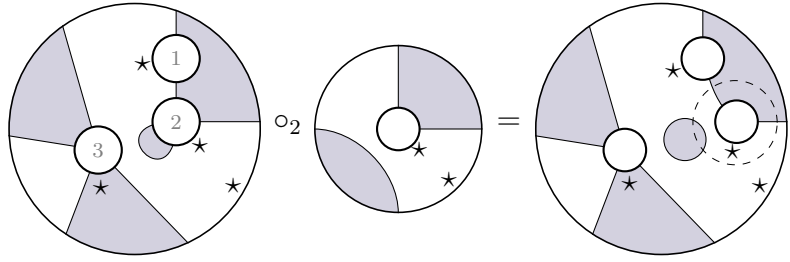
\begin{tikzpicture}[scale=.8, PAdefn]
	\clip (0,0) circle (3cm);
	
	\begin{scope}[shift=(10:1cm)]	
		\draw[shaded] (0,0)--(0:6cm)--(90:6cm)--(0,0);	
		\draw[shaded] (0,0) .. controls ++(180:2cm) and ++(-90:2cm) .. (0,0);
	\end{scope}
	
	\draw[shaded] (-150:1cm) -- (120:4cm) -- (180:4cm) -- (-150:1cm);
	\draw[shaded] (-150:1cm) -- (-120:4cm) -- (-60:4cm) -- (-150:1cm);
	
	\begin{scope}[shift=(10:1cm)]	
		\node at (0,0) [Tbox, inner sep=1.4mm] {\tiny{\textcolor{gray}{2}}};
		\node at (90:1.5cm) [Tbox, inner sep=1.4mm] {\tiny{\textcolor{gray}{1}}};
		\node at (-45:.8cm) {$\star$};
		\node at (120:1.6cm) {$\star$};
	\end{scope}
	\node at (-150:1cm) [Tbox, inner sep=1.4mm] {\tiny{\textcolor{gray}{3}}};
	\node at (-120:1.6cm) {$\star$};
	\node at (-30:2.7cm) {$\star$};
	
	\draw[very thick] (0,0) circle (3cm);

\end{tikzpicture}
\,
\circ_{2}
\,
\begin{tikzpicture}[scale=.8, PAdefn]
	\clip (0,0) circle (2cm);

	\draw[shaded] (0:4cm)--(0,0)--(90:4cm);
	\draw[shaded] (180:4cm)--(180:2cm) .. controls ++(0:1cm) and ++(90:1cm) .. (270:2cm) -- (270:4cm);
	
		\node at (0,0) [Tbox, inner sep=2mm] {};

	\node at (-45:1.7cm) {$\star$};
	\node at (-45:.7cm) {$\star$};	

	\draw[very thick] (0,0) circle (2cm);

\end{tikzpicture}
\,
=
\,
\begin{tikzpicture}[scale=.8, PAdefn]
	\clip (0,0) circle (3cm);
	
	\begin{scope}[shift=(10:1cm)]	
		\draw[shaded] (0:6cm)--(90:6cm)--(90:2cm) .. controls ++(-90:1cm) and ++(180:.5cm) .. (0:1cm)--(0:6cm);	
		\draw[shaded] (-135:.6cm) circle (.5cm);
	\end{scope}
	
	\draw[shaded] (-150:1cm) -- (120:4cm) -- (180:4cm) -- (-150:1cm);
	\draw[shaded] (-150:1cm) -- (-120:4cm) -- (-60:4cm) -- (-150:1cm);
	
	\begin{scope}[shift=(10:1cm)]	
		\node at (0:.8cm) [Tbox, inner sep=2mm] {};
		\node at (90:1.5cm) [Tbox, inner sep=2mm] {};
		\node at (-45:1cm) {$\star$};
		\node at (120:1.4cm) {$\star$};
	\end{scope}
	\node at (-150:1cm) [Tbox, inner sep=2mm] {};
	\node at (-120:1.5cm) {$\star$};
	\node at (-30:2.7cm) {$\star$};
	\draw[dashed, ultra thin] (1.75,0.15) circle (1cm);
	
	\draw[very thick] (0,0) circle (3cm);

\end{tikzpicture}
$  
\end{center}

The \textit{planar operad} is the set of all the planar tangles (up to isomorphism) with such compositions. A \textit{planar algebra} is a family of vector spaces $(\mathcal{P}_{n,\pm})_{n \in \mathbb{N}}$, called $n$-box spaces, on which \textit{acts} the planar operad, i.e. for any tangle $T$ (with one output disk and $r$ input disks with $2n_0$ and $2n_1, \dots, 2n_r$ intervals respectively) there is a linear map $$Z_T : \mathcal{P}_{n_1,\epsilon_1} \otimes \cdots \otimes  \mathcal{P}_{n_r,\epsilon_r} \to \mathcal{P}_{n_0,\epsilon_0}$$ with $\epsilon_i \in \{+,-\}$ according to the shading of the $\star$-marked intervals, and these maps (also called partition functions) respect the composition of tangle in such a way that all the diagrams as below commute.
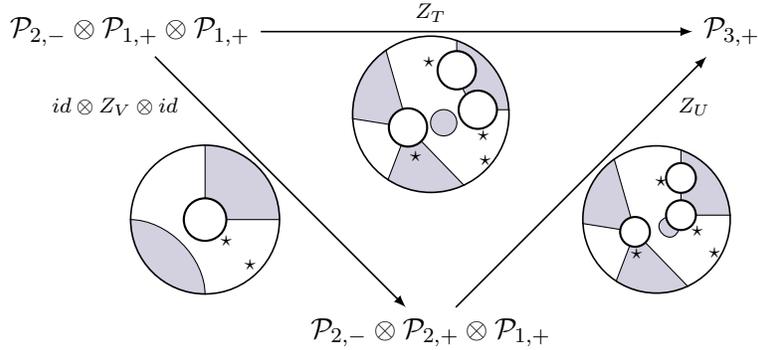
\begin{center} $\begin{tikzpicture}
\node[inner sep=0pt] at (7,0.5)
    {  
    
    \begin{tikzpicture}[scale=.33]
	\clip (0,0) circle (3cm);
	
	\begin{scope}[shift=(10:1cm)]	
		\draw[shaded] (0,0)--(0:6cm)--(90:6cm)--(0,0);	
		\draw[shaded] (0,0) .. controls ++(180:2cm) and ++(-90:2cm) .. (0,0);
	\end{scope}
	
	\draw[shaded] (-150:1cm) -- (120:4cm) -- (180:4cm) -- (-150:1cm);
	\draw[shaded] (-150:1cm) -- (-120:4cm) -- (-60:4cm) -- (-150:1cm);
	
	\begin{scope}[shift=(10:1cm)]	
		\node at (0,0) [Tbox, inner sep=1.4mm] {};
		\node at (90:1.5cm) [Tbox, inner sep=1.4mm] {};
		\node at (-45:.92cm) {\tiny{$\star$}};
		\node at (120:1.6cm) {\tiny{$\star$}};
	\end{scope}
	\node at (-150:1cm) [Tbox, inner sep=1.4mm] {};
	\node at (-120:1.6cm) {\tiny{$\star$}};
	\node at (-30:2.7cm) {\tiny{$\star$}};
	
	\draw[very thick] (0,0) circle (3cm);

\end{tikzpicture}
    };
\node[inner sep=0pt] at (1,0.5)
    {
    
    \begin{tikzpicture}[scale=.5]
	\clip (0,0) circle (2cm);

	\draw[shaded] (0:4cm)--(0,0)--(90:4cm);
	\draw[shaded] (180:4cm)--(180:2cm) .. controls ++(0:1cm) and ++(90:1cm) .. (270:2cm) -- (270:4cm);
	
		\node at (0,0) [Tbox, inner sep=2mm] {};

	\node at (-45:1.7cm) {\tiny{$\star$}};
	\node at (-45:.8cm) {\tiny{$\star$}};

	\draw[very thick] (0,0) circle (2cm);

\end{tikzpicture}
    
    };
    \node[inner sep=0pt] at (4,1.9)
    {
    
  \begin{tikzpicture}[scale=.35]
	\clip (0,0) circle (3cm);
	
	\begin{scope}[shift=(10:1cm)]	
		\draw[shaded] (0:6cm)--(90:6cm)--(90:2cm) .. controls ++(-90:1cm) and ++(180:.5cm) .. (0:1cm)--(0:6cm);	
		\draw[shaded] (-135:.7cm) circle (.5cm);
	\end{scope}
	
	\draw[shaded] (-150:1cm) -- (120:4cm) -- (180:4cm) -- (-150:1cm);
	\draw[shaded] (-150:1cm) -- (-120:4cm) -- (-60:4cm) -- (-150:1cm);
	
	\begin{scope}[shift=(10:1cm)]	
		\node at (0:.8cm) [Tbox, inner sep=1.8mm] {};
		\node at (90:1.5cm) [Tbox, inner sep=1.8mm] {};
		\node at (-44:1.4cm) {\tiny{$\star$}};
		\node at (120:2.1cm) {\tiny{$\star$}};
	\end{scope}
	\node at (-150:1cm) [Tbox, inner sep=1.8mm] {};
	\node at (-110:1.7cm) {\tiny{$\star$}};
	\node at (-40:2.7cm) {\tiny{$\star$}};
	
	\draw[very thick] (0,0) circle (3cm);

\end{tikzpicture} 
    
    };
\node (A1) at  (0,3) {\small $\mathcal{P}_{2,-} \otimes \mathcal{P}_{1,+} \otimes \mathcal{P}_{1,+}$};
\node (A2) at (4,-1) {\small $\mathcal{P}_{2,-} \otimes \mathcal{P}_{2,+} \otimes \mathcal{P}_{1,+}
$};
\node (A3) at (8,3) {\small $\mathcal{P}_{3,+}$};
\node (A4) at (4,3.25) {\tiny $Z_T$};
\node (A5) at (-0.2,2) {\tiny $id \otimes Z_V \otimes id$};
\node (A6) at (7.5,2) {\tiny $Z_U$};
\tikzstyle{segm}=[->,>=latex, semithick]
\draw [segm] (A1)to(A2); 
\draw [segm] (A1)to(A3); 
\draw [segm] (A2)to(A3);
\end{tikzpicture}$
\end{center}

For example, the family of vector spaces $(\mathcal{T}_{n,\pm})_{n \in \mathbb{N}}$ generated by the planar tangles having $2n$ intervals on their ``output'' disk and a white (or black) shaded $\star$-marked interval, admits a planar algebra structure. The Temperley-Lieb-Jones planar algebra $\mathcal{TLJ}(\delta)$ is generated by the tangles without input disk; its $3$-box space $\mathcal{TLJ}_{3,+}(\delta)$ is generated by 
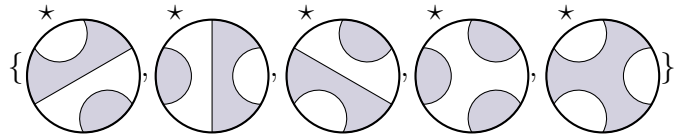
\begin{center} 
$ \{ \begin{tikzpicture}[TLEG, scale=1.5]
	\filldraw[shaded]  (30:1cm) arc (30:90:1cm) arc (30:-150:5mm) arc (150:210:1cm) -- cycle; 
	\filldraw[shaded]  (0,-1) arc (-90:-30:1cm) arc (30:210:5mm);
	\draw[thick] (0,0) circle (1cm);
	\node at (120:1.3cm) {$\star$};
\end{tikzpicture},
\begin{tikzpicture}[TLEG, scale=1.5]
	\filldraw[shaded]  (0,1) arc (90:30:1cm) arc (90:270:5mm) arc (-30:-90:1cm) -- cycle; 
	\filldraw[shaded]  (-1,0) arc (180:210:1cm) arc (-90:90:5mm) arc (150:180:1cm);
	\draw[thick] (0,0) circle (1cm);
	\node at (120:1.3cm) {$\star$};
\end{tikzpicture},
\begin{tikzpicture}[TLEG, rotate=180, scale=1.5]
	\filldraw[shaded]  (150:1cm) arc (150:90:1cm) arc (-210:-30:5mm) arc (30:-30:1cm)--cycle;
	\filldraw[shaded] (-90:1cm) arc (-90:-150:1cm) arc (150:-30:5mm);
	\draw[thick] (0,0) circle (1cm);
	\node at (-60:1.3cm) {$\star$};
\end{tikzpicture},
\begin{tikzpicture}[TLEG, scale=1.5]
	\filldraw[shaded]  (0,-1) arc (-90:-30:1cm) arc (30:210:5mm);
	\filldraw[shaded]  (-1,0) arc (180:210:1cm) arc (-90:90:5mm) arc (150:180:1cm);
	\filldraw[shaded] (90:1cm) arc (90:30:1cm) arc (-30:-210:5mm);
	\draw[thick] (0,0) circle (1cm);
	\node at (120:1.3cm) {$\star$};
\end{tikzpicture},
\begin{tikzpicture}[TLEG, scale=1.5]
	\filldraw[shaded]  (90:1cm) arc (30:-150:5mm) arc (150:210:1cm) arc (150:-30:5mm) arc (-90:-30:1cm) arc (-90:-270:5mm) arc (30:90:1cm);
	\draw[thick] (0,0) circle (1cm);
	\node at (120:1.3cm) {$\star$};
\end{tikzpicture}   \}$
\end{center}   
moreover, a closed string is replaced by a multiplication by $\delta$.\\
Note that the dimension of $\mathcal{TLJ}_{n,\pm}(\delta)$ is the Catalan number $\frac{1}{n+1}\binom{2n}{n}$.
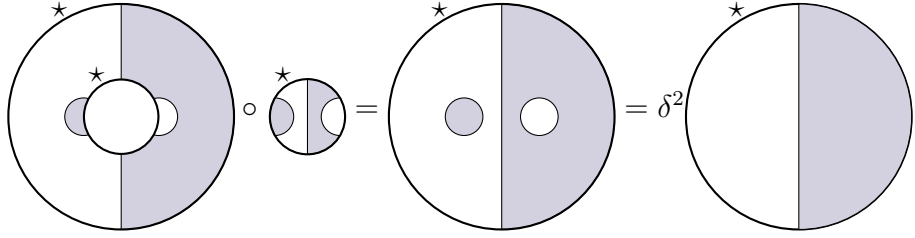
\begin{center}  

$\begin{tikzpicture}[TLEG]
	\filldraw[shaded] (90:3cm) -- (-90:3cm) arc (-90:90:3cm);
	\filldraw[unshaded] (0:1cm) circle (.5cm);
	\filldraw[shaded] (180:1cm) circle (.5cm);
	\node[Tbox,inner sep=3.5mm] at (0,0) {};
	\draw[thick] (0,0) circle (3cm);
	\node at (120:1.3cm) {$\star$};
	\node at (120:3.3cm) {$\star$};
\end{tikzpicture}
\circ
\begin{tikzpicture}[TLEG]
	\filldraw[shaded]  (0,1) arc (90:30:1cm) arc (90:270:5mm) arc (-30:-90:1cm) -- cycle; 
	\filldraw[shaded]  (-1,0) arc (180:210:1cm) arc (-90:90:5mm) arc (150:180:1cm);
	\draw[thick] (0,0) circle (1cm);
	\node at (120:1.3cm) {$\star$};
\end{tikzpicture}
=
\begin{tikzpicture}[TLEG]
	\filldraw[shaded] (90:3cm) -- (-90:3cm) arc (-90:90:3cm);
	\filldraw[unshaded] (0:1cm) circle (.5cm);
	\filldraw[shaded] (180:1cm) circle (.5cm);
	\draw[thick] (0,0) circle (3cm);
	\node at (120:3.3cm) {$\star$};
\end{tikzpicture}
=\delta^2
\begin{tikzpicture}[TLEG]
	\draw[thick]  (-3,0)arc (180:-180: 3cm);
	\filldraw[shaded]  (0,3)--(0,-3) arc (-90:90:3cm);
	\node at (120:3.3cm) {$\star$};
\end{tikzpicture}$

 \end{center} 
\subsection{Subfactor planar algebras} \hspace*{1cm} \\
A \textit{subfactor planar algebra} is a planar $\star$-algebra $(\mathcal{P}_{n,\pm})_{n \in \mathbb{N}}$ which is: 
\begin{itemize}
\item[(1)] Finite-dimensional: $dim (\mathcal{P}_{n,\pm}) < \infty$
\item[(2)] Evaluable:  $\mathcal{P}_{0,\pm} = \mathbb{C}$
\item[(3)] Spherical: $tr:=tr_r = tr_l$
\item[(4)] Positive:  $\langle a \vert b \rangle = tr(b^{\star}a)$ defines an inner product.
\end{itemize} 
Note that by (2) and (3), any closed string (shaded or not) counts for the same constant $\delta$.
\begin{center} $ a=
\begin{tikzpicture}[scale=.35, baseline]
\node at (150:4.3cm) {\tiny{$\star$}};
	\clip (0,0) circle (4cm);
  \draw[shaded] (-1,6)--(-0.5,6)--(-0.5,-6)--(-1,-6)--(-1,6);
	 \draw[shaded] (1,6)--(0.5,6)--(0.5,-6)--(1,-6)--(1,6);
	\node at (0,0) [Tbox, inner sep=2mm] {$a$};
	\node at (180:1.5cm) {\tiny{$\star$}};
	\node at (0,2.5) {...};
	\node at (0,-2.5) {...};
	\draw[very thick] (0,0) circle (4cm);	
\end{tikzpicture} \in \mathcal{P}_{n,\pm}
 \hspace{1cm} ab=
\begin{tikzpicture}[scale=.35, baseline]
\node at (150:4.3cm) {\tiny{$\star$}};
	\clip (0,0) circle (4cm);
\draw[shaded] (-0.8,6)--(-0.5,6)--(-0.5,-6)--(-0.8,-6)--(-0.8,6);
\draw[shaded] (0.8,6)--(0.5,6)--(0.5,-6)--(0.8,-6)--(0.8,6);
    \node at (0,3.2) {...};
	\node at (0,-3.2) {...};
	\node at (0,-1.5) [Tbox, inner sep=1.5mm] {{$a$} };
	\node at (0,1.5) [Tbox, inner sep=1.4mm] {{$b$} };
	\node at (-130:2cm) {\tiny{$\star$}};
	\node at (130:2cm) {\tiny{$\star$}};
	\draw[very thick] (0,0) circle (4cm);
\end{tikzpicture} $  \end{center}
\begin{center} 
$
tr_{r}(a)=\delta^{-n}
\begin{tikzpicture}[scale=.35, baseline]
\node at (150:4.3cm) {\tiny{$\star$}};
	\clip (0,0) circle (4cm);
    \draw[shaded] (0,0) circle (3.4cm);
	\draw[fill=white] (0,0) circle (3.1cm);
	\draw[shaded] (0,0) circle (2.3cm);
	\draw[fill=white] (0,0) circle (2cm);
	\node at (-2.6,0) [Tbox, inner sep=1.2mm] {$a$};
	\node at (180:3.65cm) {\tiny{$\star$}};
	\node at (2.7,0) {\tiny{...}};
	\draw[very thick] (0,0) circle (4cm);	
\end{tikzpicture} \hspace{1cm}
tr_{l}(a)=\delta^{-n}
\begin{tikzpicture}[scale=.35, baseline]
\node at (150:4.3cm) {\tiny{$\star$}};
	\clip (0,0) circle (4cm);
    \draw[shaded] (0,0) circle (3.4cm);
	\draw[fill=white] (0,0) circle (3.1cm);
	\draw[shaded] (0,0) circle (2.3cm);
	\draw[fill=white] (0,0) circle (2cm);
	\node at (2.7,0) [Tbox, inner sep=1.2mm] {$a$};
	\node at (180:-1.65cm) {\tiny{$\star$}};
	\node at (-2.7,0) {\tiny{...}};
	\draw[very thick] (0,0) circle (4cm);	
\end{tikzpicture}$
 \end{center}
%Popa95, Jones99, GuiJonShly07, Kodiyalam-Sunder09  
\begin{remark}[\cite{sk2} p229] \label{adjoint}
The tangle action deals with the adjoint by:  
$$Z_T(a_1 \otimes a_2 \otimes \cdots \otimes a_r)^{\star} = Z_{T^{\star}}(a_1^{\star} \otimes a_2^{\star} \otimes \cdots \otimes a_r^{\star})$$
with $T^{\star}$ the mirror image of $T$ and $a_i^{\star}$ the adjoint of $a_i$ in $\mathcal{P}_{n_i,\epsilon_i}$.
\end{remark}

A planar algebra $(\mathcal{P}_{n,\pm})$ is a subfactor planar algebra if and only if it is the standard invariant of an extremal subfactor $(N \subseteq M)$ with $\delta = [M:N]^{\frac{1}{2}}$ (see \cites{po2, jo4, gjs, sk3}). A finite depth or irreducible subfactor is extremal ($tr_{N'} = tr_{M}$ on $N' \cap M$).

\subsection{Basic ingredients of the $2$-box space}  \hspace*{1cm} \\
Let $(N \subseteq M)$ be a finite index irreducible subfactor. The $n$-box spaces $\mathcal{P}_{n,+}$ and $\mathcal{P}_{n,-}$ of the planar algebra $\mathcal{P}=\mathcal{P}(N \subseteq M)$, are isomorphic to $N' \cap M_{n-1}$ and $M' \cap M_{n}$ (as ${\rm C^{\star}}$-algebras).

\begin{remark}
The space $\mathcal{P}_{0,\pm} = \mathbb{C}$ because $N$ and $M$ are factors, $\mathcal{P}_{1,\pm} = \mathbb{C}$ because $(N \subseteq M)$ is irreducible, and finally $\mathcal{P}_{n,\pm}$ is finite dimensional because the index $[M:N] = \delta^2$ is finite.
\end{remark} 

\begin{definition}
Let $R(a)$ be the range projection of $a \in \mathcal{P}_{2,+}$. \\ We define the relations $a \prec b$  by $R(a) < R(b)$, $a \preceq b$  by $R(a) \le R(b)$, and $a \sim b$ by $R(a) = R(b)$.
\end{definition}    

\begin{definition}
Let $N \subseteq K \subseteq M$ be an intermediate subfactor. \\ Let the projection $e^M_K: L^2(M) \to L^2(K)$, let $id:=e^M_M$ and $e_1:=e^M_N$. 
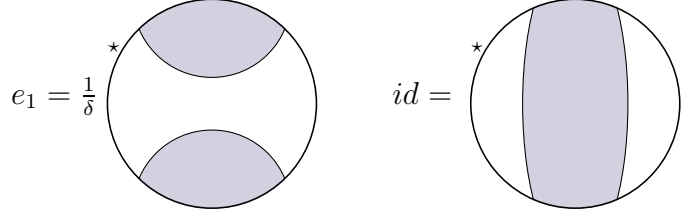
\begin{center} 
$e_1 = \frac{1}{\delta}
\begin{tikzpicture}[scale=.35, baseline]
\node at (150:4.3cm) {\tiny{$\star$}};
	\clip (0,0) circle (4cm);
  \draw[shaded] (0,4) circle (3cm);
	\draw[shaded] (0,-4) circle (3cm);
	\draw[very thick] (0,0) circle (4cm);	
\end{tikzpicture} 
 \hspace{1cm} id = 
\begin{tikzpicture}[scale=.35, baseline]
\node at (150:4.3cm) {\tiny{$\star$}};
	\clip (0,0) circle (4cm);
  \draw[shaded] (0,0) ellipse (2cm and 6cm);
	\draw[very thick] (0,0) circle (4cm);
\end{tikzpicture} $  \end{center}

\end{definition}  
Note that $tr(e_1) = \delta^{-2} = [M:N]^{-1}$ and $tr(id) = 1$.   

\begin{definition}
Let the bijective linear map $\mathcal{F}: \mathcal{P}_{2,\pm} \to  \mathcal{P}_{2,\mp}$ be the Ocneanu-Fourier transform, also called $1$-click (of the outer star) or $90^{\circ}$ rotation; and let $a * b$ be the coproduct of $a$ and $b$.
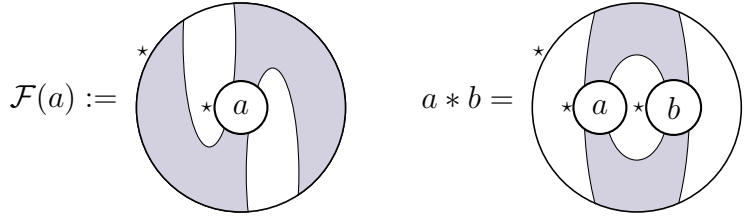
\begin{center} $ \mathcal{F}(a) := 
\begin{tikzpicture}[scale=.35, baseline]
\node at (150:4.3cm) {\tiny{$\star$}};
	\clip (0,0) circle (4cm);
	\draw[shaded] (0,0) circle (4cm);
  \draw[fill=white] (-1.2,2.5) ellipse (1cm and 4cm);
	\draw[fill=white] (1.2,-2.5) ellipse (1cm and 4cm);
	\node at (0,0) [Tbox, inner sep=1.5mm] {$a$};
	\node at (180:1.3cm) {\tiny{$\star$}};
	\draw[very thick] (0,0) circle (4cm);	
\end{tikzpicture} 
 \hspace{1cm} a * b = 
\begin{tikzpicture}[scale=.35, baseline]
\node at (150:4.3cm) {\tiny{$\star$}};
	\clip (0,0) circle (4cm);
  \draw[shaded] (0,0) ellipse (2cm and 6cm);
	\draw[fill=white] (0,0) ellipse (1.2cm and 2cm);
	\node at (-1.4,0) [Tbox, inner sep=1.5mm] {{$a$} };
	\node at (1.4,0) [Tbox, inner sep=1.4mm] {{$b$} };
	\node at (180:-0.07cm) {\tiny{$\star$}};
	\node at (180:2.65cm) {\tiny{$\star$}};
	\draw[very thick] (0,0) circle (4cm);
\end{tikzpicture} $  \end{center}
Let $\overline{a} := \mathcal{F}(\mathcal{F}(a))$ be the contragredient of $a$ (also called $180^{\circ}$ rotation).   
\end{definition} 
\begin{lemma} \label{Fourier}
Let $a \in \mathcal{P}_{2,+}$ then:
\begin{itemize}
\item $\overline{\overline{a}} = \mathcal{F}^{4}(a) = a$
\item $ \overline{\mathcal{F}(a)} = \mathcal{F}(\overline{a})  = \mathcal{F}^{-1}(a)$  
\item $ \mathcal{F}(a)^{\star} = \mathcal{F}^{-1}(a^{\star})$ and $ \mathcal{F}^{-1}(a)^{\star} = \mathcal{F}(a^{\star})$
\end{itemize} 
\end{lemma}
\begin{proof}
The map $\mathcal{F}^{4}$ corresponds to four $1$-clicks of the outer star, so it's the identity map, and so $\mathcal{F}^{3} = \mathcal{F}^{-1}$, but $\overline{\mathcal{F}(a)} = \mathcal{F}(\overline{a}) = \mathcal{F}^{3}(a)$.
The last follows by Remark \ref{adjoint}, because $\mathcal{F}^{-1}$ and  $\mathcal{F}$ are mirror images.
 \end{proof}
\begin{lemma} \label{costar}
Let $a,b \in \mathcal{P}_{2,+}$ then:
\begin{itemize}
\item $\overline{a \cdot b} = \overline{b} \cdot \overline{a}$  
\item $\overline{a * b} = \overline{b} * \overline{a}$
\item $(a \cdot b)^{\star} = b^{\star} \cdot a^{\star}$  
\item $(a * b)^{\star} = a^{\star} * b^{\star}$
\end{itemize}
\end{lemma}
\begin{proof}
Immediate by diagrams computation and Remark \ref{adjoint}.
 \end{proof}
 
% By Lemma \ref{2coprod}, $a * b = \mathcal{F}(\mathcal{F}^{-1}(a) \cdot \mathcal{F}^{-1}(b))$, so $$   (a * b)^{\star} = \mathcal{F}(\mathcal{F}^{-1}(a) \cdot \mathcal{F}^{-1}(b))^{\star} = \mathcal{F}^{-1}((\mathcal{F}^{-1}(a) \cdot \mathcal{F}^{-1}(b))^{\star}) $$ $$= \mathcal{F}^{-1}((\mathcal{F}^{-1}(b))^{\star} \cdot (\mathcal{F}^{-1}(a))^{\star}) = \mathcal{F}^{-1}(\mathcal{F}(b^{\star}) \cdot \mathcal{F}(a^{\star})) = a^{\star}  *  b^{\star} $$ Next, $$\overline{(a * b)} = \mathcal{F}^{-2}(\mathcal{F}(\mathcal{F}^{-1}(a) \cdot \mathcal{F}^{-1}(b))) = \mathcal{F}^{-1}(\mathcal{F}^{-1}(a) \cdot \mathcal{F}^{-1}(b))$$  $$= \mathcal{F}^{-1}(\mathcal{F}(\mathcal{F}^{-2}(a)) \cdot \mathcal{F}(\mathcal{F}^{-2}(b))) = \mathcal{F}^{-1}(\mathcal{F}(\overline{a}) \cdot \mathcal{F}(\overline{b})) = \overline{b} * \overline{a}$$   There is an alternative proof by diagram computations.
Note that
 $ \delta \mathcal{F}(e_1) = 
\begin{tikzpicture}[scale=.35, baseline]
\node at (150:4.3cm) {\tiny{$\star$}};
	\clip (0,0) circle (4cm);
	\draw[shaded] (0,0) circle (4cm);
  \draw[fill=white] (-1.2,2.5) ellipse (1cm and 4cm);
	\draw[fill=white] (1.2,-2.5) ellipse (1cm and 4cm);
	\draw[fill=white, dashed, ultra thin] (0,0) circle (1cm);
	\draw[shaded] (-0.27,1) .. controls ++(-100:1cm) and ++(-100:1cm) .. (0.65,0.9);
	\draw[shaded] (0.27,-1) .. controls ++(80:1cm) and ++(80:1cm) .. (-0.65,-0.9);
	\draw[very thick] (0,0) circle (4cm);	
\end{tikzpicture} 
 = 
\begin{tikzpicture}[scale=.35, baseline]
\node at (150:4.3cm) {\tiny{$\star$}};
	\clip (0,0) circle (4cm);
  \draw[fill=white] (0,0) ellipse (2cm and 9cm);
	\draw[shaded] (0,0) circle (4cm);
	 \draw[fill=white] (0,0) ellipse (2cm and 9cm);
	% \draw  (1,1)--(1,-1)--(-1,-1)--(-1,1)--(1,1);
	 \draw[very thick] (0,0) circle (4cm);
\end{tikzpicture} 
= id \in \mathcal{P}_{2,-}$ 

\begin{lemma}  \label{2coprod} The coproduct satisfies the following equality
$$ a * b = \mathcal{F}(\mathcal{F}^{-1}(a) \cdot \mathcal{F}^{-1}(b)) = \mathcal{F}^{-1}(\mathcal{F}(b) \cdot \mathcal{F}(a)).$$
It follows that  $\mathcal{F}(a * b) = \mathcal{F}(b) \cdot \mathcal{F}(a) $ and $\mathcal{F}(ab) = \mathcal{F}(a) * \mathcal{F}(b) $. 
\end{lemma}
\begin{proof} Immediate by diagrams computation.    \end{proof}

\begin{lemma} \label{constar} The contragredient is a linear $\star$-map.
\end{lemma}
\begin{proof} By Lemma \ref{Fourier}, $\mathcal{F}^4(a) = a$, so $\mathcal{F}^{2} = \mathcal{F}^{-2}$.  It follows that $$ \overline{a^{\star}} = \mathcal{F}^{2}(a^{\star}) = \mathcal{F}^{-2}(a)^{\star} = \mathcal{F}^{2}(a)^{\star} = \overline{a}^{\star} $$ The linearity comes from the linearity of $\mathcal{F}$.  \end{proof}

\begin{proposition}
There is the isomorphism of von Neumann algebra $$\mathcal{F}: (\mathcal{P}_{2,\pm},+,\cdot,(\cdot)^{\star}) \to (\mathcal{P}_{2,\mp},+, *,\overline{(\cdot)}^{\star})$$ 
\end{proposition}
\begin{proof}
Immediate by Lemmas \ref{Fourier}, \ref{costar}, \ref{2coprod}, \ref{constar}.
\end{proof}

\begin{remark}[\cite{kls} p39] In the irreducible finite index depth $2$ case, the contragredient is exactly the antipode of the Hopf ${\rm C}^*$-algebra. \end{remark}
\begin{lemma}
Let $a \in \mathcal{P}_{2,+}$ then $tr(\overline{a})=tr(a)$.
\end{lemma}
\begin{proof}
By  sphericality $tr(\overline{a})=tr_l(\overline{a})=tr_r(a)=tr(a)$.
\end{proof}
\begin{definition}
Let $ \tau = \begin{tikzpicture}[scale=.35, baseline]
\node at (150:4.3cm) {\tiny{$\star$}};
	\clip (0,0) circle (4cm);
  \draw[shaded] (3,0) ellipse (4cm and 12cm);
	\draw[fill=white] (0.9,0) ellipse (1.2cm and 2cm);
	\node at (-0.5,0) [Tbox, inner sep=2.5mm] {};
	\node at (180:1.8cm) {\tiny{$\star$}};
	\draw[very thick] (0,0) circle (4cm);
\end{tikzpicture}$ and $e_0 = \begin{tikzpicture}[scale=.35, baseline]
\node at (150:4.3cm) {\tiny{$\star$}};
	\clip (0,0) circle (4cm);
  \draw[shaded] (0,5)--(5,5)--(5,-5)--(0,-5)--(0,5);
  \draw[very thick] (0,0) circle (4cm);
\end{tikzpicture}$  
\end{definition}
In the literature, $\tau$ can be called the conditional expectation tangle.
\begin{lemma} \label{2>1}
Let $a \in \mathcal{P}_{2,+}$ then $\tau(a)= \delta tr(a) e_0$. 
\end{lemma}
\begin{proof} By irreducibility $\mathcal{P}_{1,+} = \mathbb{C}e_0$, so $\tau(a) = c e_0$ and $\delta^{2}tr(a) = \delta c$.
\end{proof}
\begin{lemma} \label{*id} Let $a \in \mathcal{P}_{2,+}$ then $a*e_1 = \delta^{-1}a$ and $a*id  = \delta tr(a) id$.
\end{lemma}
\begin{proof}
$a*e_1 = \delta^{-1} \begin{tikzpicture}[scale=.35, baseline]
\node at (150:4.3cm) {\tiny{$\star$}};
	\clip (0,0) circle (4cm);
  \draw[shaded] (0,0) ellipse (2cm and 6cm);
	\draw[fill=white] (0,0) ellipse (1.2cm and 2cm);
	\node at (-1.4,0) [Tbox, inner sep=1.5mm] {{$a$} };
	\draw[fill=white, dashed, ultra thin] (1.4,0) circle (1.05cm);
\draw[shaded] (1.05,1) .. controls ++(-78:1cm) and ++(-87:1cm) .. (1.98,1);
\draw[shaded] (1.05,-1) .. controls ++(78:1cm) and ++(87:1cm) .. (1.98,-1);
	\node at (180:2.65cm) {\tiny{$\star$}};
	\draw[very thick] (0,0) circle (4cm);
\end{tikzpicture} = \delta^{-1} a$ \text{,} 
$a*id = \begin{tikzpicture}[scale=.35, baseline]
\node at (150:4.3cm) {\tiny{$\star$}};
	\clip (0,0) circle (4cm);
  \draw[shaded] (0,0) ellipse (2cm and 6cm);
	\draw[fill=white] (0,0) ellipse (1.2cm and 2cm);
	\node at (-1.4,0) [Tbox, inner sep=1.5mm] {{$a$} };
	\draw[dashed, ultra thin] (1.4,0) circle (1.05cm);
	\node at (180:2.65cm) {\tiny{$\star$}};
	\draw[very thick] (0,0) circle (4cm);
\end{tikzpicture}$
Then by Lemma \ref{2>1}, $a*id = \delta tr(a)id$. 	
\end{proof}

 \begin{lemma} \label{tr(a*b)} \label{e1.} Let $a_1, a_2 \in \mathcal{P}_{2,+}$ then
 
 \begin{itemize}
 \item[(1)] $tr(a_1 * a_2) = \delta tr(a_1) tr(a_2)$
 \item[(2)] $e_1 \cdot (a_1 * \overline{a_2}^{\star}) = \delta \langle a_1 \vert a_2 \rangle e_1$
  \end{itemize}
 \end{lemma}
\begin{proof} (1) By Lemma \ref{*id}, $(a_1 * a_2)*id = \delta tr(a_1 * a_2) id$, but $$(a_1 * a_2)*id =a_1*(a_2 * id) =  \delta tr(a_2) a_1 * id = \delta^2 tr(a_2) tr(a_1)id$$  by associativity; it follows that $tr(a_1 * a_2) = \delta tr(a_1) tr(a_2)$.  
    
(2)  By diagram computations we get that $e_1 \cdot (a_1 * \overline{a_2}^{\star}) = e_1 \cdot ((a_2^{\star} \cdot  a_1) * id)$, and by Lemma \ref{*id},  $(a_2^{\star} \cdot a_1) * id = \delta tr(a_2^{\star} \cdot a_1) id $, but $tr(a_2^{\star} \cdot a_1) = \langle a_1 \vert a_2 \rangle$  and $e_1 \cdot id  = e_1$.
\end{proof}  

\begin{definition}
A projection $p \in \mathcal{P}_{2,+}$ is a
\begin{itemize}
\item minimal projection if for all projection $q \neq 0$, $q \cdot p = q \Rightarrow p=q$
\item minimal central projection if it is central and for all central projection $q \neq 0$, $q \cdot p = q \Rightarrow p=q$
\end{itemize}
\end{definition}

\begin{lemma} \label{overline}
If $p$ is a projection, then so is $\overline{p}$. Moreover if $p$ is a minimal (resp. minimal central) projection, then so is $\overline{p}$.
\end{lemma}
\begin{proof}
First, $\overline{p}^{\star}=\overline{p^{\star}} = \overline{p}$ and $\overline{p} \cdot \overline{p} = \overline{p \cdot p}=\overline{p}$. Next if $p$ is minimal, and if $q \cdot \overline{p} = q  $ (with $q \neq 0$ projection) then $$\overline{q} = (\overline{q \cdot \overline{p}})^{\star}= (p \cdot \overline{q})^{\star} = \overline{q} \cdot p$$ so $\overline{q}=p$ and $\overline{p}=q$. If $p$ is minimal central, then $\overline{p}$ is central because $p \cdot a=a \cdot p$ $\forall a$ if and only if $\overline{a} \cdot \overline{p}=\overline{p} \cdot \overline{a}$  $\forall a$  if and only if $a \cdot \overline{p}=\overline{p} \cdot a$  $\forall a$ (because the contragredient is bijective) if and only if $\overline{p}$ central, and it's also minimal central by the argument as above.
\end{proof}

\begin{theorem} \label{th} Let $a,b,c,d \in \mathcal{P}_{2,+}$ then

\begin{itemize}
\item[(1)] if $a,b >0$ then $a*b > 0$ 
\item[(2)] if $a, b, c, d > 0$,  $a \preceq b$ and $c \preceq d$ then $a*c \preceq b*d$
\end{itemize} 
\begin{proof} It's precisely Theorem 4.1 and Lemma 4.4 p18 of the paper \cite{li} of Z. Liu. The proof of (1) by  diagrams is the following:\\
$a * b = \ 
\begin{tikzpicture}[scale=.35, baseline]
\node at (150:4.3cm) {\tiny{$\star$}};
	\clip (0,0) circle (4cm);
  \draw[shaded] (0,0) ellipse (2cm and 6cm);
	\draw[fill=white] (0,0) ellipse (1.2cm and 2cm);
	\node at (-1.4,0) [Tbox, inner sep=1.5mm] {{$a$} };
	\node at (1.4,0) [Tbox, inner sep=1.4mm] {{$b$} };
	\node at (180:-0.05cm) {\tiny{$\star$}};
	\node at (180:2.65cm) {\tiny{$\star$}};
	\draw[very thick] (0,0) circle (4cm);
\end{tikzpicture}  \ = \ 
\begin{tikzpicture}[scale=.35, baseline]
\node at (150:4.3cm) {\tiny{$\star$}};
	\clip (0,0) circle (4cm);
  \draw[shaded] (0,0) ellipse (2cm and 6cm);
	\draw[fill=white] (0,0) ellipse (1cm and 3cm);
	\node at (-1.4,-1.2) [Tbox, inner sep=0.2mm] {\tiny{$a^{1/2}$} };
	\node at (1.4,-1.2) [Tbox, inner sep=0.2mm] {\tiny{$b^{1/2}$} };
	\node at (-1.4,1.4) [Tbox, inner sep=0.2mm] {\tiny{$a^{1/2}$} };
	\node at (1.4,1.4) [Tbox, inner sep=0.2mm] {\tiny{$b^{1/2}$} };
	\node at (82:1.4cm) {\tiny{$\star$}};
	\node at (-82:1.4cm) {\tiny{$\star$}};
	\node at (153:3cm) {\tiny{$\star$}};
	\node at (-153:3cm) {\tiny{$\star$}};
	\draw[very thick] (0,0) circle (4cm);
\end{tikzpicture} 
\ =X^{\star}X > 0$ \\     
 
For (2), by spectral theorem, for $\lambda > 0$ large enough,  $a< \lambda b$ and $c< \lambda d$;  
then by (1), $[\lambda b -a] * c > 0$  so  $a * c < \lambda b * c$; idem $b * [\lambda d -c] > 0$, so $b * c  < \lambda b * d$. It follows that $a * c < \lambda^2 b * d$, and so $a*c \preceq b*d$.
  \end{proof}
\end{theorem}

All the subfactor planar algebras here are irreducible and finite index.
\begin{definition}[\cite{li} Definition 2.7 p7]
A biprojection is a projection $b \in \mathcal{P}_{2,+}$ with $\mathcal{F}(b)$ a (positive) multiple of a projection. 
\end{definition}
By immediate diagrams computation, $e_1$ and $id$ are biprojections.
\begin{theorem}A projection $b \in \mathcal{P}_{2,+}$ is a biprojection if and only if $b*b \preceq b$.
\end{theorem} 
\begin{proof} 
Let $b$ be a biprojection then by Lemma \ref{2coprod}  $$b * b = \mathcal{F}^{-1}(\mathcal{F}(b) \cdot \mathcal{F}(b))=\mathcal{F}^{-1}(\lambda \mathcal{F}(b))= \lambda b \preceq b$$The converse is exactly Theorem 4.8 p19 of \cite{li}.
    \end{proof}
      
 \begin{remark}[\cite{la} p12] \label{biproj} A biprojection $b$ satisfies
$$e_1 \le b=b^2=b^{\star}=\overline{b} = \lambda b * b, (\lambda >0)$$ and the exchange relations: 
\begin{center}
$
\begin{tikzpicture}[scale=.35, baseline]
\node at (150:4.3cm) {\tiny{$\star$}};
	\clip (0,0) circle (4cm);
	\draw[shaded] (0,3.5) ellipse (0.5cm and 5cm);
  \draw[shaded] (0,-4.3) ellipse (2cm and 5cm);
	\draw[fill=white] (0,-3.5) ellipse (1.2cm and 3cm);
	\node at (-1.4,-1.4) [Tbox, inner sep=1.3mm] {{$b$} };
	\node at (0,1.4) [Tbox, inner sep=1.3mm] {{$b$} };
	\node at (130:2cm) {\tiny{$\star$}};
	\node at (205:3cm) {\tiny{$\star$}};
	\draw[very thick] (0,0) circle (4cm);
\end{tikzpicture}  \ = \
\begin{tikzpicture}[scale=.35, baseline]
\node at (150:4.3cm) {\tiny{$\star$}};
	\clip (0,0) circle (4cm);
	\draw[shaded] (0,3.5) ellipse (0.5cm and 5cm);
  \draw[shaded] (0,-4.3) ellipse (2cm and 5cm);
	\draw[fill=white] (0,-3.5) ellipse (1.2cm and 3cm);
	\node at (1.4,-1.4) [Tbox, inner sep=1.3mm] {{$b$} };
	\node at (0,1.4) [Tbox, inner sep=1.3mm] {{$b$} };
	\node at (130:2cm) {\tiny{$\star$}};
	\node at (-90:1.5cm) {\tiny{$\star$}};
	\draw[very thick] (0,0) circle (4cm);
\end{tikzpicture} 
 \ = \
\begin{tikzpicture}[scale=.35, baseline]
\node at (150:4.3cm) {\tiny{$\star$}};
	\clip (0,0) circle (4cm);
	\draw[shaded] (0,3.5) ellipse (0.5cm and 5cm);
  \draw[shaded] (0,-4.3) ellipse (2cm and 5cm);
	\draw[fill=white] (0,-3.5) ellipse (1.2cm and 3cm);
	\node at (1.4,-1.4) [Tbox, inner sep=1.3mm] {$b$};
	\node at (-1.4,-1.4) [Tbox, inner sep=1.3mm] {$b$};
	\node at (208:3cm) {\tiny{$\star$}};
	\node at (-86:1.5cm) {\tiny{$\star$}};
	 \fill[shaded] (0,0.65) circle (0.4cm);
	\draw[very thick] (0,0) circle (4cm);
\end{tikzpicture} 
$
\end{center}

\end{remark}
\begin{lemma} \label{prodcoprod} Let $a_1,a_2,b \in \mathcal{P}_{2,+}$ with $b$ a biprojection, then 
$$(b \cdot a_1 \cdot b) * (b \cdot a_2 \cdot b) = b \cdot (a_1 * (b \cdot a_2 \cdot b)) \cdot b = b \cdot ((b \cdot a_1  \cdot b) * a_2) \cdot b$$
$$(b*a_1*b) \cdot (b*a_2*b) = b*(a_1 \cdot (b*a_2*b))*b = b*((b*a_1*b) \cdot a_2)*b$$  
\end{lemma}
\begin{proof} 
By exchange relation on $b$ and $\mathcal{F}(b)$. %: \textcolor{red}{add diagrams?}
    \end{proof}

\begin{theorem}[\cite{bi} p212] \label{bisch}
An operator $b$ is a biprojection if and only if it is the Jones projection $e^M_K$ of an intermediate subfactor $N \subseteq K \subseteq M$. 
\end{theorem}

\begin{definition} \label{inter}
Let the intermediate subfactors $N \subseteq P \subseteq Q \subseteq M$ and let the projections $e^M_P: L^2(M) \to L^2(P)$ and $e^M_Q: L^2(M) \to L^2(Q)$ then $e^M_P \le e^M_Q$ and $e^M_P \cdot e^M_Q = e^M_P$. Let the index $[e^M_{Q}:e^M_{P}]$ be $[Q:P] = tr(e^M_{Q})/tr(e^M_{P})$, and the planar algebra $\mathcal{P}_{n,\pm}(e^M_{P} \le e^M_{Q})$ be $\mathcal{P}_{n,\pm}(P \subseteq Q)$. 
\end{definition}

%Next using $e^{N'}_{P'} \ge e^{N'}_{Q'}$  $$e^M_P * e^M_Q  = \mathcal{F}^{-1}(\mathcal{F}(e^M_Q) \mathcal{F}(e^M_P)) = \lambda \mathcal{F}^{-1}(\mathcal{F}(e^M_Q)) = \lambda e^M_Q$$ because $(M \subseteq M_1) \simeq (M' \subseteq N')$ (see \cite{js} p34), so $\mathcal{P}_{2,-}(N \subseteq M)$ is isomorphic to  $\mathcal{P}_{2,+}(M' \subseteq N')$ with $\mathcal{F}(e^M_P)$ mapping to $e^{N'}_{P'}$.
% $JM'J = M$ and $JN'J = M_1$

    \begin{lemma} \label{precobi}  Let $b$ be a biprojection and $p \le b$ a projection  then: 
\begin{itemize}
\item[(1)] $p*b = \delta tr(p)b$, and so $\mathcal{F}(b) \preceq \mathcal{F}(p)$
\item[(2)] $p*b^{\perp}  =  \delta tr(p)b^{\perp}$
\item[(3)] $  p^{\perp}*b = \delta tr(b-p)b +  \delta tr(b)b^{\perp}$
\item[(4)] $p^{\perp}*b^{\perp} = \delta [1-tr(b)]b +\delta [1-tr(p+b)] b^{\perp}$
\end{itemize}
 \end{lemma}
\begin{proof} We apply exchange relation on $[(p \cdot b) * b] \cdot b$ by two ways, we get on one hand $[(p \cdot b) * b] \cdot id = p * b$  and on the other hand   $[(p \cdot b) * id] \cdot b = \delta tr(p)b$, because  $p*id =  \delta tr(p)id$ by Lemma \ref{*id}. It follows that $p * b= \delta tr(p)b$. Now $ b^{\perp} = id-b$ so $p*b^{\perp} = p * id - p * b$, idem $p^{\perp}*b = id * b - p * b$, and finally, $p^{\perp}*b^{\perp} = id * id - p * b - p * b^{\perp} - p^{\perp}*b   $. 
\end{proof}  
    \begin{lemma} \label{cobi}  Let $b$ be a biprojection  then: 
\begin{itemize}
\item[(1)] $b*b = \delta tr(b)b$
\item[(2)] $b*b^{\perp} = b^{\perp}*b =  \delta tr(b)b^{\perp}$
\item[(3)] $b^{\perp}*b^{\perp} = \delta [1-tr(b)]b+\delta[1-2tr(b)]b^{\perp}$
\end{itemize}
 \end{lemma}
\begin{proof} Immediate by Lemma \ref{precobi} with $p = b$. 
\end{proof}
Note that $tr(b) = [id:b]^{-1} \le 1/2$  if $b < id$.   

\begin{remark} We observe that $b^{\perp}*b^{\perp} \sim \begin{cases} b \text{ if }  [id:b] = 2 \\ id \text{ if }  [id:b] > 2 \end{cases}$ \end{remark}
 
 \begin{definition}[T. Teruya \cite{teru}] \label{normal} A biprojection $b$ is normal if it is bicentral (i.e. $b$ and $\mathcal{F}(b)$ are central).
 \end{definition}
 
\textbf{Group-like structures on the $2$-box space}, i.e. like  $(R \subseteq R \rtimes G)$:
$$
\begin{array}{l|l}
\text{group } G  & \text{Jones projection } id \text{ of } \mathcal{P}_{2,+}  \\
\hline 
\text{element } g \in G  & \text{minimal projection } u \le id \\
\hline 
\text{composition } gh  & \text{coproduct } u * v  \\
\hline 
\text{neutral } eg = ge = g & \text{Jones projection } e_1 * u = u * e_1 \sim u \\
\hline 
\text{inverse } g^{-1}g=e & \text{contragredient }  \overline{u} * u \succeq e_1   \\
\hline 
\text{subgroup } H \subseteq G  & \text{(bi)projection $p$ with }    e_1 \le p \sim p*p \sim \overline{p}  
\\
\hline 
\text{normal subgroup } H \trianglelefteq G  & \text{normal biprojection }
\\
\hline 
\text{irreducible representation }  & \text{minimal central projection } p \in \mathcal{P}_{2,-}
\end{array} $$ 

The following lemma generalizes the existence of an inverse; there is uniqueness only in the minimal central case.
\begin{lemma}    \label{pre2} \label{2}  
Let $p \in \mathcal{P}_{2,+}$ be a projection, then $\overline{p}$ is also a projection, $e_1 \preceq p *\overline{p} $, and for $q$ a projection, $e_1 \preceq p *\overline{q} $ if and only if $pq \neq 0$ (so in the case $p,q$ minimal central projections, $p=q$). 
\end{lemma}
 \begin{proof}
 By Lemma \ref{overline}, $\overline{p}$ is a projection, next by Lemma \ref{e1.}, $e_1 \cdot (p *\overline{p} ) = \delta \langle p \vert p^{\star} \rangle e_1 = \delta \langle p \vert p \rangle e_1 >0$, so $e_1 \preceq p *\overline{p}$. Next $e_1 \cdot (p *\overline{q}) = \delta \langle p \vert q \rangle e_1$, but $$\langle p \vert q \rangle = \langle p^{\star}p \vert qq^{\star} \rangle = \langle pq \vert pq \rangle  \neq 0 \Leftrightarrow pq \neq 0$$ The result follows.\end{proof}    
 
 \begin{remark} If $u,v$ are minimal projections then $uv \neq 0$ if and only if $u$ and $v$ have the same central support and are not perpendicular. \end{remark}
 
The definition below generalizes the notion of subgroup generated.  
\begin{definition} \label{gener}
Let $a \in \mathcal{P}_{2,+}$ positive, and let $p_n$ be the range projection of $\sum_{k=1}^n a^{*k}$. By finiteness there exists $N$ such that for all $m \ge N$, $p_m = p_N$; and by \cite{li} Lemma 4.11 p20, $\langle a \rangle:=p_N$ is a biprojection (\textbf{generated} by $a$), the smallest biprojection $b \succeq a$. For $S$ a finite set of positive operators, let $\langle S \rangle$ be the smallest biprojection such that $b \succeq s \  \forall s\in S$. By Theorem \ref{th}, $\langle S \rangle = \langle \sum_{s \in S}s \rangle$. 
\end{definition}      

\begin{lemma} \label{equigene} Let $a, b \in \mathcal{P}_{2,+}$ with $a,b >0$ then
$$ a \sim b \Rightarrow  \langle a \rangle = \langle b \rangle$$
\end{lemma}
% spectral theorem on positive operator
\begin{proof} By spectral theorem, for $\lambda > 0$ large enough, $a < \lambda b$ and $b < \lambda a$, then $\langle a \rangle \le \langle b \rangle$ and $\langle a \rangle \ge \langle b \rangle$; the result follows.  \end{proof}

\begin{definition} Let $p \neq e_1$ be a biprojection, then a maximal sub-biprojection of $p$ is a biprojection $b < p$ such that $\forall b'$ biprojection $$b \le b' < p \Rightarrow b'=b$$ Let $p \neq id$ be a biprojection, then a minimal over-biprojection of $p$ is a biprojection $b > p$ such that $\forall b'$ biprojection $$b \ge b' > p \Rightarrow b'=b$$ A biprojection is called maximal (resp. minimal) if it is a maximal sub-biprojection of $id$ (resp. a minimal over-biprojection of $e_1$).
 \end{definition}  

\begin{theorem} \label{mini}
Let $p \in \mathcal{P}_{2,+}$ be a minimal central projection, then there exists $v \le p$ minimal projection such that $\langle v \rangle = \langle p \rangle$.
\end{theorem}
\begin{proof}
 If $p$ is a minimal projection, then it's ok. Else, let $b_1, \dots , b_n$ be the maximal sub-biprojections of $\langle p \rangle$ ($n$ is finite by \cite{wa}). If $p \not \preceq \sum_i b_i$ then there exists $v \le p $ minimal projection such that $ v  \not \le b_i \ \forall i$, so $\langle v \rangle = \langle p \rangle$ and the  results follows. Else $p \preceq \sum_{i=1}^n b_i$ (with $n>1$, otherwise $p \le b_1$ and $\langle p \rangle \le b_1$, contradiction); let $E_i=range(b_i)$ and $F=range(p)$, then $F=\sum_i E_i \cap F$ (because $p$ is a minimal central projection) with $1<n<\infty$ and $E_i \cap F \subsetneq F \ \forall i$ (otherwise $\exists i$ with $p \le b_i$, contradiction), so $\dim (E_i \cap F)< \dim (F)$ and there exists $V \subseteq F$ one-dimensional subspace  such that $V \not \subseteq E_i \cap F$  $\forall i$, and so a fortiori $V \not \subseteq E_i$   $\forall i$. It follows that $v=p_{V} \le p$  is a minimal projection such that $\langle v \rangle = \langle p \rangle$.
\end{proof}

The following lemma generalizing ``$ab=c \Rightarrow b=a^{-1}c$ and $a=cb^{-1}$" can be seen as a weak Frobenius reciprocity. It was inspired by \cite{sw}.
\begin{lemma} \label{pfr}
Let $a,b,c \in \mathcal{P}_{2,+}$ be projections with $c \preceq a*b$, then $\exists a'  \preceq c*\overline{b}$ and $\exists b'  \preceq \overline{a}*c $ such that $a'$, $b'$ are projections and $aa', bb' \neq 0$.
\end{lemma}
\begin{proof}
If $a,b,c$ are projections with $ c \preceq a*b$, by Lemma \ref{2} and Theorem \ref{th}, $e_1 \preceq c*\overline{c} \preceq  (a*b)* \overline{c}  = a*(b* \overline{c})$ by associativity, then by  Lemma \ref{2} again, $\exists a'$ projection with $a' \preceq \overline{b * \overline{c}} = c*\overline{b}$ and $aa' \ne 0$. Idem  $e_1 \preceq \overline{c} * (a*b)$, so $\exists b'$ projection with $bb' \ne 0$ and $ b' \preceq \overline{a}*c$.
\end{proof}   

\begin{remark} Lemma \ref{pfr} is optimal in general, because by Remark \ref{rkf}, the property ($F_0$) of Definition \ref{F} is not satisfied in general. \end{remark}

\subsubsection{Intermediate planar algebras and $2$-box spaces} \label{interpa} \hspace*{1cm} \\
Let $(N \subseteq M)$ be an irreducible finite index subfactor, and let $N \subseteq K \subseteq M$ be an intermediate subfactor. According to Z. Landau's PhD thesis \cite{la0}, the planar algebras $\mathcal{P}(N \subseteq K)$ and $\mathcal{P}(K \subseteq M)$ can be seen as sub-planar algebras of $\mathcal{P}(N \subseteq M)$, up to a renormalization of the partition function (see \cite{la0} 3. p98 and p105). Let the intermediate subfactors $N \subseteq P \subseteq K \subseteq Q \subseteq M$.
\begin{theorem}
The $2$-box spaces $\mathcal{P}_{2,+}(N \subseteq K)$ and $\mathcal{P}_{2,+} (K \subseteq M)$ are isomorphic to $e^M_K \mathcal{P}_{2,+}(N \subseteq M) e^M_K$ and $e^M_K * \mathcal{P}_{2,+}(N \subseteq M) * e^M_K$ respectively. 
\end{theorem}
\begin{proof} 
By \cite{li} Proposition 2.6 and Theorem 2.7 p8 (citing \cites{bjunp, la0}).
\end{proof} 
\begin{remark} These isomorphisms preserve $+$, $\times$ and $()^{\star}$; the coproduct $*$ is also preserved but up to renormalize it by $[M:K]^{1/2}$ for the first map and by $[K:N]^{-1/2}$ for the second. \end{remark}  
\begin{corollary} \label{landau1} \label{landau2}
There are $l_K$ and $r_K$, isomorphisms of  ${\rm C}^{\star}$-algebras $$l_K: \mathcal{P}_{2,+}(N \subseteq K) \to e^M_K \mathcal{P}_{2,+}(N \subseteq M) e^M_K$$ 
$$r_K : \mathcal{P}_{2,+} (K \subseteq M) \to e^M_K * \mathcal{P}_{2,+}(N \subseteq M) * e^M_K$$
with $l_K(e^K_P) = e^M_P$, $r_K(e^M_Q) = e^M_Q$ and $$\langle m(a_1), \dots, m(a_n) \rangle = m (\langle a_1, \dots, a_n \rangle)$$ for $m \in \{l_K, r_K, l^{-1}_K, r^{-1}_K \}$ and for any $a_i > 0$ in the domain of $m$.
\end{corollary}  
Note that these isomorphisms are for the usual product in both sides.
% \textcolor{red}{add the short no-extra proof and remove the big from appendix + ref}
\section{Cyclic subfactor planar algebras} \label{cycl}
In this section, we define the class of cyclic subfactor planar algebras, we show that it contains plenty of examples, and we prove that it is stable by dual, intermediate, free composition and tensor product ``generically''. Next we define the w-cyclic subfactor planar algebras.

Let $\mathcal{P}$ be a finite index irreducible subfactor  planar algebra. 
\begin{definition}
The planar algebra $\mathcal{P}$ is called
\begin{itemize}
\item distributive if the biprojections lattice is distributive.
\item Dedekind if all the biprojections are normal (Definition \ref{normal}).
\item cyclic if it is both Dedekind and distributive.
\end{itemize}
\end{definition} 

\begin{examples} A group subfactor is cyclic if and only if the group is cyclic; the maximal subfactors are cyclic, in particular  all the $2$-supertransitive subfactors, as the Haagerup subfactor \cite{asha,izha,ep}, are cyclic. Up to equivalence, exactly $23279$ among $34503$ inclusions of groups of index $< 30$, give a cyclic subfactor (more than $65\%$).
\end{examples}

\begin{definition} \label{equiv} Let $G$ be a finite group and $H$ be a subgroup. The core $H_G$ is the largest normal subgroup of $G$ contained in $H$. The subgroup $H$ is called core-free if $H_G = \{ 1 \}$; if so, the inclusion $(H \subseteq G)$ is also called core-free. Two inclusions of finite groups $(A \subseteq B)$ and $(C \subseteq D)$ are called equivalent if there is a group isomorphism $\phi: B/A_B \to D/C_D$ such that $\phi(A/A_B) = C/C_D$.   
\end{definition}

\begin{remark} A finite group subfactor remembers the group \cite{jo}, but a finite group-subgroup subfactor does not remember the (core-free) inclusion in general because (as proved by V.S. Sunder and V. Kodiyalam \cite{sk}) the inclusions $(\langle (1234) \rangle \subseteq S_4)$ and $(\langle (12)(34) \rangle \subseteq S_4)$ are not equivalent whereas their corresponding subfactors are isomorphic; but thanks to the complete characterization \cite{iz} by M. Izumi, it remembers the inclusion in the maximal case, because the intersection of a core-free maximal subgroup with an abelian normal subgroup is trivial\footnote{http://math.stackexchange.com/a/738780/84284}. \end{remark}

\begin{theorem} \label{th2}
 The free composition of irreducible finite index subfactors has no extra intermediate.
\end{theorem}
\begin{proof}
The theorem was first proved by Z. Liu (\cite{li} Theorem 2.11 p9) in the planar algebra framework. We found (independently) an other proof in the subfactor and bimodules framework, see appendix \ref{noextra}.  
\end{proof}

\begin{corollary} \label{corofree}
 The class of finite index irreducible cyclic subfactors is stable by free composition.
\end{corollary}
\begin{proof}
By Theorem \ref{th2}, the intermediate subfactors lattice of a free composition is the concatenation of the lattice of the two components, but by Lemma \ref{distri}, the distributivity is stable by concatenation. By Corollary \ref{landau1} and Lemma \ref{prodcoprod}, the biprojections remain normal.
\end{proof}

\begin{theorem} \label{th3}
Let $(N_i \subset M_i)$, $i=1,2$, be irreducible finite index subfactors. Then $$  \mathcal{L}(N_1  \subset M_1) \times \mathcal{L}(N_2 \subset  M_2) \subsetneq \mathcal{L}(N_1 \otimes N_2 \subset M_1 \otimes M_2)$$ if and only if there are intermediate subfactors $N_i \subseteq P_i \subset Q_i \subseteq M_i$,  $i=1,2$, such that $(P_i \subset Q_i)$ is depth $2$ and isomorphic to $(R^{\mathbb{A}_i} \subset R)$, with $\mathbb{A}_2 \simeq \mathbb{A}_1^{cop}$ which is the Kac algebra $\mathbb{A}_1$ with the opposite coproduct.
\end{theorem}
\begin{proof}
This theorem was proved  in the $2$-supertransitive case by Y. Watatani \cite{wa}. The general case was conjectured by the author, and proved by a discussion with Feng Xu as follows: \\  
Let the intermediate subfactors $$N_1 \otimes N_2 \subseteq P_1 \otimes P_2 \subset R \subset Q_1 \otimes Q_2 \subseteq M_1 \otimes M_2$$ with $R$ not of tensor product form, $P_1 \otimes P_2$ and $Q_1 \otimes Q_2$ the closest (below and above resp.) to $R$ among them of tensor product form. Now using \cite[Proposition 3.5 (2)]{xu}, $(P_i \subseteq Q_i)$, $i=1,2$,  are depth $2$, there corresponding Kac algebras, $\mathbb{A}_i$, $i=1,2$, are very simple and $\mathbb{A}_2 \simeq \mathbb{A}_1^{cop}$ \cite[Definition 3.6 and Proposition 3.10]{xu}. The converse is given by \cite[Theorem 3.14 (2)]{xu}.
\end{proof}

\begin{remark} \label{coroprod} By Theorem \ref{th3}, the class of (finite index irreducible) cyclic subfactors is stable by tensor product ``generically'' (i.e. if they don't have cop-isomorphic depth $2$ intermediate), because by Lemma \ref{distri}, the distributivity is stable by direct product, and by Corollary \ref{landau1} and Lemma \ref{prodcoprod}, the biprojections remain normal.  \end{remark}

\begin{lemma} \label{c*} If a subfactor is cyclic then the intermediate and dual subfactors are also cyclic.
\end{lemma}
\begin{proof} The result follows by Lemma \ref{distri} for distributive and Lemma  \ref{prodcoprod} for Dedekind. 
\end{proof}

Thanks to Theorem \ref{mini} we can give the following definition:  
\begin{definition}  \label{wcy}  The planar algebra $\mathcal{P}$ is weakly cyclic (or w-cyclic) if it satisfies one of the following equivalent assertion:  
\begin{itemize}
\item $\exists u \in  \mathcal{P}_{2,+}$  minimal projection  such that $\langle u \rangle=id$.
\item $\exists p \in  \mathcal{P}_{2,+}$  minimal central projection  such that $\langle p \rangle=id$.
\end{itemize}
Moreover, $(N \subseteq M)$ is called w-cyclic if its planar algebra is w-cyclic.
\end{definition} 

These remarks justify the choice of the words ``cyclic'' and ``w-cyclic''.
\begin{remark}
Let's call a ``group subfactor'', a subfactor of the form $(R^G \subseteq R)$ or $(R \subseteq R \rtimes G)$. Then the cyclic ``group subfactors" are exactly the ``cyclic group" subfactors.   \end{remark}
\begin{proof}
By Galois correspondence, a ``group subfactor" is cyclic if the subgroups lattice is distributive (the distributivity is invariant by taking the reversed lattice by Lemma \ref{distri}), if and only if the group is cyclic by Ore's Theorem \ref{ore1}. The normal biprojections of a group subfactor corresponds to the normal subgroups \cite{teru}, but all the subgroups of a cyclic group are normal.
\end{proof}  

\begin{remark} By Corollary \ref{wgrp2}, a finite group subfactor $(R^G \subset R)$ is w-cyclic if an only if $G$ is linearly primitive, which is strictly weaker than cyclic (see for example $S_3$), nevertheless the notion of w-cyclic is a singly generated notion in the sense that ``there is a minimal projection generating the identity biprojection''. We can also see the weakness of this assumption by the fact that the minimal projection does not necessarily generate a basis for the set of positive operators, but just the support of it, i.e. the identity. \end{remark}  

\begin{problem} Does cyclic implies w-cyclic for the planar algebra $\mathcal{P}$?
\end{problem} 

The answer is \textbf{yes} by the Theorem \ref{thm}. See the Section \ref{distex} for some extensions of this theorem.

\begin{problem} \label{cykac}
Are the depth $2$ irreducible finite index cyclic subfactor, exactly the cyclic group subfactors?
\end{problem}
The answer could be \textbf{no} because the following fusion ring (discovered by the author\footnote{http://mathoverflow.net/q/132866/34538}), the first known to be simple integral and non-trivial, \textit{could be} the Grothendieck ring of a maximal Kac algebra (see Definition \ref{maxkac}) of dimension $210$ and type $(1,5,5,5,6,7,7)$. 

\begin{scriptsize}
$$\begin{smallmatrix}
1 & 0 & 0 & 0& 0& 0& 0 \\
0 & 1 & 0 & 0& 0& 0& 0 \\
0 & 0 & 1 & 0& 0& 0& 0 \\
0 & 0 & 0 & 1& 0& 0& 0 \\
0 & 0 & 0 & 0& 1& 0& 0 \\
0 & 0 & 0 & 0& 0& 1& 0 \\
0 & 0 & 0 & 0& 0& 0& 1 
\end{smallmatrix}, \  
\begin{smallmatrix}
0 & 1 & 0 & 0& 0& 0& 0 \\
1 & 1 & 0 & 1& 0& 1& 1 \\
0 & 0 & 1 & 0& 1& 1& 1 \\
0 & 1 & 0 & 0& 1& 1& 1 \\
0 & 0 & 1 & 1& 1& 1& 1 \\
0 & 1 & 1 & 1& 1& 1& 1 \\
0 & 1 & 1 & 1& 1& 1& 1 
\end{smallmatrix}, \  
\begin{smallmatrix}
0 & 0 & 1 & 0& 0& 0& 0 \\
0 & 0 & 1 & 0& 1& 1& 1 \\
1 & 1 & 1 & 0& 0& 1& 1 \\
0 & 0 & 0 & 1& 1& 1& 1 \\
0 & 1 & 0 & 1& 1& 1& 1 \\
0 & 1 & 1 & 1& 1& 1& 1 \\
0 & 1 & 1 & 1& 1& 1& 1 
\end{smallmatrix}, \  
\begin{smallmatrix}
0 & 0 & 0 & 1& 0& 0& 0 \\
0 & 1 & 0 & 0& 1& 1& 1 \\
0 & 0 & 0 & 1& 1& 1& 1 \\
1 & 0 & 1 & 1& 0& 1& 1 \\
0 & 1 & 1 & 0& 1& 1& 1 \\
0 & 1 & 1 & 1& 1& 1& 1 \\
0 & 1 & 1 & 1& 1& 1& 1 
\end{smallmatrix}, \  
\begin{smallmatrix}
0 & 0 & 0 & 0& 1& 0& 0 \\
0 & 0 & 1 & 1& 1& 1& 1 \\
0 & 1 & 0 & 1& 1& 1& 1 \\
0 & 1 & 1 & 0& 1& 1& 1 \\
1 & 1 & 1 & 1& 1& 1& 1 \\
0 & 1 & 1 & 1& 1& 2& 1 \\
0 & 1 & 1 & 1& 1& 1& 2 
\end{smallmatrix}, \  
\begin{smallmatrix}
0 & 0 & 0 & 0& 0& 1& 0 \\
0 & 1 & 1 & 1& 1& 1& 1 \\
0 & 1 & 1 & 1& 1& 1& 1 \\
0 & 1 & 1 & 1& 1& 1& 1 \\
0 & 1 & 1 & 1& 1& 2& 1 \\
1 & 1 & 1 & 1& 2& 1& 2 \\
0 & 1 & 1 & 1& 1& 2& 2 
\end{smallmatrix}, \  
\begin{smallmatrix}
0 & 0 & 0 & 0& 0& 0& 1 \\
0 & 1 & 1 & 1& 1& 1& 1 \\
0 & 1 & 1 & 1& 1& 1& 1 \\
0 & 1 & 1 & 1& 1& 1& 1 \\
0 & 1 & 1 & 1& 1& 1& 2 \\
0 & 1 & 1 & 1& 1& 2& 2 \\
1 & 1 & 1 & 1& 2& 2& 1
\end{smallmatrix}$$  
\end{scriptsize}

\section{Ore's theorem for cyclic subfactor planar algebras} 
Let $\mathcal{P}$ be an irreducible finite index subfactor planar algebra. 

\begin{definition} \label{deflm}
A biprojection $b \in \mathcal{P}_{2,+}$ is lw-cyclic (resp. rw-cyclic) if $\exists u \in \mathcal{P}_{2,+}$ minimal projection such that $\langle u \rangle = b$ (resp. $\langle u,b \rangle = id$). Moreover it is called lrw-cyclic if it is both lw-cyclic and rw-cyclic.
\end{definition}
Let $(N \subseteq M)$  be an irreducible finite index subfactor, and let the intermediate subfactor $N \subseteq K \subseteq  M$.
\begin{theorem} \label{thmlm}
The biprojection $e^M_K \in \mathcal{P}_{2,+}(N \subseteq M)$ is lw-cyclic (resp. rw-cyclic) if and only if $(N \subseteq K)$ (resp. $(K \subseteq M)$) is w-cyclic.
\end{theorem}
\begin{proof}
Suppose that $(N \subseteq K)$ is w-cyclic, then there exists $a \in \mathcal{P}_{2,+}(N \subseteq K)$ minimal projection such that $\langle a \rangle =  e^K_K$, but by Corollary \ref{landau1}, $\langle l_K(a) \rangle=e^M_K$ and $u=l_K(a)$ is a minimal projection on $\mathcal{P}_{2,+}(N \subseteq M)$ because $l_K$ is an isomorphism of ${\rm C}^{\star}$-algebra.  

Conversely, if $e_K^M$ is lw-cyclic, then there exists $u \in \mathcal{P}_{2,+}$ minimal projection such that $\langle u \rangle = e_K^M$, so $$e_K^K = l_K^{-1}(e_K^M) = l_K^{-1}(\langle u \rangle) = \langle l_K^{-1}(u) \rangle$$ and $a = l_K^{-1}(u)$ is a minimal projection. \\

 Now suppose that $(K \subseteq M)$ is w-cyclic, then there exists $u \in \mathcal{P}_{2,+}(K \subseteq M)$ minimal projection such that $\langle u \rangle =id=  e^M_M$, but by Corollary \ref{landau2}, $\langle r_K(u) \rangle = e^M_M=id$ and  $r_K(u)$ is a minimal projection of $e^M_K * \mathcal{P}_{2,+}(N \subseteq M) * e^M_K $, so there exists $c \in \mathcal{P}_{2,+}(N \subseteq M)$ such that $r_K(u) = e^M_K * c * e^M_K$, and by Lemma \ref{cobi}, 
 $$e^M_K * r_K(u) * e^M_K = (e^M_K * e^M_K) * c * (e^M_K * e^M_K) =   [\delta tr(e^M_K)]^2 r_K(u)$$ so we can take $c= \frac{[M:K]}{[K:N]}r_K(u)$. Note that a positive operator on $e^M_K * \mathcal{P}_{2,+}(N \subseteq M) * e^M_K$ is of the form $[e^M_K * x * e^M_K][e^M_K * x * e^M_K]^{\star}$, and so it is also positive on $\mathcal{P}_{2,+}(N \subseteq M)$; it follows that $c= \frac{[M:K]}{[K:N]}r_K(u)$ is positive on $\mathcal{P}_{2,+}(N \subseteq M)$.   But for any $v \preceq c$ with $v$ minimal projection on $\mathcal{P}_{2,+}(N \subseteq M)$, by minimality and Theorem \ref{th}, $e^M_K * c * e^M_K \sim e^M_K * v * e^M_K$, and so by Lemma \ref{equigene}, $\langle e^M_K * v * e^M_K \rangle = \langle r_K(u) \rangle = id$. Finally, $\langle e^M_K * v * e^M_K \rangle =  \langle e^M_K , v \rangle$ because first $\langle e^M_K * v * e^M_K \rangle \le \langle e^M_K , v \rangle$, and next $e_1=e^M_N \le e^M_K$ so $v \preceq e^M_K * v * e^M_K$, moreover by Remark \ref{biproj} $\overline{v} \preceq \langle e^M_K * v * e^M_K \rangle$, but $$\overline{v} * e^M_K * v * e^M_K \succeq \overline{v} * e_1 * v * e^M_K \sim \overline{v} * v * e^M_K \succeq e_1 * e^M_K \sim e^M_K$$ by Lemma \ref{pre2}; conclusion   $v, e^M_K  \le  \langle e^M_K * v * e^M_K \rangle$, so we also have $\langle v, e^M_K \rangle \le \langle e^M_K * v * e^M_K \rangle$.  

 Conversely, if $e_K^M$ is rw-cyclic, then there exists $v \in \mathcal{P}_{2,+}$ minimal projection such that $\langle v, e_K^M \rangle = e_M^M$, so $$e_M^M = r_K^{-1}(e_M^M) = r_K^{-1}(\langle v,e_K^M \rangle) = \langle r_K^{-1}(v),r_K^{-1}(e_K^M) \rangle$$ but $r_K^{-1}(e_K^M) =  e_K^M$ is the trivial biprojection of $\mathcal{P}_{2,+}(K \subseteq M)$, so $e_M^M = \langle u \rangle $ with $u=r_K^{-1}(v)$ a minimal projection.  \end{proof}
  
  \begin{definition} \label{topdef}
The top intermediate subfactor planar algebra is the intermediate corresponding to the top interval (see Definition \ref{top}) of the biprojections lattice.
\end{definition}
 
  \begin{definition}
An irreducible finite index subfactor planar algebra is top w-cyclic if its top intermediate subfactor planar algebra is w-cyclic.
\end{definition}

\begin{lemma} \label{Topw} Top w-cyclic implies w-cyclic.
\end{lemma}
\begin{proof}
Let $b_1, \dots , b_n$ be the maximal biprojections, by assumption and Theorem \ref{thmlm}, $b = \bigwedge_i b_i $ is rw-cyclic, i.e.  there is a minimal projection $c$ with $\langle b,c \rangle = id$, but if $\exists i $ such that $c \le b_i$ then $\langle b,c \rangle \le b_i$, contradiction, so $\forall i $, $c \not \le b_i$ and then $\langle c \rangle = id$. 
\end{proof}

\begin{definition}
The height $h(\mathcal{P})$ is the maximal $l$ for an ordered chain  $$e_1 < b_1 < \cdots < b_{l}=id$$ of biprojections. Note that $h(\mathcal{P})<\infty$ because the index is finite.
\end{definition}    

\begin{theorem} \label{centheo}
If the biprojections of $\mathcal{P}_{2,+}$ are central and form a distributive lattice then $\mathcal{P}$ is w-cyclic.
\end{theorem}
\begin{proof} 
By Lemma \ref{c*} we can apply an induction on $h(\mathcal{P})$.  If $h(\mathcal{P})=1$ then the subfactor is maximal, so $\langle u \rangle = id$ $\forall u \in \mathcal{P}_{2,+}$ minimal projection with $u \neq e_1$. Now suppose it's true for $h(\mathcal{P})<n$, we will prove it for $h(\mathcal{P})=n \ge 2 $. By  Lemmas \ref{topBn} and \ref{Topw} we can assume the biprojections lattice to be boolean.  \\

Let $b$ be a biprojection with $e_1 < b < id$, and $b^{\complement}$ its complementary (i.e. $b \wedge b^{\complement} = e_1$ and $b \vee b^{\complement} = id$, see Definition \ref{comp}), then $e_1 < b^{\complement} < id$, so by induction and Theorem \ref{thmlm}, $b$ and $b^{\complement}$ are lw-cyclic, i.e. there are minimal projections $u, v$  such that $b = \langle u \rangle$ and  $b^{\complement} = \langle v \rangle$. Take any minimal projection $c \preceq u * v$ , then $\langle c \rangle = \langle c \rangle \vee e_1 = \langle c \rangle \vee (b \wedge b^{\complement})$  but \textit{by distributivity} $$\langle c \rangle = \langle c \rangle \vee (\langle u \rangle \wedge \langle v \rangle) = (\langle c \rangle \vee \langle u \rangle) \wedge (\langle c \rangle \vee \langle v \rangle) = \langle c,u \rangle \wedge \langle c , v \rangle$$
So by Lemma \ref{pfr}, $\langle c \rangle = \langle u',c,v \rangle \wedge \langle u,c,v' \rangle$ with $u', v'$  minimal projections and $u  u', v  v' \neq 0$, so in particular the central support $Z(u') = Z(u)$ and $Z(v') = Z(v)$. Now by assumption all the biprojections are central, so $u \le Z(u') \le \langle u',c,v \rangle$ and $v \le Z(v')  \le \langle u,c,v' \rangle$, then $\langle c \rangle =id$.  \end{proof} 

The following is the main theorem of the paper: 
\
\begin{theorem} \label{thm}
If $\mathcal{P}$ is cyclic then it is w-cyclic.
\end{theorem}
\begin{proof} Immediate by Theorem \ref{centheo} because a normal biprojection is  by definition bicentral, so a fortiori central.
\end{proof}
Recall that the converse is not true because the group $S_3$ is linearly primitive but not cyclic.
\begin{problem} What's the natural additional assumption (A) such that $\mathcal{P}$ is cyclic if and only if it is w-cyclic and (A)?
\end{problem}

\begin{definition}
The planar algebra $\mathcal{P}$ is called w$^{\star}$-cyclic if all the biprojections of $\mathcal{P}_{2,\pm}$ are lrw-cyclic (i.e. lw-cyclic and rw-cyclic).
\end{definition}

\begin{corollary} Cyclic implies w$^{\star}$-cyclic.
\end{corollary}
\begin{proof} Immediate by Theorem \ref{thm} and Lemma \ref{c*}.
\end{proof}

\begin{remark} The converse is also not true, because as first observed by Z. Liu,  $(R^{S_4} \subset R^{S_2})$ is w$^{\star}$-cyclic but not cyclic.
\end{remark}

Note that about the depth $2$ case, $(R^{S_3} \subset R)$ is w-cyclic, not cyclic, and also not w$^{\star}$-cyclic because its dual $(R \subset R \rtimes S_3)$ is not w-cyclic, because $S_3$ is not cyclic (see Corollary \ref{wgrp}).
\begin{question} Is cyclic equivalent to w$^{\star}$-cyclic for the depth $2$ case?
\end{question}

\begin{definition} A subfactor planar algebra is called $w_{+}$-cyclic if the maximal biprojections $b_1, \dots , b_n$  satisfy $(\sum_i b_i)^{\perp} \neq 0$. \end{definition}

\begin{lemma} \label{w+w} The property $w_{+}$-cyclic implies $w$-cyclic.
\end{lemma}
\begin{proof}
Let $u \le (\sum_i b_i)^{\perp}$ minimal projection, then $\langle u \rangle = id$ because $u \not \le b_i \ \forall i$.
\end{proof}

The converse is false  because $(R^{S_3} \subset R)$ is w-cyclic but not $w_{+}$-cyclic.

\begin{proposition} For a Dedekind subfactor planar algebra $w_+$-cyclic is equivalent to w-cyclic.
\end{proposition}
\begin{proof} The first implication is true in general by Lemma \ref{w+w}.  
For the other implication, just observe that ``w-cyclic'' is equivalent to ``$\exists c$ minimal projection such that $\forall i \ c \not \preceq b_i$'', then by Dedekind assumption we get $\forall i \ Z(c)b_i = 0$, so $Z(c) \preceq (\sum_i b_i)^{\perp}$.
\end{proof}

It follows that cyclic implies $w_{+}$-cyclic, but the converse is also false. Let $Q$ be the quaternion group, it is linearly primitive and Dedekind, so $(R^Q \subset R)$ is w-cyclic Dedekind and so $w_{+}$-cyclic, but not cyclic; nevertheless it is also not $w^{*}$-cyclic, and $(R^{S_4} \subset R^{S_2})$ is not Dedekind. 

\begin{question}
For a Dedekind subfactor planar algebra, is $w^*$-cyclic equivalent to distributive?
\end{question}

\section{Extensions to distributive subfactor planar algebras}  \label{distex}
In this section, we describe several extensions of Theorem \ref{thm}.
\subsection{The boolean subfactor planar algebras} \label{B_n} \hspace*{1cm} \\
The subfactor planar algebras are supposed irreducible and finite index.\\
By Lemma \ref{topBn}, the top lattice of a distributive lattice is boolean.

\begin{definition}
A subfactor planar algebra is called boolean (resp. $\mathcal{B}_n$) if its biprojections lattice is boolean (resp. $\mathcal{B}_n$).
\end{definition}

\begin{lemma} \label{maximal} A maximal subfactor planar algebra (i.e. $\mathcal{B}_1)$ is w-cyclic
\end{lemma}
\begin{proof}
By maximality $\langle u \rangle = id$ for any minimal projection $u \neq e_1$.
\end{proof}

\begin{proposition} \label{l<} A subfactor planar algebra with the maximal biprojections $b_1, \dots, b_n$  satisfying $\sum_i \frac{1}{[id:b_i]} \le 1$, is w-cyclic.
 \end{proposition}
 \begin{proof} First, by Lemma \ref{Topw}  and Lemma \ref{maximal}, we can assume $n>1$. \\  By Definition $[id : b_i] = \frac{tr(id)}{tr(b_i)}$ so by assumption $ \sum_i tr(b_i) \le tr(id)$. \\ If $\sum_i b_i \sim id$ then $\sum_i b_i \ge id$  but $ \sum_i tr(b_i) \le tr(id)$ so  $ \sum_i b_i = id$. Now $\forall i \ e_1 \le b_i$, so $ne_1 \le \sum_i b_i = id$, contradiction with $n>1$. \\ So $\sum_i b_i \prec  id$, which implies the existence of a minimal projection $u \not \le b_i$ $\forall i $, which means that $\langle u \rangle = id$.   
\end{proof}

\begin{remark} The converse of Proposition \ref{l<} is not true, $(R \subset R \rtimes \mathbb{Z}/30)$ is a counter-example, because $1/2+1/3+1/5 = 31/30 >1$. \end{remark}
\begin{corollary} \label{two} If $\mathcal{P}$ admits at most two maximal biprojections then it is  w-cyclic.
 \end{corollary}
\begin{proof} $\sum_i \frac{1}{ [id : b_{i}] } \le 1/2+1/2$, the result follows by Proposition \ref{l<}. \end{proof}
\begin{examples} All the $\mathcal{B}_2$ subfactor planar algebras are w-cyclic. 
$$
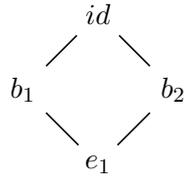
\begin{tikzpicture}
\node (A1) at (0,0) {\small $id$};
\node (A2) at (-1,-1) {\small $b_1$};
\node (A4) at (1,-1) {\small  $b_2$};
\node (A5) at (0,-2) {\small $ e_1$};
\tikzstyle{segm}=[-,>=latex, semithick]
\draw [segm] (A1)to(A2); \draw [segm] (A1)to(A4);
\draw [segm] (A2)to(A5);\draw [segm] (A4)to(A5);
\end{tikzpicture}$$ 
  \end{examples}

 \begin{lemma} \label{wmin} 
 Let $u,v \in \mathcal{P}_{2,+}$ be minimal projections. If $v \not \le \langle u \rangle$ then $\exists c \preceq u * v$ and $\exists w \preceq \overline{u} * c$ minimal projections such that $w \not \le \langle u \rangle$.
  \end{lemma}
\begin{proof} Assume that $\forall c \preceq u * v$ and $\forall w \preceq \overline{u} * c$ we have  $w \le \langle u \rangle$. Now there are minimal projections $(c_i)_i$ and $(w_{i,j})_{i,j}$ such that $u * v \sim \sum_i c_i$ and $\overline{u} * c_i \sim \sum_j w_{i , j}$. It follows that $u * v \sim \sum_{i,j} w_{i,j} \preceq \langle u \rangle$, but 
$$v \sim e_1 * v \preceq (\overline{u} * u)*v = \overline{u} * (u*v) \preceq  \langle u \rangle$$
which is in contradiction with $v \not \le \langle u \rangle$. \end{proof} 
  
In the distributive case, we can upgrade Proposition \ref{l<} as follows:

\begin{theorem} \label{le2}
A distributive subfactor planar algebra with the maximal biprojections $b_1, \dots, b_n$  satisfying $\sum_i \frac{1}{[id:b_i]} \le 2$, is w-cyclic.
\end{theorem}
\begin{proof}

By Lemma \ref{Topw} and Lemma \ref{topBn}, we can assume the subfactor planar algebra to be boolean.

If $K:=\bigwedge_{i,j, i \neq j} (b_i \wedge b_j)^{\perp} \neq 0$, then let $u \le K$ a minimal projection. If  $\langle Z(u) \rangle = id$ then ok, else $\exists i$  such that $\langle u \rangle = \langle Z(u) \rangle = b_i$. Now  $b_i^{\complement}$ (see Definition \ref{comp})  is a minimal biprojection, so $b_i^{\complement} = \langle v \rangle$ with $v$ minimal projection. Now $b_i \wedge b_i^{\complement} = e_1$ so $v \not \le \langle u \rangle$, so by Lemma \ref{wmin} there are minimal projections $c \preceq u * v$ and $w \preceq \overline{u} * c$ such that $w \not \preceq \langle u \rangle$ (and $\langle u,w \rangle = id$ by maximality). Now by Lemma \ref{pfr}, $\exists u' \preceq c * \overline{u}$ with $Z(u') = Z(u)$ and $u' \not \perp u$, but $u \le K$ so $\forall j \neq i$, $u' \not \le b_i \wedge b_j$, now $u' \le Z(u) \le b_i$,  so $\langle u' \rangle = b_i$. By the distributivity argument (as for Theorem \ref{centheo}) we conclude as follows:  $$\langle c \rangle =  \langle u,c \rangle \wedge \langle c,v \rangle \ge \langle u,w \rangle \wedge \langle u',v \rangle =  id \wedge id = id$$

Else $K = 0$, but $\forall i$, $(b_i \wedge b_j)^{\perp} \ge b_j^{\perp}$, so $\forall i$, $\bigwedge_{j \neq i} b_j^{\perp} = 0$. Let $p_1, \dots ,  p_r$ be the minimal central projections, then $b_i = \bigoplus_{s=1}^{r} p_{i,s}$ with $p_{i,s} \le p_s $ and $p_{i,1} = p_1 = e_1$. Now $b_i^{\perp} = \bigoplus_{s=1}^{r} (p_s-p_{i,s})$, so by assumption, $$ 0 = \bigwedge_{j \neq i} \bigoplus_{s=1}^{r} (p_s-p_{j,s}) = \bigoplus_{s=1}^{r} \bigwedge_{j \neq i} (p_s-p_{j,s}), \ \forall i $$ It follows that  $\forall i \forall s$, $p_s = \bigvee_{j \neq i} p_{j,s} $, so $ tr(p_s)  \le \sum_{j \neq i} tr(p_{j,s})$.  Now if $\exists s  \forall i \ p_{i,s} < p_s$, then $\langle p_s \rangle = id$, ok; else $\forall s   \exists i $ with $ \ p_{i,s} = p_s$, but $\sum_{j \neq i} tr(p_{j,s}) \ge tr(p_s)$, so $\sum_{j} tr(p_{j,s}) \ge 2tr(p_s)$. It follows that $$\sum_i tr(b_i) \ge  n \cdot tr(e_1) + 2 \sum_{s \neq 1} tr(p_s) = 2tr(id) + (n-2)tr(e_1)$$  
Now $[id:b_i] = tr(id)/tr(b_i)$ so $$\sum_i \frac{1}{[id:b_i]} \ge  2 + \frac{n-2}{[id:e_1]}$$
which contradicts the assumption, because we can assume $n>2$ by Corollary \ref{two}. The result follows.
\end{proof}

\begin{remark} The converse is not true because there exists w-cyclic distributive subfactor planar algebras with $\sum_i \frac{1}{[id:b_i]}>2$, for example, the $(S^n_2 \subset S^n_3)$ subfactor is w-cyclic and boolean but $\sum_i \frac{1}{[id:b_i]} = n/3$.
\end{remark}

We can get applications of Theorem \ref{le2} as in Section \ref{appli}, for example:

\begin{corollary}
Let $[H,G]$ be a distributive interval of finite groups. Let $H_1, \dots, H_n$ be the minimal overgroups of $H$. If $\sum_i \frac{1}{\vert H_i : H \vert} \le 2$, then there exists an irreducible complex representation $V$ of $G$ such that $G_{(V^H)} = H$ (see Definition \ref{fixstab}).
\end{corollary}
\begin{proof} As for the proof of Corollary \ref{dualore2}, but using Theorem \ref{le2}.
\end{proof}

\begin{corollary} \label{n/2}
Every $\mathcal{B}_n$ subfactor planar algebra with $[id:b] \ge n/2$ for every maximal biprojection $b$, is w-cyclic.
\end{corollary}
\begin{proof}
By assumption (following the notations of Theorem \ref{le2}) $$\sum_i \frac{1}{[id:b_i]} \le \sum_i \frac{2}{n} = 2$$ and the result follows.
\end{proof}

\begin{corollary} \label{B4w}
For $n \le 4$, every $\mathcal{B}_n$ subfactor planar algebra is w-cyclic.
\end{corollary}
\begin{proof}
Immediate by Corollary \ref{n/2} because in general $[id:b] \ge 2$, so $n \le 4$ is working.
\end{proof}

\begin{corollary}  \label{32}
Any distributive subfactor planar algebra with $<32$ biprojections or index $<32$, is w-cyclic. 
\end{corollary}
\begin{proof}
In this case, the top lattice is $\mathcal{B}_n$ with $n < 5$ because $32 = 2^5$, so we apply Corollary \ref{B4w} on the top intermediate, then Lemma \ref{Topw}.
\end{proof}

\begin{conjecture} \label{conjext}
A distributive subfactor planar algebra is w-cyclic.
\end{conjecture}
Note that by Lemmas \ref{topBn} and \ref{Topw} we can reduce to the boolean case. Moreover, we can upgrade the conjecture by replacing distributive by top boolean (i.e. with a boolean top interval, which is more general).   

 Assuming Conjecture \ref{conjext}, we get the following statement\footnote{http://mathoverflow.net/q/192046/34538}:
\begin{statement} \label{statgroup}
Let $[H,G]$ be a distributive interval of finite groups. Then there is an irreducible complex representation $V$ of $G$ such that $G_{(V^H)} = H$; if moreover $H$ is core-free, then $G$ is linearly primitive (see Definitions \ref{linprimdef} and \ref{fixstab}).
\end{statement}
\begin{proof} By Corollary \ref{wgrp}, the first part is just a group theoretic reformulation of Conjecture \ref{conjext}. So there is an irreducible complex representation $V$ with $G_{(V^{H})} = H$. Now, $V^H \subset V$ so $G_{(V)} \subset G_{(V^H)}$, but $ker(\pi_V) =  G_{(V)}$, it follows that $ker(\pi_V) \subset H$; but $H$ is a core-free subgroup of $G$, and $ker(\pi_V)$ a normal subgroup of $G$, so $ker(\pi_V)= \{ e \}$, which means that $V$ is faithful on $G$, i.e. $G$ is linearly primitive.
\end{proof}
\begin{remark} By Corollary \ref{32}, Statement \ref{statgroup} is true without assuming Conjecture \ref{conjext}, if  $\left\vert{[H,G]}\right\vert $ or $[G:H]<32$.  \end{remark}
 
In the same way, we upgrade the statement by replacing distributive by bottom boolean. We can ask if we get an equivalence:
\begin{question}
Is a finite group $G$ linealry primitive iff it admits a core-free subgroup $H$ such that the lattice $[H,F]$ is boolean, with $F$ the group generated by the minimal overgroups of $H$ in $G$? 
\end{question}

\subsection{The properties ($F_1$) and ($\tilde{Z}$)}

\begin{definition} \label{F0} \label{F}
The planar algebra $\mathcal{P}$ satisfies:
\begin{itemize}
\item ($F_0$) if for all minimal projections $a,b \in \mathcal{P}_{2,+}$, there exists a minimal projection $c \preceq a*b$ such that  $ a  \preceq c*\overline{b}  \text{ and }   b \preceq \overline{a}*c$.
\item ($F_1$) if for all minimal projections $a,b \in \mathcal{P}_{2,+}$, there exists a minimal projection $c \in \mathcal{P}_{2,+} $ such that $\langle a,c \rangle \text{ and } \langle c,b \rangle \ge \langle a,b \rangle$
\item ($F_2$)  if for all $p,q$ minimal central projections, there exists a minimal central projection $r$, such that $\langle p,r \rangle \text{ and } \langle r,q \rangle \ge \langle p,q \rangle$.
\end{itemize}
 
\end{definition} 

\begin{proposition} \label{AbF0}
If $\mathcal{P}_{2,+}$ is abelian, then $\mathcal{P}$ satisfies ($F_0$).
\end{proposition}
\begin{proof}
Because $\mathcal{P}_{2,+}$ is abelian, if $a$ and $a'$ are minimal projections with $aa' \neq 0$ then $a=a'$; the result follows by Lemma \ref{pfr}. 
\end{proof}  

\begin{proposition} \label{F0F} \label{FF'}
The property ($F_0$) implies ($F_1$) which implies ($F_2$).
\end{proposition}
\begin{proof}
 Assuming ($F_0$), let $c \preceq a*b$ such that  $ a  \preceq c*\overline{b}$ and $ b \preceq \overline{a}*c$, then $a,b \le \langle c,b \rangle$ and $\langle a,c \rangle$, so it's ($F_1$).  

Now assume ($F_1$) and let $p,q$ be minimal central projections. By Theorem \ref{mini}, there are $a,b$ minimal projections such that $\langle p \rangle = \langle a \rangle$ and  $\langle q \rangle = \langle b \rangle$, then $$\langle p,q \rangle = \langle \langle p \rangle,\langle q \rangle \rangle = \langle \langle a \rangle,\langle b \rangle \rangle = \langle a, b \rangle$$ So there is $c$ checking ($F_1$), and we take $r = Z(c)$, the central support, for checking ($F_2$).
\end{proof} 

 \begin{examples}
By Propositions \ref{AbF0}, \ref{F0F}, if $\mathcal{P}_{2,+}$ is abelian then $\mathcal{P}$ satisfies ($F_2$). We can check by hand using the table of Subsection \ref{excoprod} that $\mathcal{P}(R^{S_3} \subset R)$ also satisfies ($F_2$). 
 \end{examples}
 
Let $G$ be a finite group and $H$ a core-free subgroup.

\begin{proposition} The subfactor $(R^G \subseteq R^H)$ is ($F_2$) if and only if   for every irreducible complex representations $U,V$ of $G$, there exists $W$  also irreducible such that (see Definition \ref{fixstab}) $$G_{(W^H)} \cap G_{(V^H)} , G_{(U^H)} \cap G_{(W^H)} \subseteq G_{(U^H)} \cap G_{(V^H)}$$
\end{proposition}
\begin{proof}
It's just a reformulation of ($F_2$) using Theorem \ref{subsy}.
\end{proof}  

\begin{remark} \label{rkf}
Using GAP we have found\footnote{http://mathoverflow.net/a/201693/34538} a counter-example of the property ($F_2$), of the form $\mathcal{P}(R^G \subseteq R^H)$ with $\vert G \vert = 32$ and $\vert H \vert = 2$.   
\begin{verbatim}
gap> G:=TransitiveGroup(16,39);  H:=Stabilizer(G,1);  
\end{verbatim}
It follows that ($F_0$) and ($F_1$) are not satisfied in general.
\end{remark} 
\begin{question} Does every irreducible depth $2$ subfactor planar algebra satisfies ($F_2$), ($F_1$) or even ($F_0$)?   
\end{question} 

\begin{proposition}
For a distributive subfactor planar algebra, w-cyclic is equivalent to ($F_1$).
\end{proposition}
\begin{proof}
The first implication is immediate in general because if there is a minimal projection $c_0$ with $\langle c_0 \rangle = id$ then (following the notation of Definition \ref{F}) we get the property ($F_1$) by using the projection $c=c_0$.     

For the other implication, we follow the proof and notations of the Theorem \ref{centheo}, but instead of taking $c \preceq u * v$, we take the minimal projection $c$ given by the property ($F_1$) for $a = u$ and $b = v$, and then
$\langle c \rangle =  \langle u,c \rangle \wedge \langle c , v \rangle \ge \langle u,v \rangle \wedge \langle u , v \rangle = id \wedge id = id$.
\end{proof}

\begin{definition} \label{ZZ} \label{Z}
The planar algebra $\mathcal{P}$ satisfies 
\begin{itemize}
\item (ZZ) if the coproduct of any two central operators is central.
\item (Z) if any minimal central projection generates a central biprojection.
\item ($\tilde{Z}$) if all the intermediate subfactor planar algebras satisfy (Z).
\end{itemize} 
\end{definition} 

\begin{proposition} \label{ZZZ}
The property (ZZ) implies (Z).
\end{proposition}
\begin{proof}
By Definition \ref{gener} and assuming (ZZ), a minimal central projection generates a central biprojection.
\end{proof}
 
\begin{examples} \label{cynoz}  If $\mathcal{P}_{2,+}$ is abelian, then $\mathcal{P}$ is \textit{a fortiori} (ZZ) and ($\tilde{Z}$). By Theorem \ref{thmZZ}, if $\mathcal{P}$ is depth $2$, then it is (ZZ).
 The planar algebra $\mathcal{P}(R^{S_4} \subset R ^{\langle (1,2) \rangle})$ satisfies (Z) but not (ZZ), so the converse of Proposition \ref{ZZZ} is false. Moreover $\mathcal{P}(R^{S_4} \subset R^{\langle (1,2)(3,4) \rangle})$ does not satisfy (Z), so (Z) is not true in general (see Subsection \ref{excoprod}). There is a distributive example of index $110$ which is not $(Z)$, given\footnote{http://math.stackexchange.com/a/721164/84284} by $(R^{PSL(2,11)} \subset R^{\mathbb{Z}/6 })$. \end{examples}
 % [as for example for $(R \rtimes H \subseteq R \rtimes G)$]
% \textcolor{red}{add examples} Non-abelian depth $>2$, $\tilde{Z}$ and counter-examples.

\begin{proposition} \label{dedez}
A Dedekind subfactor planar algebra satisfies ($\tilde{Z})$.
\end{proposition}
\begin{proof}
A Dedekind subfactor planar algebra satisfies obvioulsy (Z), now using Lemma  \ref{prodcoprod}, any intermediate of a Dedekind subfactor planar algebra is also Dedekind, so the result follows. 
\end{proof}

\begin{theorem}
A distributive ($\tilde{Z})$ subfactor planar algebra is w-cyclic.
\end{theorem}
\begin{proof} 
Same first paragraph that for the proof of Theorem \ref{centheo}. Next let $b_1, \dots , b_n$ be the maximal biprojections. By induction and Theorem \ref{thmlm}, $b_i$ is lw-cyclic, i.e.  $b_i = \langle u_i \rangle$ with $u_i$ a minimal projection. If $\langle Z(u_i) \rangle = id$ then the result follows by Definition \ref{wcy}, else by maximality $\langle Z(u_i) \rangle = b_i$, and by (Z), $b_i$ is central; but by the boolean structure, every biprojection is the meet of some maximal biprojections, then all the biprojections are central; the result follows by Theorem \ref{centheo}. 
 \end{proof} 

\section{Applications}  \label{appli}
In this section, we describe some applications of Theorem \ref{thm} to subfactor planar algebras, subfactors, quantum groups and finite group theories. Similar applications, using the extensions of Section \ref{distex}, are also discussed. 
\subsection{Applications to subfactor planar algebras theory}     \hspace*{1cm}\\
Let $\mathcal{P}$ be a finite index irreducible subfactor planar algebra.  As an application of the main Theorem \ref{thm} and extensions, we get non-trivial upper bounds for the minimal number of minimal projections of $\mathcal{P}_{2,+}$ generating the identity biprojection.  

\begin{definition} \label{length}
Let, respectively,
\begin{itemize}
\item $cl(\mathcal{P})$ the cyclic length 
\item $wcl(\mathcal{P})$ the w-cyclic length 
\item $tcl(\mathcal{P})$ the top cyclic length 
\item $dl(\mathcal{P})$ the distributive length 
\item $tbl(\mathcal{P})$ the top boolean length 
\item $tb_nl(\mathcal{P})$ the top $\mathcal{B}_{\le n}$ length 
\end{itemize}  of $\mathcal{P}$, be the minimal length $l$ for an ordered chain of biprojections $$e_1=b_0 < b_1 < \dots < b_{l} = id$$ such that $\mathcal{P}(b_i < b_{i+1})$ is respectively
\begin{itemize}
\item cyclic 
\item w-cyclic 
\item top cyclic [i.e. its top intermediate (see Definition \ref{topdef}) is cyclic]
\item distributive
\item top boolean
\item top $\mathcal{B}_{r}$ with $r \le n$
\end{itemize}  Note that $tbl(\mathcal{P}) \le dl(\mathcal{P}) \le cl(\mathcal{P}) \le h(\mathcal{P})$.  \\ Moreover if $n \le m$ then $tb_nl(\mathcal{P}) \ge tb_ml(\mathcal{P}) \ge tbl(\mathcal{P})$. \\ 
We define also the lengths $bcl$, $bbl$ and $bb_nl$ as for $tcl$, $tbl$ and $tb_nl$ but replacing top by bottom (see Definition \ref{bottom}).
\end{definition} 

\begin{corollary} \label{interdistrib}
Let $a < b$ be biprojections. If $\mathcal{P}(a < b)$ is w-cyclic, then there is a minimal projection $u \in \mathcal{P}_{2,+}(e_1 < id)$ such that $\langle a, u \rangle = b$.  
\end{corollary}
\begin{proof} Let $$l_{b}: \mathcal{P}_{2,+}(e_1 \le b) \to b\mathcal{P}_{2,+}b$$ and (with $a' = l_{b}^{-1}(a)$) $$r_{a'}: \mathcal{P}_{2,+}(a \le b) \to a' * \mathcal{P}_{2,+}(e_1 \le b) * a'$$   be isomorphisms of ${\rm C}^{\star}$-algebras from Corollary \ref{landau1}, then by assumption the planar algebra $\mathcal{P}(a \le b)$ is w-cyclic, then $a'$ is rw-cyclic by Theorem \ref{thmlm}, i.e. there is a minimal projection $u'$ such that $\langle a', u' \rangle = id$. Then by applying the map $l_{b}$ we get $$b=l_b(id)=\langle l_{b}(a'), l_{b}(u') \rangle = \langle a, u \rangle$$ with $u=l_{b}(u')$.
\end{proof}

\begin{lemma} 
Let $p_1, \dots, p_n$ be minimal central projections, then there are minimal projections $u_i \le p_i$ such that $\langle u_1, \dots, u_n \rangle = \langle p_1, \dots, p_n \rangle$.
\end{lemma}
\begin{proof} 
By Theorem \ref{mini}, there are minimal projections $u_i \le p_i$ such that $\langle u_i \rangle = \langle p_i \rangle$, then \\  $\langle p_1, \dots, p_n \rangle = \langle \langle p_1 \rangle, \dots, \langle p_n \rangle \rangle = \langle \langle u_1 \rangle, \dots, \langle u_n \rangle \rangle = \langle u_1, \dots, u_n \rangle$
\end{proof}

This lemma allows the following remark.

\begin{remark}
The w-cyclic length $wcl(\mathcal{P})$ admits the following equivalent definitions: 
\begin{itemize}
\item the minimal number $n$ of minimal projections $u_1, \cdots, u_n \in \mathcal{P}_{2,+}$ such that $\langle u_1, \dots, u_n \rangle = id$.
\item the minimal number $n$  of minimal central projections $p_1, \dots, p_n \in \mathcal{P}_{2,+}$ such that $\langle p_1, \dots, p_n \rangle = id$.
\end{itemize}
\end{remark} 

\begin{corollary} \label{cychain} \label{cychain2}
The w-cyclic length of $\mathcal{P}$ is less than its cyclic length, its top cyclic length and its top $\mathcal{B}_{\le 4}$ length, \\  i.e. $wcl(\mathcal{P}) \le tcl(\mathcal{P}) \le cl(\mathcal{P})$ and $wcl(\mathcal{P}) \le tb_4l(\mathcal{P})$.
\end{corollary}
\begin{proof} Immediate from Theorem \ref{thm}, Corollary \ref{interdistrib}, Corollary \ref{B4w}, Lemmas \ref{topBn} and \ref{Topw}.
\end{proof}

 So these lengths gives non-trivial upper bounds for the minimal number of minimal central projections generating the identity biprojection. 
 
 \begin{remark} \label{topesti} By using Theorem \ref{le2} directly, we can define as above another length $tb'l$ which gives a better estimate than $tb_4l$, \\ i.e. $wcl(\mathcal{P}) \le tb'l(\mathcal{P}) \le tb_4l(\mathcal{P})$
 \end{remark}  
 
Assuming Conjecture \ref{conjext} we get the following statement: 
 
\begin{statement} \label{statement}
The w-cyclic length of $\mathcal{P}$ is less than its distributive length and its top boolean length, i.e. $wcl(\mathcal{P}) \le tbl(\mathcal{P}) \le dl(\mathcal{P})$.
\end{statement} 
 
 \subsection{Applications to subfactors theory}   \hspace*{1cm}\\
 We now reformulate for a finite index irreducible subfactor $(N \subseteq M)$. 
From the one-to-one correspondence between the intermediate subfactors and the biprojections (theorem \ref{bisch}), we define:

\begin{definition}
Let the cyclic length $cl(N \subseteq M)$  be the minimal length $n$ for an ordered chain of  intermediate subfactors $$N=P_0 \subsetneq P_1 \subsetneq \cdots \subsetneq P_{n} = M$$ such that all the intermediate subfactors of $(P_i \subsetneq P_{i+1})$ are normal \cite{teru} and form a distributive lattice.
\end{definition} 

\begin{definition}
Let the w-cyclic length $wcl(N \subseteq M)$  be the minimal number $n$ of algebraic irreducible sub-$N$-$N$-bimodules of $M$, noted $B_1, \dots , B_n$, generating $M$ as von Neumann algebra: $\langle B_1, \dots , B_n \rangle = M$. 
\end{definition} 

\begin{lemma} \label{clwcl} The cyclic length $cl(N \subseteq M) = cl(\mathcal{P}(N \subseteq M))$ and the w-cyclic length $wcl(N \subseteq M) = wcl(\mathcal{P}(N \subseteq M))$.
\end{lemma}
\begin{proof}
The first equality is immediate, the second follows from the one-to-one correspondence between the poset of projections $p$ of $\mathcal{P}_{2,+}$ and the poset of algebraic sub-$N$-$N$-bimodules $X_p$ of $M$ (see Subsection \ref{coralg}), and the Corollary \ref{coralgcor} stating that $\langle X_p \rangle = X_{\langle p \rangle}$. 
\end{proof} 
 
\begin{corollary} \label{cychainsub} 
The w-cyclic length of $(N \subseteq M)$ is less than its cyclic length, i.e. $wcl(N \subseteq M) \le cl(N \subseteq M)$.
\end{corollary}
\begin{proof} Immediate by Corollary \ref{cychain} and Lemma \ref{clwcl}.
\end{proof}

\begin{corollary} \label{mainsub}  If $(N \subseteq M)$ is a cyclic subfactor then there is an algebraic irreducible sub-$N$-$N$-bimodule $B$ of $M$ such that $\langle B  \rangle = M$, with $\langle B  \rangle$ the von Neumann algebra generated by the subset $B \subset M$.
\end{corollary}   

 So the cyclic length gives a non-trivial upper bound for the minimal number of algebraic irreducible sub-$N$-$N$-bimodules of $M$ generating $M$ as von Neumann algebra.   
 
 By defining the analogous of the lengths $tcl$, $dl$, $tbl$, $tb_nl$ of Definition \ref{length} in the subfactors framework, we get better upper bounds (by Corollary \ref{cychain} and Statement \ref{statement}).  
 
The cyclic length is not equal to the w-cyclic length in general (see the Remark \ref{cldugrp}), nevertheless the equality occurs for the finite abelian group subfactor planar algebras (see Proposition \ref{abelian}).     

\begin{theorem}[Lukacs-Palfy \cite{lupa}]
A finite group $G$ is abelian if and only if $\mathcal{L}(G \times G)$ is modular (see Remark \ref{modular}).
\end{theorem}

\begin{definition} \label{abeliandef}
The planar algebra $\mathcal{P}$ is called abelian\footnote{http://mathoverflow.net/q/156374/34538} if the biprojections are normal and $\mathcal{P} \otimes \mathcal{P}$ admits a modular biprojections lattice.
\end{definition}
It follows by the Galois correspondence that a finite group subfactor planar algebra is abelian if and only if the group is abelian.

\begin{question} \label{bothj}
Is the planar algebra $\mathcal{P}$ abelian if and only if its cyclic length is equal to its w-cyclic length and the w-cyclic length of its dual?
\end{question}  

\begin{remark} 
Any maximal subfactor planar algebra $\mathcal{P}$ is abelian, but $\mathcal{P}_{2,+}$ can be a non-abelian algebra; see for example\footnote{http://mathoverflow.net/a/158374/34538} $(R^{L_2(11)} \subset R^{D_{12}})$. 
\end{remark}  

The following questions point to a generalization of the fundamental theorem of abelian groups: 

\begin{question} \label{q1}
If the planar algebra $\mathcal{P}$ is cyclic, is it also abelian?
\end{question}
\begin{question} \label{q2}
Is a tensor product of abelian subfactor planar algebra also abelian? \end{question}
\begin{question} \label{q3}
Is every abelian subfactor planar algebra a tensor product of cyclic subfactor planar algebras?
\end{question}

\subsection{Applications to finite groups theory} \label{appgrp} \hspace*{1cm} \\
Let $[H,G]$ be an interval of finite groups. As an application we get a dual version of Ore's Theorem \ref{ore2}.  Next we get non-trivial upper bounds for the minimal number of elements (resp. irreducible complex representations) generating a finite group (resp. the left regular representation).
\begin{definition} \label{cyclicinc}
The inclusion $(H \subseteq G)$ is called $H$-cyclic if $\exists g \in G$ with $\langle H,g \rangle = G$.
\end{definition}

\begin{definition} \label{linprimdef}
The group $G$ is linearly primitive if it admits an irreducible complex representation $V$ which is faithful, or equivalently such that for all irreducible complex representation $W$ there is $n >0$ with $W \le V^{\otimes n}$, or equivalently such that $V$ generates the left regular representation (in the sense that it can appear as a direct component of a combination of $V$ for $\oplus$ and $\otimes$).
\end{definition}

 \begin{definition} \label{fixstab} Let $W$ be a representation of a group $G$, $K$ a subgroup of $G$, and $X$ a subspace of $W$. Let the \textit{fixed-point subspace} $$W^{K}:=\{w \in W \ \vert \  kw=w \ , \forall k \in K  \}$$ and the \textit{pointwise stabilizer subgroup} $$G_{(X)}:=\{ g \in G \  \vert \ gx=x \ , \forall x \in X \}$$  \end{definition} 

\begin{definition} \label{linprim}
The inclusion $(H \subseteq G)$ is called linearly primitive if there is an irreducible complex representation $V$ of $G$ with $G_{(V^H)} = H$.
\end{definition}

\begin{remark} The group $G$ is linearly primitive if and only if the interval $[\{e\} , G]$ is linearly primitive.  \end{remark}

\begin{definition} \label{Nint}
An intermediate subgroup $H \subseteq K \subseteq  G$ is called a normal-intermediate subgroup if $\forall g \in G$ $HgK = KgH$.
\end{definition}

\begin{lemma} \label{teruya}
Let $b_K \in \mathcal{P}_{2,+}(R^G \subseteq R^H)$ be the biprojection corresponding to  the intermediate subgroup $H \subseteq K \subseteq  G$, then $K$ is a normal-intermediate if and only if the biprojection $b_K$ is normal.
\end{lemma}
\begin{proof} See \cite{teru} proposition 3.3 on page 476. \end{proof}

\begin{lemma} \label{corrminstab}
Let $p_x \in \mathcal{P}_{2,+}(R^G \subseteq R)$ be a minimal projection on the one-dimensional subspace $\mathbb{C}x$ and $H$ a subgroup of $G$ then $$p_x \le b_H  
\Leftrightarrow H \subset G_x$$
\end{lemma}
\begin{proof} We use the notations of Subsection \ref{excoprod} $$b_H = \vert H \vert^{-1}\sum_{h \in H} \pi_{V}(h)$$ If $p_x \le b_H$ then $b_H x = x $ and 
$\forall h \in H $ we have that $$ \pi_V(h)  x = \pi_V(h)  [b_H  x] =  [\pi_V(h)  \cdot b_H]  x = b_H  x  = x$$
which means that $h \in G_x$, and so $H \subset G_x$.    

Conversely, if $H \subset G_x$ [i.e. $\forall h \in H$ we have $ \pi_V(h)  x = x$] then $b_H  x = x$, which means that $p_x \le b_H$.
\end{proof}

\begin{corollary} \label{wgrp}
Let $G$ acting outerly  on the hyperfinite ${\rm II}_1$ factor $R$.
\begin{itemize}
 \item  $(R \rtimes H \subseteq R \rtimes G)$ is w-cyclic if and only if $(H \subseteq G)$ is $H$-cyclic.
\item $(R^G \subseteq R^H)$ is w-cyclic iff $(H \subseteq G)$ is linearly primitive.
\end{itemize} 
\end{corollary}
\begin{proof}
By theorem  \ref{thmlm},  $(R \rtimes H \subseteq R \rtimes G)$ is w-cyclic if and only if $$\exists u \in \mathcal{P}_{2,+}(R \subseteq R \rtimes G) \simeq \bigoplus_{g \in G} \mathbb{C}e_g \simeq \mathbb{C}^G$$ minimal projection such that  $\langle e^{R \rtimes G}_{R \rtimes H}, u \rangle = id$, if and only if  $\langle H,g \rangle = G$ with $u=e_g$; and $(R^G \subseteq R^H)$ is w-cyclic if and only if $$\exists u \in \mathcal{P}_{2,+}(R^G \subseteq R) \simeq \bigoplus_{V_i \ irr.}End(V_i) \simeq \mathbb{C}G$$ minimal projection such that  $\langle u \rangle = e^R_{R^H}$, if and only if,  by Lemma \ref{corrminstab}, $H=G_x$ with $u=p_x$ the projection on $\mathbb{C}x \subseteq V_i$ (with $Z(p_x) = p_{V_i}$). \\ 
Note that $H \subset G_{(V_i^H)} \subset G_x$ so $H = G_{(V_i^H)}$.   \end{proof}
In particular, we get:
\begin{corollary} \label{wgrp2}  The subfactor $(R^G \subseteq R)$ (resp. $(R \subseteq R \rtimes G)$) is w-cyclic if and only if $G$ is  linearly primitive  (resp. cyclic). \end{corollary}    

\begin{examples} The subfactors $(R^{S_3} \subset R)$, $(R^{S_4} \subset R^{\langle (1,2) \rangle })$ and its dual are w-cyclic, but  $(R \subset R \rtimes S_3)$ and $(R^{S_4} \subset R^{\langle (1,2)(3,4) \rangle})$ are not.  \end{examples} 
 
\begin{definition}
The inclusion $(H \subseteq G)$ is called
\begin{itemize}
\item distributive if $[H,G]$ is a distributive lattice
\item Dedekind if every $K \in [H,G]$ is normal-intermediate
\item cyclic if it is both distributive and Dedekind
\end{itemize}
\end{definition}
 
\begin{corollary} \label{dualore2} If the inclusion $(H \subseteq G)$ is distributive then it is $H$-cyclic; if moreover it is cyclic then it is linearly primitive.
\end{corollary}
\begin{proof}
By Galois correspondence, Theorems \ref{centheo}, \ref{thm}, Corollary \ref{wgrp}, Lemma \ref{teruya} and $\mathcal{P}_{2,+}(R \rtimes H \subset R \rtimes G)$ commutative.
\end{proof}

The first part of Corollary \ref{dualore2}  is Ore's Theorem \ref{ore2}, but the second part (which is a dual version of the first) is new in group theory.

\begin{definition}
Let the cyclic length $cl(G)$ (resp. distributive length $dl(G)$) be the minimal length for an ordered chain of subgroups $$\{1 \}=H_0 \subset  H_1 \subset  \dots \subset  H_{n} = G$$ such that the inclusion $(H_i \subseteq H_{i+1})$ is cyclic (resp. distributive).  \\ Note that $dl(G) \le cl(G) = cl(R \subseteq R \rtimes G)$. 
\end{definition}
We can define the group theoretic analogous of the lengths $tcl$, $tbl$, $tb_nl$, $bcl$, $bb_nl$, $bbl$ of Definition \ref{length}.

\begin{corollary} The group $G$ can be generated by $dl(G)$ elements.  
\end{corollary}
\begin{proof}
Upgrade Corollary \ref{cychain} for $(R \subseteq R \rtimes G)$ using Theorem \ref{centheo}.
\end{proof}

\begin{remark} \label{clgrp} The length $dl(G)$ can be a strict upper bound for the minimal number of generators of $G$, because\footnote{http://math.stackexchange.com/q/1281368/84284} $S_n$ can be generated by two elements and $dl(S_n) = 2$ for $3 \le n \le 7$, but $dl(S_8) > 2$.
\end{remark}
The length $tbl(G)$ gives a better upper bound.

\begin{corollary} The left regular representation of $G$ can be generated (for $\oplus$ and $\otimes$) by $cl(G)$ irreducible complex representations. 
\end{corollary}
\begin{proof}
Reformulate Corollary \ref{cychain2} for $(R^G \subseteq R)$ using Corollary \ref{coprofusion}.
\end{proof} 
\begin{remark} \label{cldugrp} The length $cl(G)$ can be a strict upper bound for  the minimal number of irreducible complex representation generating the left regular representation of $G$ because $cl(S_3) = 2$ and $S_3$ is linearly primitive.
\end{remark}
The lengths $bcl(G)$ and $bb_4l(G)$ gives better upper bounds (see Corollary \ref{cychain}), and $bbl(G)$ is conjectured to be also (see Statement \ref{statement}).

\begin{proposition} \label{abelian} If a finite group is abelian then the cyclic length is  both the minimal number of elements and of irreducible complex representations, generating the group and the left regular representation. 
\end{proposition}
\begin{proof}
By the fundamental theorem of finite abelian groups and the fact that they are isomorphic to their dual and Dedekind. 
\end{proof}

\begin{question}
Is the converse true? (see also Question \ref{bothj})
\end{question}

\subsection{Applications to quantum groups theory}  \hspace*{1cm} \\
Let $\mathbb{A}$ be a finite dimensional Kac algebra (see Subsection \ref{kac}).
\begin{definition}
$\mathbb{A}$  is linearly primitive if there is an irreducible complex representation $V$ such that, for all irreducible complex representation $W$ there is $n >0$ with $W \le V^{\otimes n}$, or equivalently, the projection $p_V$ generates $\mathbb{A}$ as left coideal subalgebra (thanks to Theorem \ref{thmcopro}).
\end{definition}  

\begin{corollary}
The Kac algebra $\mathbb{A}$ is linearly primitive if and only if the depth $2$ irreducible finite index subfactor $(R^{\mathbb{A}} \subseteq R)$ is w-cyclic.
\end{corollary} 
\begin{proof}
Immediate by Definition \ref{gener}  and Corollary \ref{coprofusion}.
\end{proof}

\begin{definition} \label{maxkac} The Kac algebra $\mathbb{A}$ is called maximal if is has no left coideal subalgebra other than $\mathbb{C}$ and $\mathbb{A}$.
\end{definition}

\begin{definition} The Kac algebra $\mathbb{A}$ is called cyclic if all the left coideal subalgebras are normal Kac subalgebras (see \cite{teru} (3) p474), and form a distributive lattice $\mathcal{L}(\mathbb{A})$ (see Problem \ref{cykac}).
\end{definition}

\begin{corollary} \label{linkac} If $\mathbb{A}$ is cyclic, then it is linearly primitive.
\end{corollary}
\begin{proof}
It is a reformulation of Theorem \ref{thm} for the Kac algebras. 
\end{proof}

\begin{remark} \label{cycoid} Let $\mathbb{I} \subseteq \mathbb{J} \subseteq \mathbb{A}$ be left coideal subalgebras of $\mathbb{A}$  then by using proposition 4.2. p52 in \cite{ilp}, together with Proposition \ref{coropro} and Theorem \ref{th},  we could define in the Kac algebra framework the notion of linearly primitive (resp. cyclic) for the inclusion $(\mathbb{I} \subseteq \mathbb{J})$, to be equivalent to  $(R^{\mathbb{J}} \subseteq R^{\mathbb{I}})$ w-cyclic (resp. cyclic), for finally getting a generalization of Corollary \ref{linkac}.  \end{remark} 

\begin{definition}
Let the cyclic length $cl(\mathbb{A})$ be the minimal length for an ordered left coideal subalgebras chain  $$\mathbb{C} \subsetneq \mathbb{B}_1 \subsetneq \cdots \subsetneq \mathbb{B}_{n} = \mathbb{A}$$ with the subfactor $(\mathbb{B}_{i-1} \subset \mathbb{B}_{i})$ cyclic (see Remark \ref{cycoid}).
\end{definition} 
We can define the Kac algebra theoretic analogous of the lengths $tcl$, $tbl$, $tb_nl$, $bcl$, $bb_nl$, $bbl$ of Definition \ref{length}.

\begin{corollary} The cyclic length $cl(\mathbb{A})$ is both an upper bound for the minimal number of: 
\begin{itemize}
\item minimal central projections generating  $\mathbb{A}$ as a left coideal subalgebra.
\item irreducible complex corepresentations generating (for $\oplus$ and $\otimes$) the dual $\mathbb{A}^{\star}$ of $\mathbb{A}$, as a corepresentation. 
\end{itemize}
\end{corollary}
\begin{proof}
We reformulate Corollary \ref{cychain2} for $(R \subseteq R \rtimes \mathbb{A})$  and  $(R^{\mathbb{A}} \subseteq R)$.
\end{proof}

\begin{remark} The length $cl(\mathbb{A})$ can be a strict upper bound (see Remark \ref{clgrp}). The lengths $tcl$ and $tb_4l$ (resp. $bcl$ and $bb_4l$) gives better upper bounds (see Corollary \ref{cychain}), and $tbl$ (resp. $bbl$) is conjectured to be also (see Statement \ref{statement}).
\end{remark} 

Following the Definition \ref{abeliandef}, we define: 
\begin{definition}
A Kac algebra $\mathbb{A}$ is $l$-abelian if all the left coideal subalgebras are normal Kac subalgebras and $\mathcal{L}(\mathbb{A} \otimes \mathbb{A})$ is modular.
\end{definition}
The $l$ of $l$-abelian is for lattice. This attribute is necessary for not confusing with a usual abelian Kac algebra. 
\begin{problem}
Is there a non-trivial $l$-abelian Kac algebra?
\end{problem}
We can reformulate the questions from \ref{bothj} to \ref{q3}  in this framework.

\section{Appendix}
\subsection{Some correspondences}
\subsubsection{Correspondence sub-bimodules and $2$-box projections} \label{coralg}
Let $(N \subseteq M)$ be a finite index irreducible subfactor, and $\mathcal{P}=\mathcal{P}(N \subseteq M)$ its planar algebra. We can see $M$ as an algebraic $N$-$N$-bimodule \cite{joalg}, it decomposes into irreducible algebraic $N$-$N$-bimodules  $$M = \bigoplus_i V_i \otimes B_i$$ with $B_1, \dots , B_n$ the (equivalent class representatives of the) irreducible algebraic sub-$N$-$N$-bimodules of $M$, and $V_i$ the multiplicity space.  Now $\mathcal{P}_{2,+} = N' \cap M_1$ is a finite dimensional ${\rm C}^{\star}$-algebra and by the double characterization of the principal graph (see \cite{js} section 4.2), we get that  
$$ \mathcal{P}_{2,+} \simeq \bigoplus_i End(V_i)$$  

The one-to-one correspondence between the projections of $p \in \mathcal{P}_{2,+}$ and the algebraic sub-$N$-$N$-bimodules $X_p$ of $M$, comes from the one-to-one correspondence between the projections of $\mathcal{P}_{2,+}$ and the subspaces of $V=\bigoplus_i V_i$ on one hand (through image and range projection), and from the one-to-one correspondence between the subspaces of $V$ and the algebraic sub-$N$-$N$-bimodules of $M$ on the other hand (by definition of the decomposition of $M$). Moreover this correspondence is an isomorphism of poset, which means
$$p \le q \Leftrightarrow X_p \subseteq X_q$$
The set of biprojections is a subposet of the set of projections and the set of intermediate subfactors is also a subposet of the set of algebraic sub-$N$-$N$-bimodules of $M$, and using \cite{bi}, $p$ is a biprojection if and only if $ X_p$   is an intermediate subfactor.
 
\begin{corollary} \label{coralgcor} The biprojection $\langle p \rangle$ generated by the projection $p \in \mathcal{P}_{2,+}$ corresponds to the intermediate subfactor $\langle X_p \rangle$ generated by the algebraic sub-$N$-$N$-bimodule $X_p$ of $M$, which can be reformulated by 
$$X_{\langle p \rangle} = \langle X_{ p } \rangle$$ 
\end{corollary}
\begin{proof}  First $\langle p \rangle$ is the smallest biprojection $b \ge p$, whereas $\langle X_p \rangle$ is the smallest von Neumann algebra containing $X_p$ in $M$, but $N \subseteq \langle X_p \rangle$ because $X_p$ is a $N$-$N$-bimodule, and so $\langle X_p \rangle$ is an intermediate subfactor of $(N \subseteq M)$ by irreducibility. The result follows by the poset isomorphism respecting the one-to-one correspondence between the subposet of biprojections and  the subposet of intermediate subfactors. \end{proof}

Let $X_p X_q$ be the pointwise product of $X_p$ and $X_q$ in $M$, then $Vect_{\mathbb{C}}(X_p X_q)$ is an algebraic sub-$N$-$N$-bimodules of $M$.
\begin{question}
Is $Vect_{\mathbb{C}}(X_p X_q) $equals to $ X_{R(p * q)}$? \footnote{http://mathoverflow.net/q/209195/34538}  \\ (with $R(x)$ the range projection of $x$)
\end{question}

\subsubsection{Galois correspondence subgroups/subsystems} \hspace*{1cm} \\
In \cite{ilp} page 49, there is the following result on compact groups: 

\begin{theorem} \label{subsy}
Let $G$ be a compact group and $Rep(G)$ the category
of finite dimensional unitary representations of $G$.
For $\pi \in Rep(G)$ $H_\pi$ denotes the representation space of $\pi$.
Suppose we have a Hilbert subspace $K_\pi\subset H_\pi$ for each
$\pi\in Rep(G)$ satisfying the following:
$$K_\pi\oplus K_\sigma \subseteq K_{\pi \oplus \sigma}, \
\pi,\sigma \in Rep(G),$$
$$K_\pi\otimes K_\sigma \subseteq K_{\pi \otimes \sigma}, \
\pi,\sigma \in Rep(G),$$
$$\overline{K_\pi}=K_{\overline{\pi}}, \ \pi \in Rep(G),$$
where $\overline{\pi}$ is the complex conjugate representation and
$\overline{K_\pi}$ is the image of $K_\pi$ under the natural map from $H_\pi$
to its complex conjugate Hilbert space.
Then there exists a closed subgroup $H \subseteq G$ such that
$$K_\pi=\{\xi \in H_\pi; \pi(h)\xi=\xi, \ h\in H\}$$
\end{theorem} 

\subsection{Some results for the depth $2$ case}   \label{kac}
Let $(N \subseteq M)$ be an irreducible depth $2$  subfactor of finite index $[M:N]=: \delta^{2}$.    

\begin{theorem}
The subfactor $(N \subseteq M)$ is given by a Kac algebra, i.e. a Hopf ${\rm C}^{\star}$-algebra $(\mathbb{A}, \Delta, \epsilon, S)$ with $\mathbb{A}=N' \cap M_1=\mathcal{P}_{2,+}(N \subseteq M)$, $(N \subseteq M) \simeq (R^{\mathbb{A}} \subseteq R) $ and $dim(\mathbb{A})=[M:N]$.
\end{theorem} 
\begin{proof} See \cite{szka,lonka,david} and the recent ``planar algebra" approach \cite{dako}. \end{proof}
\underline{Trivial case}: $\mathbb{A}=\mathbb{C}G$, $\Delta(g) = g \otimes g$, $\epsilon(g)=1$ and $S(g)= g^{-1}$.
\begin{theorem}[Galois correspondence]
Every intermediate subfactor $R^{\mathbb{A}} \subseteq P \subseteq R$ are of the form $P=R^{\mathbb{B}}$ with $\mathbb{B} \subseteq \mathbb{A}$ a left coideal $\star$-subalgebra (i.e. $\Delta(\mathbb{B}) \subseteq \mathbb{A} \otimes \mathbb{B}$ and $\mathbb{B}^{\star} = \mathbb{B}$), and conversely. 
\end{theorem}
\begin{proof} \cite{ilp} Theorem 4.4 p54. \end{proof}

\begin{theorem}[Schur's lemma] Let $\mathcal{A}$ be a finite dimensional C$^{\star}$-algebra, $V$ a representation, $V_1$ and $V_2$ irreducible representations, then 
 \begin{itemize} 
\item the action of $\mathcal{A}$ on $V$ is irreducible (i.e. has no invariant subspace) if and only if $\pi_{V}(\mathcal{A})' = \mathbb{C}I_{V}$.
\item if $T \in Hom_{\mathcal{A}}(V_{1},V_{2})$ (i.e. intertwines the action of $\mathcal{A}$) then $T=0$ or $T$ is an isomorphism. \end{itemize}  \end{theorem} 

\begin{theorem}[Double commutant theorem] Let $\mathcal{A} \subseteq End(V)$ be a finite dimensional C$^{\star}$-algebra, then $\mathcal{A}'' $ is equal to $ \mathcal{A}$.  \end{theorem}  

\begin{theorem} \label{multimat} The finite dimensional Kac algebra $\mathbb{A}$ admits finitely many nonequivalent irreducible complex representations $H_{1}, \dots, H_{r}$ so that as ${\rm C}^{\star}$-algebra, $\mathbb{A} \simeq  End(H_1) \oplus \cdots \oplus End(H_r)$.
\end{theorem}  

\begin{definition} \label{kacirr} Let $V$ and $W$ be two representations of $\mathbb{A}$, then $\mathbb{A}$ acts on $V \otimes W$ by using the comultiplication $\Delta$ as follows:   
   $$\forall x \in \mathbb{A}, \forall v \in V, \forall w \in W : \Delta(x) \cdot (v\otimes w) = \sum (x_{(1)} \cdot v)\otimes (x_{(2)} \cdot w)$$ \end{definition} 

\begin{definition}[Fusion rules] The previous action of $\mathbb{A}$ on $H_{i}\otimes H_{j}$ decomposes into irreducible representations 
$$ H_{i}\otimes H_{j} = \bigoplus_{k}{M_{ij}^{k} \otimes H_{k}}    $$   
with $M_{ij}^{k}$ the multiplicity space.  \end{definition}  
\begin{remark}
Let $n_k = \dim (H_{k})$ and $n_{ij}^{k} = \dim (M_{ij}^{k})$, then
$$\sum n_{i} \cdot n_{j} = \sum n_{ij}^{k} \cdot n_{k}$$
\end{remark}

The following theorem explains the relation between comultiplication and fusion rules.
\begin{theorem} \label{thmcopro} The inclusion matrix of the unital inclusion of finite dimensional ${\rm C}^{\star}$-algebras $\Delta(\mathbb{A}) \subseteq \mathbb{A} \otimes \mathbb{A}$ is $\Lambda = (n_{ij}^{k})$.  \end{theorem}
\begin{proof}

The irreducible representations of  $\mathbb{A} \otimes \mathbb{A}$ are $(H_{i}\otimes H_{j})_{i,j}$, so by double commutant theorem and Schur's lemma we get that   
$$\pi_{H_{i}\otimes H_{j}}(\mathbb{A} \otimes \mathbb{A}) = \pi_{H_{i}\otimes H_{j}}(\mathbb{A} \otimes \mathbb{A})'' = End(H_{i}\otimes H_{j}) \simeq M_{n_{i}n_{j}}(\mathbb{C})$$ 
Moreover by Definition \ref{kacirr} and fusion rules we get that
$$\pi_{H_{i}\otimes H_{j}}(\Delta(\mathbb{A})) \simeq \bigoplus_{k}{M_{ij}^{k} \otimes \pi_{H_{k}}(\mathbb{A})}  \simeq \bigoplus_{k}{M_{ij}^{k} \otimes M_{n_{k}}(\mathbb{C})}$$  
 Let $V = \bigoplus_{i,j} H_{i}\otimes H_{j}$, then by applying Theorem \ref{multimat} to $\mathbb{A} \otimes \mathbb{A}$ we get the isomorphism of inclusions:  
$$(\Delta(\mathbb{A})  \subseteq \mathbb{A}\otimes \mathbb{A}) \simeq (\pi_{V}(\Delta(\mathbb{A})) \subseteq \pi_{V}(\mathbb{A}\otimes \mathbb{A}))$$  
But $\pi_V =  \bigoplus_{i,j} \pi_{H_{i}\otimes H_{j}} $ and  the inclusion matrix of
$$\pi_{H_{i}\otimes H_{j}}(\Delta(\mathbb{A})) \subseteq \pi_{H_{i}\otimes H_{j}}(\mathbb{A} \otimes \mathbb{A})$$ is $(n_{ij}^1, \dots , n_{ij}^r)$, so the result follows.   \end{proof} 
%$$\pi_{V}(\mathbb{A} \otimes \mathbb{A}) = \pi_{V}(\mathbb{A} \otimes \mathbb{A})''  \simeq \bigoplus_{i,j} M_{n_{i}n_{j}}(\mathbb{C}) $$  
 % The comultiplication $\Delta : \mathbb{A} \hookrightarrow \mathbb{A} \otimes \mathbb{A} $, gives an inclusion of C$^{\star}$-algebra:  $\mathbb{A} \simeq \Delta(\mathbb{A})  \subseteq \mathbb{A}\otimes \mathbb{A}$ 

\begin{theorem}[Splitting] \label{split} There is the following planar reformulation of the comultiplication $\Delta(x) = \sum x_{(1)} \otimes x_{(2)}$ for $x \in \mathbb{A} = \mathcal{P}_{2,+}(R^{\mathbb{A}} \subseteq R)$.  

$\begin{tikzpicture}[scale=.5, PAdefn]
	\clip (0,0) circle (3cm);
	\draw[shaded] (-0.15,0) -- (100:4cm) -- (80:4cm) -- (0.15,0);
	\draw[shaded] (0,0) -- (-120:4cm) -- (-60:4cm) -- (0,0);
	\node at (0,0) [Tbox, inner sep=2mm] {\small{$x$} };
	%\node at (180:1.3cm) {$\star$};
 \draw[fill=white] (-0.7,-4) .. controls ++(120:3cm) and ++(60:3cm) .. (0.7,-4);	
	%\draw[very thick] (0,0) circle (3cm);
\end{tikzpicture}  
\hspace{-0.6cm} =  \ \sum 	
	\begin{tikzpicture}[scale=.5, PAdefn]
	\clip (0,0) circle (3cm);
	\draw[shaded] (-2,0) -- (0,5) -- (2,0) -- (-2,0);
	\draw[fill=white] (-0.7,-0.5) .. controls ++(120:3cm) and ++(60:3cm) .. (0.7,-0.5);
	%\draw[shaded] (0,1) -- (-120:4cm) -- (-60:4cm) -- (0,1);
	\draw[shaded] (-1.6,0) -- (-120:4cm) -- (-100:4cm) -- (-1.4,0);
	\draw[shaded] (1.4,0) -- (-80:4cm) -- (-60:4cm) -- (1.6,0);
	\node at (-1.5,0) [Tbox, inner sep=0.8mm] {\tiny{$x_{(1)}$} };
	\node at (1.5,0) [Tbox, inner sep=0.8mm] {\tiny{$x_{(2)}$} };
	%\node at (180:1.3cm) {$\star$};
	%\draw[very thick] (0,0) circle (3cm);
\end{tikzpicture} 
\ \text{ and}  
\begin{tikzpicture}[scale=.5, PAdefn]
	\clip (0,0) circle (3cm);
	\draw[shaded] (-0.15,0) -- (260:4cm) -- (280:4cm) -- (0.15,0);
	\draw[shaded] (0,0) -- (60:4cm) -- (120:4cm) -- (0,0);
	\node at (0,0) [Tbox, inner sep=2mm] {\small{$x$} };
	%\node at (0:1.3cm) {$\star$};
 \draw[fill=white] (-0.7,4) .. controls ++(240:3cm) and ++(300:3cm) .. (0.7,4);	
	%\draw[very thick] (0,0) circle (3cm);
\end{tikzpicture} 
\hspace{-0.6cm} =  \ \sum 	
\begin{tikzpicture}[scale=.5, PAdefn]
	\clip (0,0) circle (3cm);
	\draw[shaded] (-2,0) -- (0,-5) -- (2,0) -- (-2,0);
	\draw[fill=white] (-0.7,0.5) .. controls ++(240:3cm) and ++(300:3cm) .. (0.7,0.5);
	%\draw[shaded] (0,1) -- (-280:4cm) -- (-300:4cm) -- (0,1);
	\draw[shaded] (-1.6,0) -- (-240:4cm) -- (-260:4cm) -- (-1.4,0);
	\draw[shaded] (1.4,0) -- (80:4cm) -- (60:4cm) -- (1.6,0);
	\node at (-1.5,0) [Tbox, inner sep=0.8mm] {\tiny{$x_{(1)}$} };
	\node at (1.5,0) [Tbox, inner sep=0.8mm] {\tiny{$x_{(2)}$} };
	%\node at (180:1.3cm) {$\star$};
	%\draw[very thick] (0,0) circle (3cm);
\end{tikzpicture} 
$
\end{theorem} 
\begin{proof}
See \cite{kls} p39.
\end{proof}

\begin{corollary} \label{thmZZ}
If $a,b \in \mathbb{A}$ are central, then $a * b$ is also central.
\end{corollary}
\begin{proof}
 This diagrammatic proof by splitting is due to Vijay Kodiyalam.\\
 $(a*b) \cdot x =  \hspace{-0.6cm}
\begin{tikzpicture}[scale=.5, PAdefn]
	\clip (0,0) circle (3cm);
	\draw[shaded] (-0.15,0) -- (100:4cm) -- (80:4cm) -- (0.15,0);
	\draw[shaded] (0,1.2) -- (-130:4cm) -- (-50:4cm) -- (0,1.2);
	\draw[shaded] (0,1.2) -- (-130:4cm) -- (-50:4cm) -- (0,1.2);
	\node at (0,1.2) [Tbox, inner sep=1.35mm] {\tiny{$x$} };
	\draw[fill=white] (0,-1.2) circle (0.75cm);
	\node at (-1.2,-1.2) [Tbox, inner sep=1.35mm] {\tiny{$a$} };
	\node at (1.2,-1.2) [Tbox, inner sep=1.35mm] {\tiny{$b$} };
 %\draw[fill=white] (-0.7,-4) .. controls ++(120:3cm) and ++(60:3cm) .. (0.7,-4);	
\end{tikzpicture}  
\hspace{-0.6cm} =  \sum   \textbf{\hspace{-0.3cm}}	
	\begin{tikzpicture}[scale=.5, PAdefn]
	\clip (0,0) circle (3cm);
	\draw[shaded] (-2,1) -- (0,5) -- (2,1) -- (-2,1);
	\draw[shaded] (-2,-1) -- (0,-5) -- (2,-1) -- (-2,-1);
	\draw[shaded] (-1.65,-1.2) -- (1.65,-1.2) -- (1.65,1.4) -- (-1.65,1.4) -- (-1.65,-1.2);
	\draw[fill=white] (0,0) ellipse (0.8cm and 2.3cm);
	\node at (-1.2,-1.2) [Tbox, inner sep=1.35mm] {\tiny{$a$} };
	\node at (1.2,-1.2) [Tbox, inner sep=1.35mm] {\tiny{$b$} };
	\node at (-1.2,1.4) [Tbox, inner sep=0.1mm] {\tiny{$x_{(1)}$} };
	\node at (1.2,1.4) [Tbox, inner sep=0.1mm] {\tiny{$x_{(2)}$} };
	%\node at (180:1.3cm) {$\star$};
\end{tikzpicture} 
\hspace{-0.2cm} =  \sum  \hspace{-0.3cm}	
	\begin{tikzpicture}[scale=.5, PAdefn]
	\clip (0,0) circle (3cm);
	\draw[shaded] (-2,1) -- (0,5) -- (2,1) -- (-2,1);
	\draw[shaded] (-2,-1) -- (0,-5) -- (2,-1) -- (-2,-1);
	\draw[shaded] (-1.65,-1.2) -- (1.65,-1.2) -- (1.65,1.4) -- (-1.65,1.4) -- (-1.65,-1.2);
	\draw[fill=white] (0,0) ellipse (0.8cm and 2.3cm);
	\node at (-1.2,-1.2) [Tbox, inner sep=0.1mm] {\tiny{$x_{(1)}$} };
	\node at (1.2,-1.2) [Tbox, inner sep=0.1mm] {\tiny{$x_{(2)}$} };
	\node at (-1.2,1.4) [Tbox, inner sep=1.35mm] {\tiny{$a$} };
	\node at (1.2,1.4) [Tbox, inner sep=1.35mm] {\tiny{$b$} };
	%\node at (180:1.3cm) {$\star$};
\end{tikzpicture} 
  \hspace{-0.15cm} =  \hspace{-0.4cm}
\begin{tikzpicture}[scale=.5, PAdefn]
	\clip (0,0.2) circle (3cm);
	\draw[shaded] (-0.15,2) -- (-80:4.2cm) -- (-100:4.2cm) -- (0.15,2);
	\draw[shaded] (0,-1) -- (50:5cm) -- (130:5cm) -- (0,-1);
	\node at (0,-1) [Tbox, inner sep=1.35mm] {\tiny{$x$} };
	\draw[fill=white] (0,1.4) circle (0.75cm);
	\node at (-1.2,1.4) [Tbox, inner sep=1.35mm] {\tiny{$a$} };
	\node at (1.2,1.4) [Tbox, inner sep=1.35mm] {\tiny{$b$} };
 %\draw[fill=white] (-0.7,-4) .. controls ++(120:3cm) and ++(60:3cm) .. (0.7,-4);	
\end{tikzpicture}  
\hspace{-0.4cm} = x \cdot (a*b)$
 \end{proof}  
 
Let the inner product $\langle a \vert b \rangle = tr(b^{\star}a) = \delta^{-2}  \hspace{-0.6cm}
\begin{tikzpicture}[scale=.5, PAdefn]
	\clip (0,0) circle (3.2cm);
	\draw[shaded] (0.8,0) ellipse (1.5cm and 3cm);
	\draw[fill=white] (0.8,0) ellipse (0.8cm and 2.2cm);
	\node at (0,1.5) [Tbox, inner sep=1.2mm] {\small{$a$} };
	\node at (0,-1.5) [Tbox, inner sep=0.7mm] {\small{$b^{\star}$} };	
\end{tikzpicture}$

 \begin{lemma} \label{tr(x(a*b))}
 Let $a,b,x \in \mathbb{A} = \mathcal{P}_{2,+}(N \subseteq M)$ then $$\langle a * b \vert x\rangle = \delta \sum \langle a \vert x_{(1)} \rangle \langle b \vert x_{(2)}\rangle $$
 \end{lemma}  
\begin{proof} By Theorem \ref{split} (splitting) and Lemma \ref{tr(a*b)} we have
$$tr(x^{\star} \cdot (a*b)) = \delta \sum tr(x^{\star}_{(1)}a) tr(x^{\star}_{(2)}b)$$
the result follows by definition of the inner product. 
\end{proof} 

 Let $(b_{\alpha})_{\alpha}$ be an orthonormal basis of $\mathbb{A}$, i.e. $\langle b_{\alpha} \vert b_{\beta}\rangle = \delta_{\alpha,\beta}$. We get  structure constants of the comultiplication and the coproduct by: $$\Delta(b_{\alpha}) = \sum_{\beta,\gamma} c_{\beta \gamma}^{\alpha} (b_{\beta} \otimes b_{\gamma}) \text{ and } b_{\beta} * b_{\gamma} =  \sum_{\alpha} d_{\beta \gamma}^{\alpha}  b_{\alpha}$$
   
\begin{proposition} \label{coropro} \label{coprod} The coproduct can be reformulated as follows:   $$b_{\beta} * b_{\gamma} = \delta \sum_{\alpha} \overline{c_{\beta \gamma}^{\alpha}} b_{\alpha}  $$ \end{proposition} 
\begin{proof} It suffices to prove that $$d_{\beta \gamma}^{\alpha} = \delta \overline{c_{\beta \gamma}^{\alpha}}$$
We compute the inner product   $\langle b_{\beta} * b_{\gamma} \vert b_{\alpha}\rangle$ in two different manners: one by the structure constants of $b_{\beta} * b_{\gamma}$ and the other diagrammatically. First $$\langle b_{\beta} * b_{\gamma}\vert b_{\alpha}\rangle = \sum_{\alpha'} d_{\beta \gamma}^{\alpha'} \langle b_{\alpha'} \vert b_{\alpha}\rangle = \sum_{\alpha'} d_{\beta \gamma}^{\alpha'} \delta_{\alpha' \alpha} = d_{\beta \gamma}^{\alpha}$$ Next by Lemma \ref{tr(x(a*b))} $$\langle b_{\beta} * b_{\gamma} \vert  b_{\alpha}\rangle = \sum_{\beta' \gamma'} \overline{c_{\beta' \gamma'}^{\alpha}} \delta \langle b_{\beta} \vert b_{\beta'} \rangle \langle b_{\gamma} \vert b_{\gamma'} \rangle = \delta \overline{c_{\beta \gamma}^{\alpha}}$$    \end{proof}

Let $p_1, \dots , p_r$ be the minimal central projections of $\mathcal{P}_{2,+}(N \subseteq M) =\mathbb{A} = \bigoplus_{i} End(H_i)$, i.e. $p_i$ is the projection on $H_i$.

\begin{corollary} \label{coprofusion} The coproduct of minimal central projections is related to the fusion rules as follows: $$p_i * p_j \sim \sum_k n_{ij}^{k} p_k$$ \end{corollary} 
\begin{proof} By Corollary \ref{thmZZ} $$p_i * p_j \sim \sum_k \epsilon_{ij}^{k} p_k$$ with $\epsilon_{ij}^{k} \in \{0,1\}$.
Now by Theorem  \ref{thmcopro}, $n_{ij}^{k} \neq 0$ if and only if $\langle \Delta(p_k) \vert p_i \otimes p_j \rangle \neq 0$, if and only if
 $\epsilon_{ij}^{k} \neq 0$ by Proposition \ref{coropro}. \end{proof}
\begin{remark}  Corollary \ref{coprofusion} is not true at depth $>2$ in general, because it implies the property (ZZ) of Definition \ref{ZZ}, which is not true in general (see Proposition \ref{ctZZ}). How to generalize this result is an important problem\footnote{http://mathoverflow.net/q/179188/34538}, and we can expect that the coproduct corresponds to a non-central truncation of the fusion.  \end{remark}

\subsection{Some coproduct tables} \label{excoprod}
We would like to compute the coproduct tables for group subfactors. Let $G$ be a finite group, $\delta = \vert G \vert^{1/2}$ and let $V_1, \dots , V_n$ be the (equivalent class representatives of the) irreducible complex representations of $G$. By \cite{js} p53 $$\mathcal{P}_{2,+}:= \mathcal{P}_{2,+}(R^G \subseteq R)  \simeq \mathbb{C}G \simeq  \bigoplus_i End(V_i)$$ as $\rm{C}^{\star}$-algebras, and also $$\mathcal{P}_{2,-} \simeq \mathcal{P}_{2,+}(R \subseteq R \rtimes G) \simeq \mathbb{C}^G \simeq  \bigoplus_{g \in G} \mathbb{C}e_g$$ 

  Let $V = \bigoplus_i V_i$, then the Fourier transform $\mathcal{F}: \mathcal{P}_{2,+} \to \mathcal{P}_{2,-}$ is given by $\mathcal{F}(\pi_V(g)) = \delta e_g$, then $$e_g * e_h = \mathcal{F}(\mathcal{F}^{-1}(e_g) \cdot \mathcal{F}^{-1}(e_h)) =  \delta^{-2} \mathcal{F}( \pi_V(g) \cdot \pi_V(h)) = \delta^{-1} e_{gh}$$ And idem $$\pi_V(g) * \pi_V(h) = \delta_{g,h} \delta  \pi_V(g)$$  

Let $(b_{\alpha})_{\alpha}$ be an  orthogonal basis of $\mathcal{P}_{2,+}$, then $\pi_V(g) = \sum_{\alpha} \lambda_{\alpha, g} b_{\alpha}$ and by inverting the matrix $M=(\lambda_{\alpha, g})$ we get $b_{\alpha} = \sum_{g \in G} \mu_{g, \alpha} \pi_V(g)$. It follows that  $$b_{\alpha} * b_{\beta} =  \sum_{g,h \in G} \mu_{g,\alpha}\mu_{h, \beta} \pi_V(g) * \pi_V(h) = $$ $$ \sum_{g,h \in G} \mu_{g,\alpha}\mu_{h, \beta} \delta_{g,h} \delta  \pi_V(g) =  \delta \sum_{\gamma} (\sum_{g \in G} \mu_{g, \alpha}\mu_{g, \beta} \lambda_{\gamma, g}) b_{\gamma}$$ 

Let $G = S_3$ then $\mathcal{P}_{2,+}(R^{S_3} \subset R)  = \mathbb{C} \oplus \mathbb{C} \oplus M_2(\mathbb{C})$. There is\footnote{\tiny http://groupprops.subwiki.org/wiki/Linear$\_$representation$\_$theory$\_$of$\_$symmetric$\_$group:S3} a matrix basis $\{e_1 , e_2 , e_{11}, e_{12}, e_{21}, e_{22}\}$ such that   (with $\zeta_3 = e^{2i\pi/3}$): \\  $\pi_{V}(e) =  (1) \oplus (1) \oplus \begin{pmatrix}
1& 0 \\
0& 1 
\end{pmatrix}$, $\pi_{V}(123) = (1) \oplus (1) \oplus \begin{pmatrix}
\zeta_3& 0 \\
0& \bar{\zeta_3} 
\end{pmatrix}$, \\  
 $\pi_{V}(132)=(1) \oplus (1) \oplus \begin{pmatrix}
\bar{\zeta_3}& 0 \\
0& \zeta_3 
\end{pmatrix}$,  $\pi_{V}(12) =(1) \oplus (-1) \oplus \begin{pmatrix}
0& 1 \\
1& 0 
\end{pmatrix}$, \\ 
 $\pi_{V}(23)= (1) \oplus (-1) \oplus \begin{pmatrix}
0& \bar{\zeta_3} \\
\zeta_3& 0 
\end{pmatrix}$, $\pi_{V}(13)= (1) \oplus (-1) \oplus  \begin{pmatrix}
0& \zeta_3 \\
\bar{\zeta_3}& 0 
\end{pmatrix}$     

By the computation above we get this table (up to dividing by $\sqrt{6}$):  
$$\begin{array}{c||c|c|c|c|c|c} 
  * & e_1 & e_2 & e_{11}& e_{12}& e_{21} & e_{22}  \\  \hline \hline 
 e_1 & e_1 & e_2 & e_{11}& e_{12}& e_{21} & e_{22}  \\ \hline
 e_2 & e_2 & e_1 & e_{11}& -e_{12}& -e_{21}& e_{22}  \\ \hline
 e_{11} & e_{11} & e_{11} & 2e_{22}& 0 & 0 & 2(e_1+e_2)  \\ \hline
 e_{12} & e_{12} & -e_{12} & 0& 2e_{21} & 2(e_1-e_2)& 0  \\ \hline
 e_{21} & e_{21} & -e_{21} & 0& 2(e_1-e_2)& 2e_{12}& 0  \\ \hline
 e_{22} & e_{22} & e_{22} & 2(e_1+e_2)& 0 & 0 & 2e_{11}
\end{array} $$
\begin{remark} By Remark \ref{wgrp2}, $(R^{S_3} \subset R)$ is w-cyclic because $S_3$ is linearly primitive. Now by the table above we see directly that $\langle e_{11} \rangle = id$ because $\sum_{n=1}^3 e_{11}^{*n} = e_{11} + 2e_{22} + 4(e_1 + e_2) \sim id$. \end{remark}
Now let $H$ be a subgroup of $G$, then by \cite{js} p141, $$\mathcal{P}_{2,+}:= \mathcal{P}_{2,+}(R^G \subseteq R^H) \simeq \bigoplus_{V_i^H \neq 0} End(V_i^H)$$ as $\rm{C}^{\star}$-algebras, and also $$\mathcal{P}_{2,-} \simeq  \mathcal{P}_{2,+}(R \rtimes H \subseteq R \rtimes G) \simeq \bigoplus_{g \text{ repr.}} \mathbb{C}e_{HgH}$$
 By Section \ref{interpa} $$\mathcal{P}_{2,+} \simeq b_H \mathcal{P}_{2,+}(R^G \subseteq R) b_H$$ and $$\mathcal{P}_{2,-} \simeq e_H * \mathcal{P}_{2,+}(R \subseteq R \rtimes G) * e_H$$ as $2$-box spaces, with $b_H = \vert H \vert^{-1}\sum_{h \in H} \pi_{V}(h)$ and $e_H  = \sum_{h \in H} e_h$. The algebra $\bigoplus_{V_i^H \neq 0} End(V_i^H)$ is a subalgebra of $\bigoplus_{V_i} End(V_i)$, and for a well-chosen basis, the coproduct table on $\mathcal{P}_{2,+}$ is a sub-table of the one of $\mathcal{P}_{2,+}(R^G \subseteq R)$. Finally we get the coproduct table of $\mathcal{P}_{2,-}$ as a renormalized double-coset grouping of the coproduct table of  $\mathcal{P}_{2,+}(R \subseteq R \rtimes G)$; more precisely $$e_{Hg_1H} * e_{Hg_2H} = \sum_{g_3 \text{ repr.}} c_{g_1,g_2}^{g_3} e_{Hg_3H}$$ such that $$(\sum_{g \in Hg_1H}e_g) * (\sum_{g \in Hg_2H}e_g) = \vert H \vert^{1/2} \sum_{g_3 \text{ repr.}} c_{g_1,g_2}^{g_3} (\sum_{g \in Hg_3H}e_g)$$

Let $[H , G] = [ \langle(1,2)\rangle , S_4]$ then $\mathcal{P}_{2,-} =  \bigoplus_{i=1}^7 \mathbb{C}e_i $.  
By the computation above we get this table (up to dividing by  $\sqrt{12}$):
$$\begin{array}{c||c|c|c|c|c|c|c}  
* & e_1 & e_2 & e_3 & e_4 & e_5 & e_6 & e_7 \\   \hline \hline  
  e_1 & e_1 & e_2 & e_3 & e_4 & e_5 & e_6 & e_7 \\   \hline 
  e_2 &e_2 & 2e_1+ e_2 & e_4+ e_5 & e_3+ e_5 & e_3+ e_4 & e_6+ 2e_7 & e_6 \\  \hline 
  e_3 & e_3 & e_5+ e_6 & 2e_1+ e_3 & e_4+ 2e_7 & e_2+ e_6 & e_2+ e_5 & e_4 \\ \hline
   e_4 & e_4 & e_4+ 2e_7 & e_2+ e_5 & e_5+ e_6 & e_2+ e_6 & 2e_1+ e_3 & e_3 \\  \hline 
 e_5 & e_5 & e_3+ e_6 & e_2+ e_4 & e_3+ e_6 & 2e_1+ 2e_7 & e_2+ e_4 & e_5 \\   \hline
e_6 &  e_6 & e_3+ e_5 & e_6+ 2e_7 & 2e_1+ e_2 & e_3+ e_4 & e_4+ e_5 & e_2 \\ \hline
 e_7 & e_7 & e_4 & e_6 & e_2 & e_5 & e_3 & e_1  
 \end{array}  $$  
 
 \begin{proposition} \label{ctZZ} The subfactor $(R^{S_4} \subset R^{ \langle(1,2)\rangle})$  is not (ZZ), i.e. there are two central operators $a,b$ with $a*b$ non central.
 \end{proposition}
 \begin{proof}
 As observed by V. Kodiyalam, the following are equivalent:  
 \begin{itemize}
 \item[(1)] $a  c = c  a$, $b  c = c  b$, $\forall c \in \mathcal{P}_{2,+}$ $\Rightarrow$ $(a*b)  c=c  (a*b)$,  $\forall c$ 
\item[(2)] $x*z = z*x$, $y*z = z*y$, $\forall z \in \mathcal{P}_{2,-}$ $\Rightarrow$  $(x  y)*z=z*(x  y)$, $\forall z$  
\end{itemize} 
But $x = e_2+e_3+e_7$ and $y=e_5 + e_7$ are central for the coproduct, whereas $x  y = e_7$ is not.   
So this contradicts (2), and so (1) is not true for $\mathcal{P}_{2,+}(R^{S_4} \subseteq R^{ \langle(1,2)\rangle})$, which means that it's not (ZZ).  
 \end{proof}
 
 \begin{remark}
We can check that $(R^{S_4} \subset R^{ \langle(1,2)\rangle})$ is (Z), i.e. any minimal central projection generates a central biprojection. Contrariwise $(R^{S_4} \subset R^{\langle (1,2)(3,4) \rangle})$ is not (Z) (so not (ZZ) by Proposition \ref{ZZZ}). 
 \end{remark}
 
% \textcolor{red}{Add more?}

\subsection{No extra intermediate for the free composition} \label{noextra} \hspace*{1cm} \\
Let $N \subseteq M$ be an irreducible finite index subfactor, and let $P$ be an intermediate subfactor  ($N \subseteq P \subseteq M$).
 Let $\alpha =  {_NP_P}$ and $\beta = {_PM_M}$ be algebraic $N$-$P$ and $P$-$M$ bimodules.

\begin{definition} Let $\gamma$ be a $A$-$B$ bimodule, the sub-bimodules of $(\overline{\gamma}\gamma)^n$, $\gamma(\overline{\gamma}\gamma)^n$ or $(\overline{\gamma}\gamma)^n \overline{\gamma}$ with  $n \in \mathbb{N}$, are called the $\gamma$-colored bimodules.
\end{definition}

\begin{definition}[\cite{bj}] The subfactor $(N \subseteq M)$ is a free composition of $N \subseteq P$ and $P \subseteq M$ if the set $\Xi$ of irreducible $P$-$P$ sub-bimodules of $(\beta \overline{\beta}\overline{\alpha}\alpha)^n$, $n \in \mathbb{N}$, is the free product $\Xi_{\alpha} \star \Xi_{\beta}$, with $\Xi_{\gamma}$ the set of  irreducible $\gamma$-colored $P$-$P$ bimodules. 
\end{definition}

\begin{lemma}[\cite{bj}] \label{previous} Let $\xi = {_A\gamma_{1} \otimes_P \gamma_{2} \otimes_P   \cdots  \otimes_P {\gamma_{r}}_B}$ with $\gamma_i$ a non-trivial irreducible $\alpha$ or $\beta$-colored bimodule, with $\gamma_{2i}$ and $\gamma_{2i+1}$ differently colored, and  $A, B \in \{ N,P,M \}$. Then $\xi$ is an irreducible $A$-$B$ bimodule, uniquely determined  by the sequence $(\gamma_1,  \dots  , \gamma_r)$. 
\end{lemma}
For simplifying we just write $\xi = \gamma_{1} \gamma_{2}   \cdots  \gamma_{r}$.

\begin{theorem}
If $(N \subseteq M)$ is such a free composition (via $P$ as above) and if $L$ is another intermediate subfactor $N \subseteq L \subseteq M$, then  $N \subseteq L \subseteq P$ or  $P \subseteq L \subseteq M$.
\end{theorem}

\begin{proof}
Let $L$ be an intermediate subfactor $N \subseteq L \subseteq M$. Let $\lambda =  {_NL_L}$ and $\mu = {_LM_M}$, then $\lambda \mu = \alpha \beta = {_NM_M} := \rho $.  
Now,  $\alpha \overline{\alpha} = \bigoplus m_i \otimes \eta_i $ (with $m_i$ the multiplicity space of the irreducible $N$-$N$-bimodule $\eta_i$, and $\eta_0 = id$), $\beta \overline{\beta} = \bigoplus n_i \otimes \xi_i $, and $\rho \overline{\rho}= \alpha \beta \overline{\beta} \overline{\alpha}= \alpha \overline{\alpha} \oplus \bigoplus_{i \neq 0} n_i \otimes \alpha  \xi_i \overline{\alpha}$. By Lemma \ref{previous} $\alpha  \xi_i \overline{\alpha}$ ($i \neq 0$) is an irreducible (uniquely determined) $N$-$N$ bimodule, so that the depth $2$ vertices in the principal graph $\Gamma_{\rho}$ of $N \subseteq M$ are exactly $\alpha  \xi_i \overline{\alpha}$ and $\eta_j	$ for $i,j \neq 0$. Now we see that $\alpha \overline{\alpha} \le \rho \overline{\rho}$ and  idem $\lambda \overline{\lambda} \le \rho \overline{\rho}$.  \newline  
Case $1$:  $\lambda \overline{\lambda} \le \alpha \overline{\alpha}$ then $L \subseteq P$ because $\lambda \overline{\lambda}={_NL_N}$ and  $\alpha \overline{\alpha}={_NP_N}$.  \newline  
Case $2$: $\lambda \overline{\lambda} \not \le \alpha \overline{\alpha}$. We will prove that then $\alpha \overline{\alpha}  \le \lambda \overline{\lambda}$, so that $P \subseteq L$.  \newline  
By assumption, $\exists i_0 \neq 0$ such that $\alpha \xi_{i_0} \overline{\alpha} \le \lambda \overline{\lambda}$, then $\overline{\alpha \xi_{i_0} \overline{\alpha}} =  \alpha \overline{\xi_{i_0}} \overline{\alpha} \le \lambda \overline{\lambda}$ too, so $$(\lambda \overline{\lambda})^2 \ge \alpha \overline{\xi_{i_0}} \overline{\alpha} \alpha \xi_{i_0} \overline{\alpha} \ge \alpha \overline{\xi_{i_0}} \xi_{i_0} \overline{\alpha} \ge \alpha \overline{\alpha}$$
We now show that in $\Gamma_{\rho}$ there is no square $[\rho, \eta_i, \zeta , \alpha \xi_j \overline{\alpha}]$ with $i,j \neq 0$ and $\zeta$ a depth $3$ object.  We suppose that such a $\zeta$ exists, then, $\zeta \le \eta_i \alpha \beta$ and $\zeta \le  \alpha \xi_j \overline{\alpha} \alpha \beta$. So on one hand, $\zeta = \nu \beta$ with $\nu$ an $\alpha$-colored irreducible $N$-$P$ bimodule, and on the other hand, $\zeta = \alpha \xi_i \eta_j \beta$ ($j \neq 0$) or $\alpha \gamma$ with $\gamma$ a $\beta$-colored irreducible $P$-$M$ bimodule. But $\nu \beta \neq \alpha \xi_i \eta_j \beta$  by Lemma \ref{previous}, and also $\nu \beta \neq \alpha \gamma$ (because else $\zeta = \alpha \beta$, which is not possible because $\zeta$ is depth $3$).  The non-existence of the previous square follows.   \newline
Thanks to $\alpha \overline{\alpha} \le (\lambda \overline{\lambda})^2$ the sub-objects of $\alpha \overline{\alpha}$ appear at depth $0$, $2$ or $4$ in the principal graph $\Gamma_{\lambda}$ of $N \subseteq L$. If there exists such a sub-object $\eta_{j_0}$ at depth $4$ in $\Gamma_{\lambda}$, then $\eta_{j_0}$  and $\alpha \xi_{i_0} \overline{\alpha}$ (both depth $2$ in $\Gamma_{\rho}$) would be related via a depth $3$ object in $\Gamma_{\rho}$ (because $\eta_{j_0} \le \alpha \xi_{i_0} \overline{\alpha} \rho \overline{\rho}$), which is impossible by the non-existence of the previous square. It follows that the sub-objects of $\alpha \overline{\alpha}$ appear just at depth $0$ or $2$ in $\Gamma_{\lambda}$, i.e. $\alpha \overline{\alpha} \le \lambda \overline{\lambda}$ and $P \subseteq L$.   
\end{proof}

\section{Acknowledgments} Thanks to Vaughan Jones for being the first to believe in my cyclic subfactors theory, in July 2012, and to encourage me to develop it. Thanks to Antony Wassermann for initiating me to the von Neumann algebras and subfactors theory as PhD advisor. Thanks to Marie-Claude David, Nicolas Thiery and Leonid Vainerman for hosting, advices, exchanges and encouragement. Thanks to Dietmar Bisch and Vaughan Jones for hosting me one month at the mathematics department of Vanderbilt. Thanks to Shamindra Ghosh and Paramita Das for advising me a postdoc in IMSc. I am grateful to Dietmar Bisch, David Evans, Vaughan Jones and Scott Morrison for having written recommendation letters for me to get this postdoc at the IMSc which supports this research. I am grateful to Vijay Kodiyalam and Viakalathur Shankar Sunder for hosting me at the IMSc, for permitting me to give a lot of talks on my work and for useful advices and fruitful discussions. A special thanks to Zhengwei Liu for introducing me to several aspects of the planar algebras theory, for having translated my initial problem in the planar algebra framework, for regular fruitful exchanges and for useful comments on my work. Thanks to David Penneys for encouraging me to find applications of my result. Thanks to Emily Peters for allowing me to use her diagrams in Section \ref{planar}. Thanks to Mamta Balodi for having read back with me this paper in detail. Thanks also to Keshab Chandra Bakshi, Jyotishman Bhowmick, Marcel Bischoff, Arnaud Brothier, Shamindra Ghosh, Andre Henriquez, Derek Holt, Alexander Hulpke, Masaki Izumi, Scott Morrison, Issan Patri, David Penneys, Emily Peters, John Shareshian, Ajit Iqbal Singh, Jack Schmidt, Noah Snyder, Richard Stanley and Feng Xu, for useful advices and exchanges. Finally, thanks to mathoverflow and all its users, for hugely many exchanges.


\begin{bibdiv}
\begin{biblist}
\bib{asha}{article}{
   author={Asaeda, M.},
   author={Haagerup, U.},
   title={Exotic subfactors of finite depth with Jones indices
   $(5+\sqrt{13})/2$ and $(5+\sqrt{17})/2$},
   journal={Comm. Math. Phys.},
   volume={202},
   date={1999},
   number={1},
   pages={1--63},
   issn={0010-3616},
   review={\MR{1686551 (2000c:46120)}},
   doi={10.1007/s002200050574},
}
\bib{asch}{article}{
   author={Aschbacher, Michael},
   title={Overgroup lattices in finite groups of Lie type containing a
   parabolic},
   journal={J. Algebra},
   volume={382},
   date={2013},
   pages={71--99},
   issn={0021-8693},
   review={\MR{3034474}},
   doi={10.1016/j.jalgebra.2013.01.034},
}
\bib{bl}{article}{
   author={Bhattacharyya, Bina}, author={Landau, Zeph},
   title={Intermediate Standard Invariants and Intermediate Planar Algebras},
   note={Preprint, submitted to J. Funct. Anal.}
}
\bib{bi}{article}{
   author={Bisch, Dietmar},
   title={A note on intermediate subfactors},
   journal={Pacific J. Math.},
   volume={163},
   date={1994},
   number={2},
   pages={201--216},
   issn={0030-8730},
   review={\MR{1262294 (95c:46105)}},
   doi={10.2140/pjm.1994.163.201}
}
\bib{bj}{article}{
   author={Bisch, Dietmar},
   author={Jones, Vaughan},
   title={A note on free composition of subfactors},
   conference={
      title={Geometry and physics},
      address={Aarhus},
      date={1995},
   },
   book={
      series={Lecture Notes in Pure and Appl. Math.},
      volume={184},
      publisher={Dekker, New York},
   },
   date={1997},
   pages={339--361},
   review={\MR{1423180 (97m:46102)}},
}
\bib{bjunp}{article}{
   author={Bisch, Dietmar},
   author={Jones, Vaughan},
   title={The free product of planar algebras, and subfactors},
   note={unpublished},
}
\bib{dako}{article}{
   author={Das, Paramita},
   author={Kodiyalam, Vijay},
   title={Planar algebras and the Ocneanu-Szyma\'nski theorem},
   journal={Proc. Amer. Math. Soc.},
   volume={133},
   date={2005},
   number={9},
   pages={2751--2759 (electronic)},
   issn={0002-9939},
   review={\MR{2146224 (2006b:46084)}},
   doi={10.1090/S0002-9939-05-07789-0},
}
\bib{david}{article}{
   author={David, Marie-Claude},
   title={Paragroupe d'Adrian Ocneanu et alg\`ebre de Kac},
   language={French},
   journal={Pacific J. Math.},
   volume={172},
   date={1996},
   number={2},
   pages={331--363},
   issn={0030-8730},
   review={\MR{1386622 (98f:46051)}},
}
\bib{gr}{book}{
   author={Gr{\"a}tzer, George},
   title={Lattice theory: foundation},
   publisher={Birkh\"auser/Springer Basel AG, Basel},
   date={2011},
   pages={xxx+613},
   isbn={978-3-0348-0017-4},
   review={\MR{2768581 (2012f:06001)}},
   doi={10.1007/978-3-0348-0018-1},
}
\bib{gjs}{article}{
   author={Guionnet, A.},
   author={Jones, V. F. R.},
   author={Shlyakhtenko, D.},
   title={Random matrices, free probability, planar algebras and subfactors},
   conference={
      title={Quanta of maths},
   },
   book={
      series={Clay Math. Proc.},
      volume={11},
      publisher={Amer. Math. Soc., Providence, RI},
   },
   date={2010},
   pages={201--239},
   review={\MR{2732052 (2012g:46094)}},
}
\bib{izha}{article}{
   author={Izumi, Masaki},
   title={The structure of sectors associated with Longo-Rehren inclusions.
   II. Examples},
   journal={Rev. Math. Phys.},
   volume={13},
   date={2001},
   number={5},
   pages={603--674},
   issn={0129-055X},
   review={\MR{1832764 (2002k:46161)}},
   doi={10.1142/S0129055X01000818},
}
\bib{iz}{article}{
   author={Izumi, Masaki},
   title={Characterization of isomorphic group-subgroup subfactors},
   journal={Int. Math. Res. Not.},
   date={2002},
   number={34},
   pages={1791--1803},
   issn={1073-7928},
   review={\MR{1920326 (2003f:46100)}},
   doi={10.1155/S107379280220402X},
}
\bib{ilp}{article}{
   author={Izumi, Masaki},
   author={Longo, Roberto},
   author={Popa, Sorin},
   title={A Galois correspondence for compact groups of automorphisms of von
   Neumann algebras with a generalization to Kac algebras},
   journal={J. Funct. Anal.},
   volume={155},
   date={1998},
   number={1},
   pages={25--63},
   issn={0022-1236},
   review={\MR{1622812 (2000c:46117)}},
   doi={10.1006/jfan.1997.3228},
}
\bib{jo}{article}{
   author={Jones, Vaughan F. R.},
   title={Actions of finite groups on the hyperfinite type ${\rm II}_{1}$\
   factor},
   journal={Mem. Amer. Math. Soc.},
   volume={28},
   date={1980},
   number={237},
   pages={v+70},
   issn={0065-9266},
   review={\MR{587749 (81m:46094)}},
   doi={10.1090/memo/0237},
}
\bib{jo2}{article}{
   author={Jones, Vaughan F. R.},
   title={Index for subfactors},
   journal={Invent. Math.},
   volume={72},
   date={1983},
   number={1},
   pages={1--25},
   issn={0020-9910},
   review={\MR{696688 (84d:46097)}},
   doi={10.1007/BF01389127},
}
\bib{jo3}{article}{
   author={Jones, Vaughan F. R.},
   title={A converse to Ocneanu's theorem},
   journal={J. Operator Theory},
   volume={10},
   date={1983},
   number={1},
   pages={61--63},
   issn={0379-4024},
   review={\MR{715556 (84m:46086)}},
}
\bib{jo4}{article}{
   author={Jones, Vaughan F. R.},
   title={Planar algebras, I},
   date={1999},
   pages={122pp},
   journal={arXiv:math/9909027}, 
   note={to appear in New Zealand Journal of Mathematics},
}
\bib{joalg}{article}{
   author={Jones, Vaughan F. R.},
   title={Two subfactors and the algebraic decomposition of bimodules over
   $\rm II_1$ factors},
   journal={Acta Math. Vietnam.},
   volume={33},
   date={2008},
   number={3},
   pages={209--218},
   issn={0251-4184},
   review={\MR{2501843 (2010f:46095)}},
}
\bib{js}{book}{
   author={Jones, Vaughan F. R.},
   author={Sunder, V. S.},
   title={Introduction to subfactors},
   series={London Mathematical Society Lecture Note Series},
   volume={234},
   publisher={Cambridge University Press, Cambridge},
   date={1997},
   pages={xii+162},
   isbn={0-521-58420-5},
   review={\MR{1473221 (98h:46067)}},
   doi={10.1017/CBO9780511566219},
}
\bib{jms}{article}{
   author={Jones, Vaughan F. R.},
   author={Morrison, Scott},
   author={Snyder, Noah},
   title={The classification of subfactors of index at most 5},
   journal={Bull. Amer. Math. Soc. (N.S.)},
   volume={51},
   date={2014},
   number={2},
   pages={277--327},
   issn={0273-0979},
   review={\MR{3166042}},
   doi={10.1090/S0273-0979-2013-01442-3},
}
\bib{kls}{article}{
   author={Kodiyalam, Vijay},
   author={Landau, Zeph},
   author={Sunder, V. S.},
   title={The planar algebra associated to a Kac algebra},
   note={Functional analysis (Kolkata, 2001)},
   journal={Proc. Indian Acad. Sci. Math. Sci.},
   volume={113},
   date={2003},
   number={1},
   pages={15--51},
   issn={0253-4142},
   review={\MR{1971553 (2004d:46075)}},
   doi={10.1007/BF02829677},
}
\bib{sk}{article}{
   author={Kodiyalam, Vijay},
   author={Sunder, V. S.},
   title={The subgroup-subfactor},
   journal={Math. Scand.},
   volume={86},
   date={2000},
   number={1},
   pages={45--74},
   issn={0025-5521},
   review={\MR{1738515 (2001b:46103)}},
}
\bib{sk2}{article}{
   author={Kodiyalam, Vijay},
   author={Sunder, V. S.},
   title={On Jones' planar algebras},
   journal={J. Knot Theory Ramifications},
   volume={13},
   date={2004},
   number={2},
   pages={219--247},
   issn={0218-2165},
   review={\MR{2047470 (2005e:46119)}},
   doi={10.1142/S021821650400310X},
}
\bib{sk3}{article}{
   author={Kodiyalam, Vijay},
   author={Sunder, V. S.},
   title={From subfactor planar algebras to subfactors},
   journal={Internat. J. Math.},
   volume={20},
   date={2009},
   number={10},
   pages={1207--1231},
   issn={0129-167X},
   review={\MR{2574313 (2011d:46129)}},
   doi={10.1142/S0129167X0900573X},
}
\bib{la0}{book}{
   author={Landau, Zeph A.},
   title={Intermediate subfactors},
   note={Thesis (Ph.D.)--University of California at Berkeley},
   date={1998},
   pages={132},
}
\bib{la}{article}{
   author={Landau, Zeph A.},
   title={Exchange relation planar algebras},
   booktitle={Proceedings of the Conference on Geometric and Combinatorial
   Group Theory, Part II (Haifa, 2000)},
   journal={Geom. Dedicata},
   volume={95},
   date={2002},
   pages={183--214},
   issn={0046-5755},
   review={\MR{1950890 (2003k:46091)}},
   doi={10.1023/A:1021296230310},
}
\bib{li}{article}{
   author={Liu, Zhengwei},
   title={Exchange relation planar algebras of small rank},
   date={2013},
   pages={35pp},
   journal={arXiv:1308.5656}, 
   note={to appear in Trans. AMS.},
}

\bib{lonka}{article}{
   author={Longo, Roberto},
   title={A duality for Hopf algebras and for subfactors. I},
   journal={Comm. Math. Phys.},
   volume={159},
   date={1994},
   number={1},
   pages={133--150},
   issn={0010-3616},
   review={\MR{1257245 (95h:46097)}},
}

\bib{lupa}{article}{
   author={Luk{\'a}cs, E.},
   author={P{\'a}lfy, P. P.},
   title={Modularity of the subgroup lattice of a direct square},
   journal={Arch. Math. (Basel)},
   volume={46},
   date={1986},
   number={1},
   pages={18--19},
   issn={0003-889X},
   review={\MR{829806 (87d:20031)}},
   doi={10.1007/BF01197131},
}
\bib{nk}{article}{
   author={Nakamura, Masahiro},
   author={Takeda, Zir{\^o}},
   title={On the fundamental theorem of the Galois theory for finite
   factors. },
   journal={Proc. Japan Acad.},
   volume={36},
   date={1960},
   pages={313--318},
   issn={0021-4280},
   review={\MR{0123926 (23 \#A1247)}},
}
\bib{oc}{book}{
   author={Ocneanu, Adrian},
   title={Actions of discrete amenable groups on von Neumann algebras},
   series={Lecture Notes in Mathematics},
   volume={1138},
   publisher={Springer-Verlag, Berlin},
   date={1985},
   pages={iv+115},
   isbn={3-540-15663-1},
   review={\MR{807949 (87e:46091)}},
}
\bib{or}{article}{
   author={Ore, Oystein},
   title={Structures and group theory. II},
   journal={Duke Math. J.},
   volume={4},
   date={1938},
   number={2},
   pages={247--269},
   issn={0012-7094},
   review={\MR{1546048}},
   doi={10.1215/S0012-7094-38-00419-3},
}
\bib{palfy}{article}{
   author={P{\'a}lfy, P{\'e}ter P.},
   title={Groups and lattices},
   conference={
      title={Groups St. Andrews 2001 in Oxford. Vol. II},
   },
   book={
      series={London Math. Soc. Lecture Note Ser.},
      volume={305},
      publisher={Cambridge Univ. Press, Cambridge},
   },
   date={2003},
   pages={428--454},
   review={\MR{2051548 (2005b:20051)}},
   doi={10.1017/CBO9780511542787.014},
}
\bib{ep}{article}{
   author={Peters, Emily},
   title={A planar algebra construction of the Haagerup subfactor},
   journal={Internat. J. Math.},
   volume={21},
   date={2010},
   number={8},
   pages={987--1045},
   issn={0129-167X},
   review={\MR{2679382 (2011i:46077)}},
   doi={10.1142/S0129167X10006380},
}
\bib{po}{article}{
   author={Popa, Sorin},
   title={Classification of amenable subfactors of type II},
   journal={Acta Math.},
   volume={172},
   date={1994},
   number={2},
   pages={163--255},
   issn={0001-5962},
   review={\MR{1278111 (95f:46105)}},
   doi={10.1007/BF02392646},
}
\bib{po2}{article}{
   author={Popa, Sorin},
   title={An axiomatization of the lattice of higher relative commutants of
   a subfactor},
   journal={Invent. Math.},
   volume={120},
   date={1995},
   number={3},
   pages={427--445},
   issn={0020-9910},
   review={\MR{1334479 (96g:46051)}},
   doi={10.1007/BF01241137},
}
\bib{sc}{book}{
   author={Schmidt, Roland},
   title={Subgroup lattices of groups},
   series={de Gruyter Expositions in Mathematics},
   volume={14},
   publisher={Walter de Gruyter \& Co., Berlin},
   date={1994},
   pages={xvi+572},
   isbn={3-11-011213-2},
   review={\MR{1292462 (95m:20028)}},
   doi={10.1515/9783110868647},
}
\bib{sta}{book}{
   author={Stanley, Richard P.},
   title={Enumerative combinatorics. Volume 1},
   series={Cambridge Studies in Advanced Mathematics},
   volume={49},
   edition={2},
   publisher={Cambridge University Press, Cambridge},
   date={2012},
   pages={xiv+626},
   isbn={978-1-107-60262-5},
   review={\MR{2868112}},
}
\bib{sw}{article}{
   author={Sunder, V. S.},
   author={Wildberger, N. J.},
   title={Actions of finite hypergroups},
   journal={J. Algebraic Combin.},
   volume={18},
   date={2003},
   number={2},
   pages={135--151},
   issn={0925-9899},
   review={\MR{2002621 (2004g:20092)}},
   doi={10.1023/A:1025107014451},
}
\bib{su}{article}{
   author={Suzuki, Noboru},
   title={A linear representation of a countably infinite group},
   journal={Proc. Japan Acad},
   volume={34},
   date={1958},
   pages={575--579},
   issn={0021-4280},
   review={\MR{0100805 (20 \#7233)}},
}
\bib{szka}{article}{
   author={Szyma{\'n}ski, Wojciech},
   title={Finite index subfactors and Hopf algebra crossed products},
   journal={Proc. Amer. Math. Soc.},
   volume={120},
   date={1994},
   number={2},
   pages={519--528},
   issn={0002-9939},
   review={\MR{1186139 (94d:46061)}},
   doi={10.2307/2159890},
}
\bib{teru}{article}{
   author={Teruya, Tamotsu},
   title={normal intermediate subfactors},
   journal={J. Math. Soc. Japan},
   volume={50},
   date={1998},
   number={2},
   pages={469--490},
   issn={0025-5645},
   review={\MR{1613172 (99e:46080)}},
   doi={10.2969/jmsj/05020469},
}
\bib{xu}{article}{
   author={Xu, Feng},
   title={On a problem about tensor products of subfactors},
   journal={Adv. Math.},
   volume={246},
   date={2013},
   pages={128--143},
   issn={0001-8708},
   review={\MR{3091803}},
   doi={10.1016/j.aim.2013.06.026},
}
\bib{wa}{article}{
   author={Watatani, Yasuo},
   title={Lattices of intermediate subfactors},
   journal={J. Funct. Anal.},
   volume={140},
   date={1996},
   number={2},
   pages={312--334},
   issn={0022-1236},
   review={\MR{1409040 (98c:46134)}},
   doi={10.1006/jfan.1996.0110},
}
\end{biblist}
\end{bibdiv}
\end{document}